\documentclass{amsart}
   \usepackage{pkgfile}
\usepackage{amsfonts,amscd,amssymb,amsmath,amsthm,mathdots,upgreek}
\usepackage{stmaryrd}
\SetSymbolFont{stmry}{bold}{U}{stmry}{m}{n}
\usepackage{graphicx,caption,subcaption,mathrsfs,appendix}
\usepackage{color}
\usepackage{enumerate}
\usepackage{bm}
\usepackage[colorlinks=true,linkcolor=blue]{hyperref}

\numberwithin{equation}{section}

\usepackage{xparse}
\definecolor{airforceblue}{rgb}{0.36, 0.54, 0.66}
\definecolor{amethyst}{rgb}{0.6, 0.4, 0.8}
\definecolor{applegreen}{rgb}{0.55, 0.71, 0.0}
\def\corAB{}

\definecolor{purple}{rgb}{0.9,0,0.8}
\begin{document}

%\newtheorem{theorem}{Theorem}
%\newtheorem{proposition}[theorem]{Proposition}
%\newtheorem{conjecture}[theorem]{Conjecture}
%\newtheorem{corollary}[theorem]{Corollary}
%\newtheorem{lemma}[theorem]{Lemma}
%\theoremstyle{definition}
%\newtheorem{dfn}{Definition}
%\newtheorem{assumption}[theorem]{Assumption}
%\newtheorem{claim}[theorem]{Claim}
%\newtheorem{remark}[theorem]{Remark}
%
%\numberwithin{theorem}{section}
%\numberwithin{dfn}{section}

%%%%%%%%%%%%%%%%%%%%%%%%%%%%%%%%%%%%%%%%%%
%%% General macros
%%%%%%%%%%%%%%%%%%%%%%%%%%%%%%%%%%%%%%%%%%

%\newcommand{\B}{\mathbb{B}}
%\newcommand{\R}{\mathbb{R}}
%\newcommand{\T}{\mathcal{T}}
%\newcommand{\C}{\mathbb{C}}
%\newcommand{\D}{\mathbb{D}}
%\newcommand{\G}{\mathcal{G}}
%\newcommand{\Z}{\mathbb{Z}}
%\newcommand{\Q}{\mathbb{Q}}
%\newcommand{\E}{\mathbb E}
%\renewcommand{\P}{\mathbb P}
%\newcommand{\N}{\mathbb N}
%%\newcommand{\gd}{\mathfrak{d}}
%\newcommand{\gd}{\mathfrak{d}}
%\newcommand{\gb}{\mathfrak{b}}
%\newcommand{\gL}{\mathfrak{L}}
%\newcommand{\vep}{\varepsilon}
%\newcommand{\cS}{\mathcal{S}}
%\newcommand{\cI}{\mathcal{I}}
%\newcommand{\cY}{\mathcal{Y}}
%\newcommand{\cB}{\mathcal{B}}
%\newcommand{\cM}{\mathcal{M}}
%\newcommand{\cN}{\mathcal{N}}
%\newcommand{\cX}{\mathcal{X}}
\newcommand{\frakg}{\mathfrak{g}}

\newcommand{\barray}{\begin{eqnarray*}}
\newcommand{\earray}{\end{eqnarray*}}

\newcommand{\dvec}[1]{ \llbracket #1 \rrbracket}

\newcommand{\Def}{:=}

%%%%%%%%%%%%%%%%%%%%%%%%%%%%%%%%%%%%%%%%%%
%%% Probability Macros
%%%%%%%%%%%%%%%%%%%%%%%%%%%%%%%%%%%%%%%%%%

\DeclareDocumentCommand \Pr { o }
{%
\IfNoValueTF {#1}
{\operatorname{Pr}  }
{\operatorname{Pr}\left[ {#1} \right] }%
}
\newcommand{\Prob}{\Pr}
\newcommand{\Exp}{\mathbb{E}}
\newcommand{\expect}{\mathbb{E}}
\newcommand{\Pto}{\overset{\mathbb{P}}{\to} }
\newcommand{\weakto}{\Rightarrow}
\newcommand{\lawequals}{\overset{\mathcal{L}}{=}}
\newcommand{\prob}{\Pr}
\newcommand{\pr}{\Pr}
\newcommand{\filt}{\mathscr{F}}
\newcommand{\ohadI}{\mathbbm{1}}
\DeclareDocumentCommand \one { o }
{%
\IfNoValueTF {#1}
{\ohadI }
{\ohadI\left\{ {#1} \right\} }%
}
\newcommand{\Bernoulli}{\operatorname{Bernoulli}}
\newcommand{\Binomial}{\operatorname{Binom}}
\newcommand{\Binom}{\Binomial}
\newcommand{\Poisson}{\operatorname{Poisson}}
\newcommand{\Exponential}{\operatorname{Exp}}

%\newcommand{\Var}{\operatorname{Var}}
%\newcommand{\Cov}{\operatorname{Cov}}

%%%%%%%%%%%%%%%%%%%%%%%%%%%%%%%%%%%%%%%%%%
%%% Random Matrix Macros
%%%%%%%%%%%%%%%%%%%%%%%%%%%%%%%%%%%%%%%%%%

\newcommand{\Id}{\operatorname{Id}}
\newcommand{\Span}{\operatorname{span}}

%%%%%%%%%%%%%%%%%%%%%%%%%%%%%%%%%%%%%%%%%%
%%% Paper-Specific Macros
%%%%%%%%%%%%%%%%%%%%%%%%%%%%%%%%%%%%%%%%%%

\newcommand{\nicefrac}{\frac}
\newcommand{\half}{\frac12}
\DeclareDocumentCommand \JB { O{n} O{\lambda} } {J_{{#1}}({#2})}

\DeclareDocumentCommand \LP { O{\ESD} } {U_{ {#1} }}
\newcommand{\LPL}{ \LP[{\mu_N}] }

\newcommand{\ESD}{ L_N^{\model} }

\newcommand{\model}{\mathcal{M}}
\newcommand{\SVD}{\Sigma}
\newcommand{\LL}{\mathcal{L}}
\newcommand{\PI}{\Pi}
\DeclareDocumentCommand \PG { O{n} }
{
\mathfrak{S}_{{ #1 }}
}
\newcommand{\RS}{\mathcal{C}}

\newcommand{\TODO}[1]{ {\bf TODO: #1} }

\newcommand{\row}{X}
\newcommand{\col}{Y}
\newcommand{\srow}{x}
\newcommand{\scol}{y}
\newcommand{\csrow}{w}
\newcommand{\cscol}{z}
\newcommand{\COMP}[1]{ \check{#1} }

\date{\today}
\title[Spectral properties of random perturbation of Toeplitz] {Spectral properties of random perturbations of non-normal Toeplitz matrices}
%and twisted Toeplitz cases}
%\shortitle{Regularization of non-normal matrices}
\author[A.\ Basak]{Anirban Basak$^*$}%\thanks{${}^*$Partially supported by
 %grant 147/15 from the Israel Science Foundation, Start-up Research Grant (SRG/2019/001376) from Science and Engineering Research Board of Govt.~of India, and ICTS--Infosys Excellence Grant.}
 \address{$^*$International Center for Theoretical Sciences
\newline\indent Tata Institute of Fundamental Research
\newline\indent Bangalore 560089, India
}
% \author[E.\ Paquette]{Elliot Paquette$^\ddagger$}%\thanks{${}^\ddagger$Partially supported by NSF postdoctoral fellowship DMS-1606310}
% \address{$^\ddagger$Department of Mathematics, The Ohio State University  
% \newline\indent Tower 100, 231 W 18th Ave, Columbus, Ohio 43210, USA}
%\author[O.\ Zeitouni]{Ofer Zeitouni$^{\mathsection}$}%\thanks{${}^{\mathsection}$Partially 
%%supported by  grant 147/15 from the Israel Science Foundation}
%\address{$^{\mathsection}$Department of Mathematics, Weizmann Institute of Science 
% \newline\indent POB 26, Rehovot 76100, Israel
% \newline \indent and
% \newline\indent Courant Institute, New York University
% \newline \indent 251 Mercer St, New York, NY 10012, USA}
%\author{Anirban Basak, Elliot Paquette and Ofer Zeitouni}
%  \address{Department of Mathematics, Weizmann Institute of Science}}
%  \author{Elliot Paquette
%    \address{Faculty fo Mathematics, Ohio State University}}
%    \author{Ofer Zeitouni
%\address{Department of Mathematics, Weizmann Institute of Science}
%\address{Faculty of Mathematics, Ohio State University}
  %\address{Department of Mathematics, Weizmann Institute of Science}

%\thanks{This work was partially supported by a grant from the Israel
%Science Foundation}

\begin{abstract}
Spectral properties of Toeplitz operators and their finite truncations have long been central in operator theory. In the finite dimensional, non-normal setting, the spectrum is notoriously unstable under perturbations. Random perturbations provide a natural framework for studying this instability and identifying spectral features that emerge in typical noisy situations. This article surveys recent advances on the spectral behavior of (polynomially vanishing) random perturbations of Toeplitz matrices, focusing mostly on the limiting spectral distribution, the distribution of outliers, and localization of eigenvectors, and highlight the major techniques used to study these problems. We complement the survey with new results on the limiting spectral distribution of polynomially vanishing random perturbation of Toeplitz matrices with continuous symbols, on the limiting spectral distribution of finitely banded Toeplitz matrices under exponentially and super-exponentially vanishing random perturbations, and on the complete localization of outlier eigenvectors of randomly perturbed Jordan blocks.
\end{abstract}

%Ofer
\maketitle

%\tableofcontents

\section{Introduction}
\subsection*{Instability and insufficiency of the spectrum} 
Early usages of eigenvalues and eigenfunctions date back to the eighteenth and nineteenth centuries, primarily in the context of understanding physical problems involving vibrations and heat dissiptaions. Since then eigenvalues and eigenfunctions have been widely used to successfully study complex systems arising in varied areas, including, but not limited to, quantum mechanics, neuroscience and medical imaging, structural engineering, data science and machine learning, and even in the PageRank algorithm.  A closer inspection reveals that such successes are largely obtained when operators under consideration are normal or are nearly normal. 

Let us make a quick detour and recall the definition of a normal operator: A bounded linear operator $T$ on a complex Hilbert space $\mathscr{H}$ is said to be normal if it commutes with its Hermitian adjoint $T^*$. That is, $T T^*=T^* T$. Otherwise, it is said to be non-normal.  

Non-normal operators arise naturally in many real-world systems that are non-conservative, dissipative, or open, i.e.~where energy is not preserved. For example, open quantum systems, shear flows in fluid dynamics, linearized models of weather systems and climate dynamics, and neural network models with asymmetric connectivity.  Compared to normal operators, the eigenvalues of non-normal operators pose additional challenges: (i) eigenvalue analysis in many applications turns out to be insufficient and therefore can be misleading, and (ii) eigenvalues are sensitive to perturbations \corAB{and often} yielding unreliable results. Let us illustrate these two features through the following two simple examples.

\begin{example}\label{ex:fAfB}\footnote{Example taken from \cite{tref91}.} Define 
\begin{equation}\label{eq:def-AB}
A:=\begin{pmatrix} -1 & 1 \\ 0 & -1 \end{pmatrix} \quad \text{ and } \quad B:=\begin{pmatrix} -1 & 5 \\ 0 & -2 \end{pmatrix}. 
\end{equation}
\corAB{Let} $f_A(t):= \| \exp(tA)\|_2$ and $f_B(t):=\|\exp(tB)\|_2$ for $t \geqslant 0$, where $\|\cdot \|_2$ denotes the operator norm induced by the standard Euclidean norm. In the left panel of Figure \ref{fig12} the red and the blue curves represent $f_A$ and $f_B$, respectively, as a function of $t$. Since both eigenvalues of both $A$ and $B$ are {\em strictly negative} it can deduced that $f_A(t), f_B(t) \to 0$ as $t \to \infty$. In fact, as the maximal eigenvalues  of $A$ and $B$ are identical it further follows that $\f1t \log(f_A(t)/f_B(t)) \to 0$ as $t \to \infty$ (cf.~\cite[Theorem 15.3]{TE05}). %For a large $t$, the slopes of these two curves can be determined via an eigenvalue analysis of the corresponding matrices. 
However, the `hump'-like structure in the blue curve cannot be explained solely by the eigenvalues of $B$. Such hump-like structures are quite ubiquitous in dynamical systems and are known as the `transient behavior' of the system. 
\end{example} 

\begin{figure}[htbp]
  \centering
   \begin{minipage}[b]{0.4\linewidth}
 % \centering
   \includegraphics[width=\textwidth]{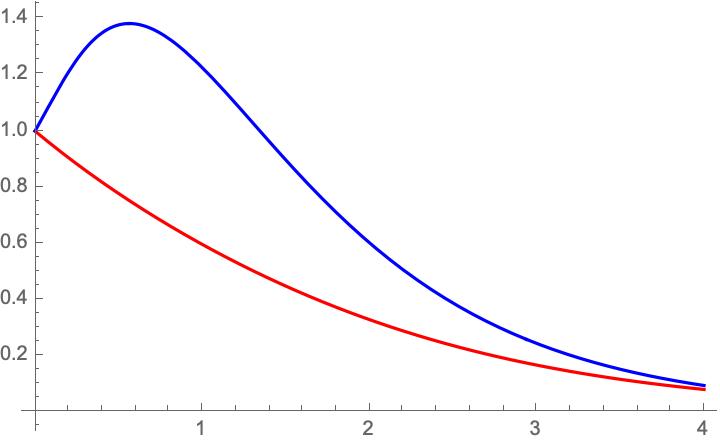}
 
%   \caption{}
%   \label{fig1}
 \end{minipage}
 \hspace{1cm}
 \begin{minipage}[b]{0.3\linewidth}
  \includegraphics[width=\textwidth]{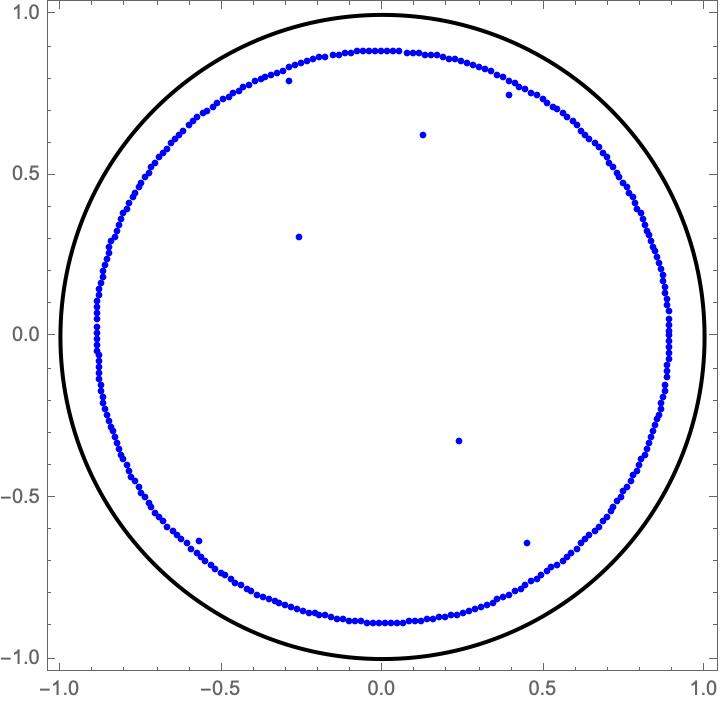}
%  \caption{}
%   \label{fig2}
 \end{minipage}
 \caption{\texttt{Left panel:} The red and the blue curves represent $f_A(t)$ and $f_B(t)$ (defined in Example \ref{ex:fAfB}), respectively, as a function of $t$.  \texttt{Right panel:} Eigenvalues of $\widehat{J}_N= U_N J_N U_N^*$, $N=1000$, with $U_N$ a {\em simulated} (in Mathematica) Haar unitary matrix, are plotted in blue. The black curve is the unit circle.}
 \label{fig12}
\end{figure}

\begin{example}\label{ex:J} Let $J_N$ be the standard Jordan block of dimension $N$ with zero on the \corAB{main} diagonal.  That is, the entries on the first super-diagonal position of $J_N$ are one, and all other 
entries are zero. Simulate a \corAB{Haar} unitary matrix $U_N$ from the unitary group of dimension $N$ and set $\widehat{J}_N:= U_N J_N U_N^*$. Clearly, the eigenvalues of $J_N$ and $\widehat{J}_N$ are same and are all 
equal to zero. However, when Mathematica is asked to compute the eigenvalues of $\widehat{J}_N$, it
fails miserably to locate the correct eigenvalues. See the right panel of Figure \ref{fig12}. The 
eigenvalues of $\widehat J_N$, for $N=1000$, are plotted in blue and the unit circle $\mathbb{S}^1$ on the complex plane is drawn in black. 
Note that most of the eigenvalues computed through Mathematica are found near $\mathbb{S}^1$. This anomaly is due to the rounding or the machine error. 
\end{example}

\subsection*{Pseudospectrum} 
The phenomena noted in Examples \ref{ex:fAfB} and \ref{ex:J} cannot be explained by analyzing the eigenvalues of relevant matrices. Notice that both these examples involve non-normal matrices. Instead of the spectrum, the correct object to look at for such matrices is the `pseudospectrum'. The earliest definition of the pseudospectrum can be traced back Varah's Ph.D.~thesis in $1967$ \cite{V67}. Seminal contributions by Lloyd N.~Trefethen and collaborators in the nineties popularized the concept of the pseudospectrum. We refer the reader to \cite[Chapter 6]{TE05} for a brief historical account of the advancement of this area. 

From a mathematical perspective a key difference between a normal and a non-normal operator is that the former admits a spectral theorem while the latter does not. In fact, from the spectral theorem it follows that for any bounded normal operator $T$, 
\beq\label{eq:resolvent-normal}
\|(z-T)^{-1}\| = \f{1}{{\rm dist}(z, \sigma(T))}, \qquad \text{ for all } z \notin \sigma(T),
\eeq 
(see Section \ref{sec:notation} for the notational conventions adopted in this article) where 
\[
\rho(T):= \{\lambda \in \C: (\lambda -T)^{-1} \text{ is a densely defined bounded linear operator on } \mathscr{H}\}
\] 
is the resolvent set of $T$, and $\sigma(T):= \C \setminus \rho(T)$ is the spectrum of $T$. In contrast, for a non-normal operator the \abbr{LHS} of \eqref{eq:resolvent-normal} can only be lower bounded by its \abbr{RHS}. Moreover, for a $z$ of distance order one from $\sigma(T)$, for a non-normal operator $T$, one can have both $\|(z-T)^{-1}\| \asymp 1$ and $\|(z-T)^{-1}\| \gg 1$, commonly known in the literature as `zone of spectral stability' and `zone of spectral instability', respectively. The notion of pseudospectrum is essentially an attempt to characterize and quantify these zones.

%A finite dimensional matrix $A$ is said to be normal or non-normal depending on whether  $AA^*=A^*A$ or $AA^* \ne A^*A$, where $A^*$ is  the complex conjugate transpose of $A$. 

\begin{dfn}\label{dfn:pseudospectrum}
Let $M \in {\rm Mat}_N(\C)$, the set of all $N \times N$ matrices with possibly complex valued entries. Let $\| \cdot \|$ be any norm on $\C^N$ and for any $B \in {\rm Mat}_N(\C)$, by a slight abuse of notation, we write $\|B\|$ to denote the operator norm of $B$ induced by the chosen norm on $\C^N$. Fix any $\vep >0$. 
The $\vep$-pseudospectrum of $M$ is defined as
\[
\sigma_\vep(M):=\left\{ z \in \C: \| (z -M)^{-1}\| > \vep^{-1}\right\}. 
\]
\end{dfn}
In this article, we only discuss the notion of pseudospectrum for finite dimensional matrices. It has a natural extension for bounded linear operators on Banach spaces, see \cite[Chapter 4]{TE05}. 

The following result provides two alternate characterizations of $\sigma_\vep(M)$. 

\begin{proposition}[{\cite[Theorem 2.1]{TE05}}]\label{prop:te05-1}
Let $M$, $\vep$, and the norm $\| \cdot \|$ be as in Definition \ref{dfn:pseudospectrum}. Then
\begin{align*}
\sigma_\vep(M) & = \left\{z \in \C: \exists \, E \in {\rm Mat}_N(\C) \text{ with } \|E\| < \vep \text{ such that } z \in \sigma(M+E) \right\}\\
& = \left\{ z \in \C: \exists \, v \in \C^N \text{ with } \|v\|=1 \text{ such that } \|(z - M) v \| < \vep\right\}.
\end{align*}
\end{proposition}

It is trivial to note that $\sigma_\vep(M) \supset \sigma(M)$. Proposition \ref{prop:te05-1} additionally shows that $\sigma_\vep(M)$ consists of regions in the complex plane that can be reached by the eigenvalues of $M$ perturbed by some matrix of operator norm at most $\vep$. The following result captures a key difference between the pseudospectra of normal and non-normal matrices. 

\begin{proposition}[{\cite[Theorems 2.2 and 2.3]{TE05}}]\label{prop:te05-2}
Let $M$, $\vep$, and the norm $\| \cdot \|$ be as in Definition \ref{dfn:pseudospectrum}. The following properties of $\sigma_\vep(M)$ hold.
\begin{enumerate}
\item[(i)] $\sigma_\vep(M) \supset \sigma(M) + \D(0,\vep)$.
\item[(ii)] If $M$ is normal and $\|\cdot\|=\|\cdot\|_2$, the standard Euclidean norm then
\beq\label{eq:pseudo-normal}
\sigma_\vep(M) = \sigma(M) +\D(0,\vep), \quad \forall \vep >0. 
\eeq
Conversely, if $\|\cdot \|=\|\cdot\|_2$ and \eqref{eq:pseudo-normal} holds then $M$ must be a normal matrix. 
\item[(iii)] Let $M$ be diagonalizable. That is, there exists a diagonal matrix $D$ and a non-singular matrix $V$ such that $M=V D V^{-1}$.  For matrix $A$, let $\kappa(A)$ be its condition number. That is, $\kappa(A)= \|A\| \cdot \|A^{-1}\|$, \corAB{for $A$} invertible and $\kappa(A)=\infty$ otherwise. If $\| \cdot\|=\|\cdot\|_2$ then
\[
\sigma_\vep(M) \subset \sigma(M)+ \D(0, \vep \kappa(V)), \qquad \forall \vep >0. 
\] 
\end{enumerate}
\end{proposition}

The key takeaway from Proposition \ref{prop:te05-2} is that for a normal matrix its $\vep$-pseudospectrum, when the chosen norm is the standard Euclidean norm, is simply the $\vep$-fattening of its spectrum. In contrast, the $\vep$-pseudospectrum, even for a small $\vep$, of a non-normal matrix can be quite far from its spectrum. 

\begin{figure}[htbp]
  \centering
   \begin{minipage}[b]{0.3\linewidth}
 % \centering
   \includegraphics[width=\textwidth]{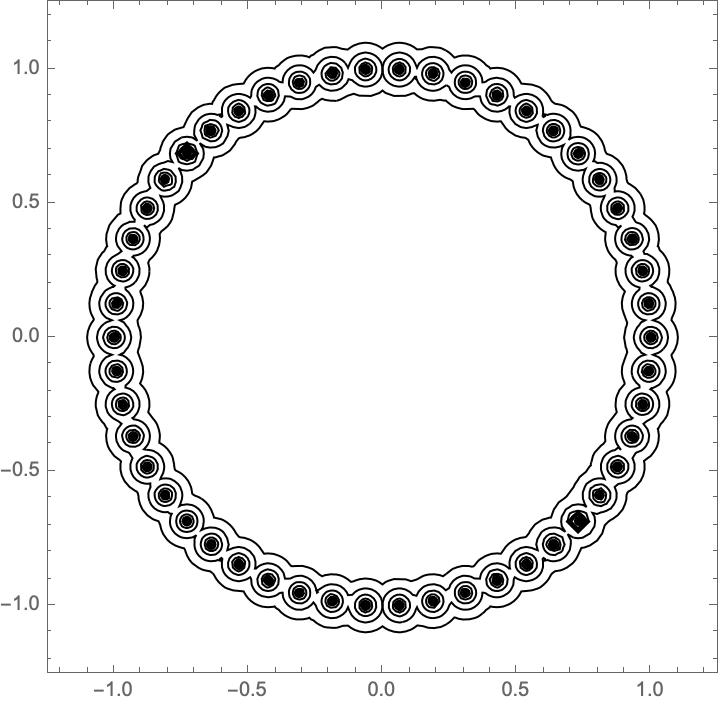}
 
%   \caption{}
%   \label{fig3}
 \end{minipage}
 \hspace{3cm}
 \begin{minipage}[b]{0.3\linewidth}
  \includegraphics[width=\textwidth]{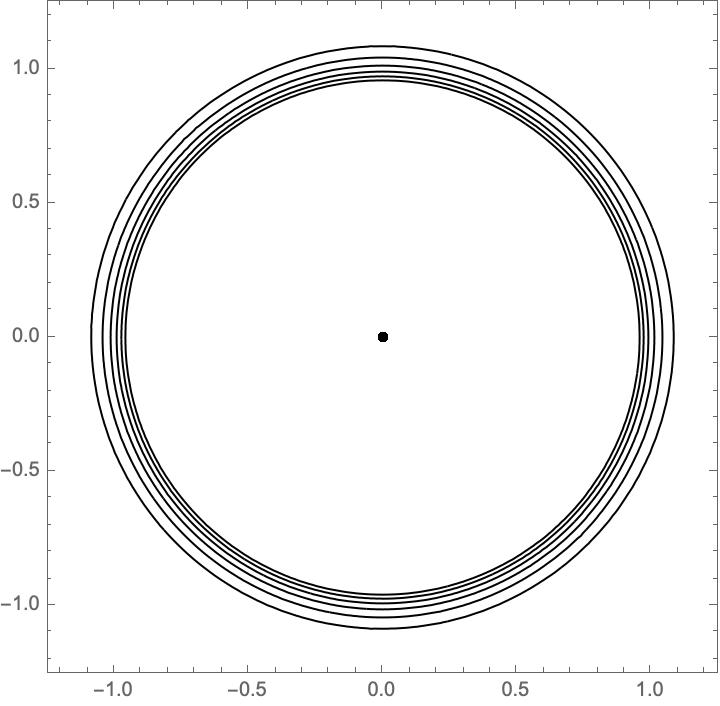}
%  \caption{}
%   \label{fig4}
 \end{minipage}
 \caption{$\vep$-pseudospectral level lines for $\vep=10^{-1}, 10^{-1.2}, \ldots, 10^{-2}$ for $C_N$ (\texttt{left panel}) and $J_N$ (\texttt{right panel}). $N=50$. All eigenvalues of $J_N$ are zero, while eigenvalues of $C_N$ (defined in Example \ref{ex:J1}), \corAB{a rank one perturbation of $J_N$ that is a circulant matrix and hence normal}, are equispaced on the unit circle. Both sets of eigenvalues are marked with black dots. Compare the stark difference of the pseudospectral level lines in left and right panels.}
 \label{fig34}
\end{figure}

\begin{example}\label{ex:J1}
Let $F_N(\vep)$ be the $N \times N$ matrix whose bottom left entry is $\vep = \vep_N \in (0,1]$, and the rest are zero. An easy computation shows that the eigenvalues of $\wt J_N(\vep):=J_N+F_N(\vep)$ are equispaced on the boundary of $\D(0,\vep^{1/N})$. This implies that even if $\vep=\vep_N \ll 1$ such that $\vep_N \asymp N^{-\gamma}$, for some $\gamma >0$, then the eigenvalues of $\wt J_N(\vep_N)$ approach the unit circle, as $N \to \infty$. On the other hand, if $\vep_N \asymp r^N$, for some $r \in (0,1)$, then the eigenvalues of $\wt J_N(\vep_N)$ approach the boundary of $\D(0,r)$. This shows that even an exponentially vanishing perturbation (in dimension) of $J_N$ has a {\em macroscopic} effect on its spectrum, for all large $N$. In particular, in terms of the pseudospectrum we have $\sigma_{r^{N}}(J_N) \supset {\D(0,r)}$ for any $r \in (0,1)$. This is evident in Figure \ref{fig34}. On the left panel we have $\vep$-pseudospectral level lines (the norm chosen is the Euclidean norm) of $J_N$ for $\vep=10^{-1}, 10^{-1.2}, \ldots, 10^{-2}$ and $N=50$. On the right we have the same for $C_N:=\wt J_N(1)$, a circulant matrix (and hence normal). Observe that all its pseudospectral level lines are tightly knit around its eigenvalues -- an illustration of Proposition \ref{prop:te05-2}(ii).

Let us now recall the definitions of $\wh J_N$ and $U_N$ from Example \ref{ex:J}. The machine error encountered during the simulation of $U_N$ and in the subsequent computation of the eigenvalues of $\wh J_N$ essentially implies that the eigenvalues that we see in the right panel of Figure \ref{fig12} are actually eigenvalues of some perturbation of $J_N$. These perturbations due to machine errors are typically polynomially small in the dimension. Therefore, as demonstrated in the paragraph above, we end up finding these eigenvalues near the unit circle, even for a moderately large value of $N$. 
\end{example}

\begin{figure}[htbp]
  \centering
   \begin{minipage}[b]{0.4\linewidth}
 % \centering
   \includegraphics[width=\textwidth]{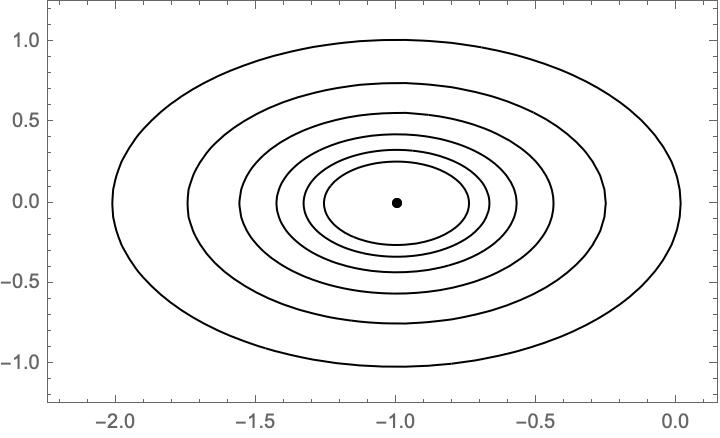}
 
%   \caption{}
%   \label{fig5}
 \end{minipage}
 \hspace{1cm}
 \begin{minipage}[b]{0.4\linewidth}
  \includegraphics[width=\textwidth]{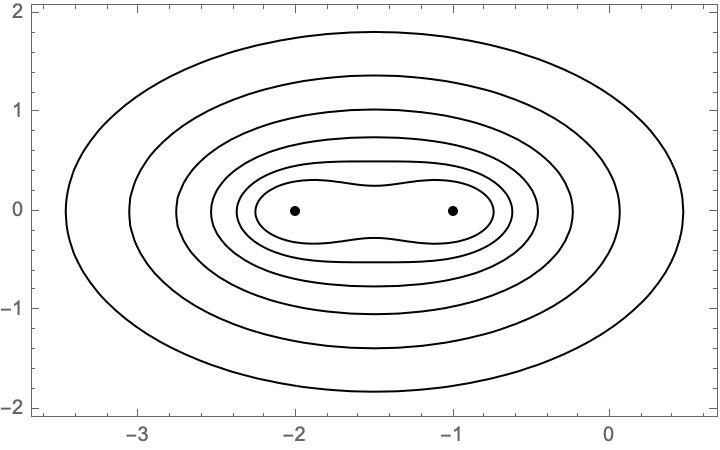}
%  \caption{}
%   \label{fig6}
 \end{minipage}
 \caption{$\vep$-pseudospectral level lines for $A$ (\texttt{left panel}) and $B$ (\texttt{right panel}), see \eqref{eq:def-AB} for their definitions, with $\vep=10^{-0.2}, 10^{-0.4}, \ldots, 10^{-1.2}$.}
 \label{fig:56}
\end{figure}

\begin{repexample}{ex:fAfB}[revisited]
The fundamental difference between the pseudospectral level lines of $A$ and $B$, as shown in Figure \ref{fig:56}, is that those of $A$ do not protrude into the right half of the plane, while those of $B$ do. This is indeed responsible for producing the hump-like structure of $f_B(t)$ for an intermediate range of $t$. See \cite[Chapter 15]{TE05} for further explanation and a proof. 

Example \ref{ex:fAfB} is a toy instance of the widely prevalent phenomena of transient behavior of dynamical systems due to the pseudospectral `bulge' of underlying operators, appearing in varied domains such as fluid dynamics, electrical power systems, epidemiology, aeroelasticity, chemical kinetics, and control systems.  
One such concrete example is the onset of turbulence  in the plane Couette flow at a high Reynolds number: The spectrum of the Navier-Stokes evolution operator linearized about the laminar flow is always contained in the left half of the plane. However, for a sufficiently large Reynolds number and a small $\vep>0$ its  $\vep$- pseudospectrum protrudes a distance `much' greater than $\vep$ into the right half-plane, and as a result certain perturbations of the plane Couette flow grow transiently at that high Reynolds number eventually decaying due to viscosity.  
\end{repexample}

\subsection*{Random perturbations of non-normal operators}

One notes from Proposition \ref{prop:te05-1} that pseudospectra describe the worst case behavior under a perturbation of a given size. To determine the pseudospectral level lines one essentially needs to choose a fine meshed grid and compute the norm of the resolvent at each of the grid points, making it computationally expensive and not scalable for large matrices. 
On the other hand, random perturbations reflect the typical case behavior. Most physical or biological systems are not adversarially perturbed -- they are simply noisy. Random perturbations model those more realistically. For example, see Figures \ref{fig7} and \ref{fig8}. 
Further, in many applications eigenvalue clouds resulting from random perturbations tend to align along pseudospectral level lines revealing pseudospectral boundary statistically (e.g., compare Figures \ref{fig8} and \ref{fig9}). 
Hence, even without determining the pseudospectral level lines, which involves heavy computations, one can identify salient characteristics of the underlying operator, such as pseudospectral bulges. These advantages of typical perturbation over a worst case perturbation analysis has led to a significant surge in the last twenty years  in studying spectral features of random perturbations of non-normal \corAB{matrices and operators}. 

\begin{figure}[htbp]
  \centering
   \begin{minipage}[b]{0.3\linewidth}
 % \centering
   \includegraphics[width=\textwidth]{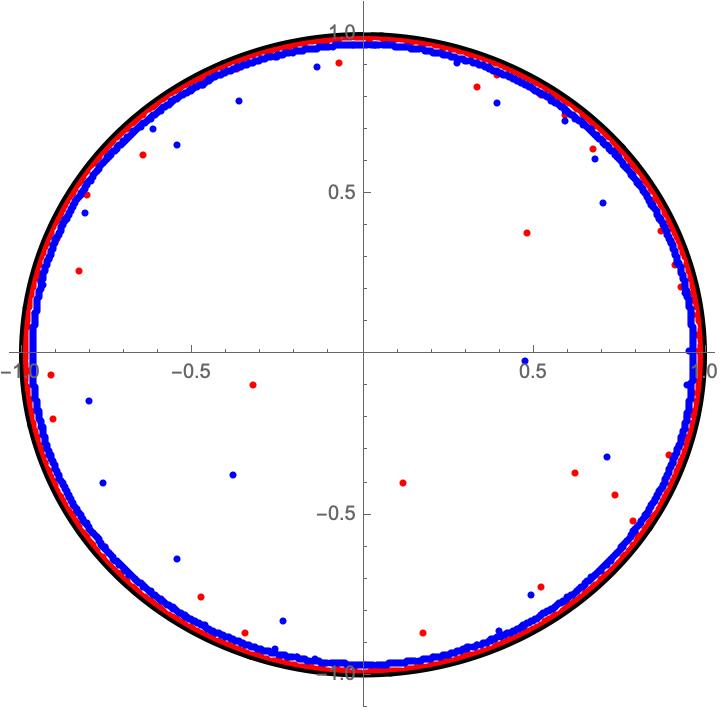}
 
%   \caption{}
%   \label{fig3}
 \end{minipage}
 \hspace{2cm}
 \begin{minipage}[b]{0.35\linewidth}
  \includegraphics[width=\textwidth]{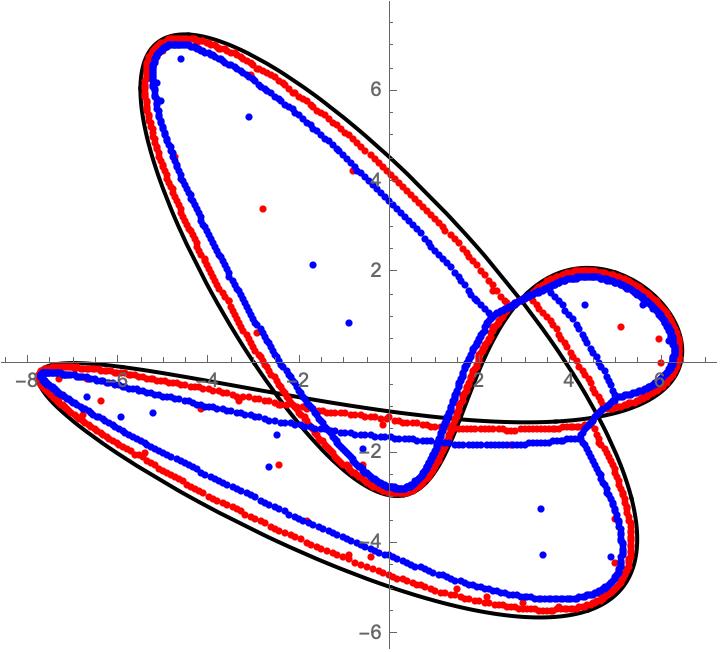}
%  \caption{}
%   \label{fig4}
 \end{minipage}
 \caption{Eigenvalues of $U_N A_N U_N^*$ are in blue, where $U_N$ is a simulated Haar unitary matrix via Mathematica, while eigenvalues of $A_N+N^{-2} G_N$ are in red, where $G_N$ is a matrix with i.i.d.~standard complex Gaussians.  $N=1000$. $A_N=T_N({\bm a})$ for symbols \corAB{${\bm a}(\zeta)=\zeta^{-1}$} (\texttt{left panel}) and ${\bm a}(\zeta)=2 \zeta^{-3}-\zeta^{-2} +2 {\rm i} \zeta^{-1}-4\zeta - 2 {\rm i} \zeta^2$ (\texttt{right panel}). See Definition \ref{dfn:toep-finite} for the definition of a Toeplitz matrix with a given symbol. Symbol curves ${\bm a}(\mathbb{S}^1)$ are in black.}
 \label{fig7}
\end{figure}

While the focus of this article is on the spectral properties of random perturbations of large non-normal matrices, there is a substantial and growing body of work regarding the same for non-normal operators. We refer the reader to the recent survey article \cite{V24} for a brief overview of results in that direction. 

Returning to the limiting spectral properties of random perturbation of large finite dimensional matrices perhaps one of the earliest works was done by \'Sniady in \cite{S02}, where the problem was addressed from a different perspective. To describe \'Sniday's perspective we need a couple of definitions. The first one is a notion of convergence borrowed from the free probability theory. 

\begin{dfn}[Convergence in $*$-moments]
Let $(\mathscr{A}, \phi)$ be a non commutative $W^*$ probability space. That is, $\mathscr{A}$ is a von Neumann algebra (cf.~\cite[Definition 5.2.21]{AGZ}) and $\phi$ is a normal faithful tracial state (cf.~\cite[Definitions 5.2.9 and 5.2.26]{AGZ}). Let ${\bm x} \in \mathscr{A}$ and $\{A_N\}_{N \in \N}$ be a sequence of matrices with $A_N \in {\rm Mat}_N(\C)$. The sequence of matrices $\{A_N\}_{N \in \N}$ is said to converge to ${\bm x}$ in $*$-moments if 
$\lim_{N \to \infty} N^{-1}\tr P(A_N, A_N^*) = \phi(P({\bm x}, {\bm x}^*))$, for every bivariate non-commutative polynomial $P$. 
\end{dfn}

The next notion generalizes the notion of spectral measure for non-normal operators.

\begin{dfn}[Brown measure]
Let $(\mathscr{A}, \phi)$ be as above. For any bounded self-adjoint element ${\bm h} \in \mathscr{A}$, we let ${\bm \nu}_{{\bm h}}$ to be its spectral measure, i.e.~the unique probability measure compactly supported on reals given by $\phi({\bm h}^n)=\int t^n d{\bm \nu}_{{\bm h}}(t)$ for all $n \in \N$. For any bounded operator ${\bm x} \in \mathscr{A}$ we define its Brown measure ${\bm \mu}_{{\bm x}}$  to be the probability measure on $\C$ given by 
${\bm \mu}_{{\bm x}}= \f{1}{2\pi}\Delta_z \int \log (t) \, {\bm \nu}_{|{\bm x}-z|} dt$, where the equality is in the sense of distribution, $|{\bm h_0}|:=\sqrt{{\bm h}_0^*{\bm h}_0}$ for any operator ${\bm h}_0$, and $\Delta_z := \f{\partial^2}{\partial ({\rm Re}\, z)^2}+ \f{\partial^2}{\partial ({\rm Im}\, z)^2}$, for $z \in \C$, is  the two dimensional Laplacian.
\end{dfn}

It can be checked that for $\mathscr{A}={\rm Mat}_N(\C)$ and $\phi(\cdot)=\f1N \tr(\cdot)$, the Brown measure of any $A_N \in {\rm Mat}_N(\C)$ is $L_{A_N}$, the empirical measure of its eigenvalues. Furthermore, if $\{A_N\}_{N \in \N}$ is a sequence of normal matrices, then \corAB{the} convergence to a bounded normal operator ${\bm x}$ in $*$-moments is sufficient for their respective Brown measures to converge. This fails as soon as $\{A_N\}_{N \in \N}$ are non-normal. For example, \corAB{it can be checked that} $\{J_N\}_{N \in \N}$ converges to ${\bm u}$, the Haar unitary operator (i.e.~$\phi({\bm u}^n)=\phi(({\bm u}^*)^n)=0$ for all $n \in \N$), in $*$-moments, while Brown measures of $J_N$ and ${\bm u}$ are Dirac measure at zero and the uniform measure on the unit circle on $\C$, respectively. \'Sniday asked whether one could randomly perturb $\{A_N\}_{N \in \N}$ so that the empirical measure of the eigenvalues of randomly perturbed $A_N$ converge to the Brown measure of the limiting operator ${\bm x}$ (in the sense of $*$-moment convergence). It was answered affirmatively in \cite{S02}, and such a phenomena is termed in the literature as `random regularization of Brown measure'. 

The proof in \cite{S02} uses a soft analysis (stochastic calculus and a diagonal subsequence extraction argument) and it does not shed any light on the strength of the random perturbation needed to achieve this regularization. The work \cite{GWZ} shows that if one is able to find a deterministic perturbation, say $\wt E_N$, so that $L_{\wt A_N}$, where $\wt A_N:=A_N+ \wt E_N$, converges to ${\bm \mu}_{{\bm x}}$ then any polynomially vanishing Gaussian perturbation of $A_N$ will have the same limiting spectral distribution (\abbr{LSD}). Therefore, in light of the discussion in Example \ref{ex:J1}, the \abbr{LSD} of any polynomially vanishing Gaussian perturbation of $J_N$ is the law of ${\bm u}$. In a follow up work \cite{W} the assumption on the distribution of the entries of the perturbation was relaxed. 
On the other hand, relying on a completely different approach \cite{DH} provides quantitative bounds on the number of eigenvalues for Gaussian random perturbations of Jordan blocks in a given region in the complex plane. 
Let us also mention in passing that there are examples of sequences of deterministic matrices such that no polynomially vanishing Gaussian perturbation regularizes the Brown measure, see \cite{FPZ}. 

Following \cite{DH, FPZ, GWZ, W} a substantial body of research has emerged that explores the spectral properties of polynomially vanishing random perturbations of Toeplitz and related matrices. This article aims to provide a comprehensive account of the current state of the art, elucidating the central ideas that underpin this growing literature, while also augmenting the existing body of results with some extensions, including results on \abbr{LSD} under exponentially small random perturbation and complete localization of \corAB{eigenvectors corresponding to} outlier eigenvalues. Section \ref{sec:results} introduces necessary terminologies, followed by results on the limiting spectral distribution, limiting outlier distribution, and localization of eigenvectors, while the latter sections are devoted for 
highlighting the key approaches and strategies in addressing these problems and for the proofs of the new results.

\subsection{Notational conventions}\label{sec:notation} 
We use the following standard asymptotic notation: For two sequences of nonnegative reals $\{a_n\}_{n \in \N}$ and $\{b_n\}_{n \in \N}$ we write $a_n \ll b_n$, $b_n \gg a_n$, and $a_n = o(b_n)$ to denote $\limsup_{n \to \infty} a_n/b_n=0$, while to denote $\limsup_{n \to \infty} a_n /b_n  < \infty$, the notation $a_n \lesssim b_n$, $b_n \gtrsim a_n$, and $a_n=O(b_n)$ are used. We further write $a_n \asymp b_n$ if $a_n \lesssim b_n$ and $b_n \lesssim a_n$. For $a,b \in \R$ we write $a \vee b:= \max\{a,b\}$, $a \wedge b:= \min\{a,b\}$, and $(a)_+:= a \vee 0$. For $m <n \in \N$ we let $[m,n]:=\{m+1,m+2, \ldots, n\}$ and further use the shorthand $[n]$ for the discrete interval $[1,n]$. 

For a matrix $\cM_n$ of dimension $n \times n$ we write $0 \le s_1(\cM_n) \le s_2(\cM_n) \leq \cdots \leq s_n(\cM_n)$ to denote the singular values of $\cM_n$ arranged in an nondecreasing order. It ill be convenient to write $s_{\min}(\cdot)$ for the smallest singular value of a matrix. Unless otherwise mentioned, the notation $\|\cM_n\|$ and $\|\cM_n\|_{{\rm HS}}$ will be used to denote the spectral and Hilbert-Schmidt norms of $\cM_n$ with respect to the standard Euclidean norm on $\C^N$, respectively. For $S_1, S_2 \subset [n]$ we let $\cM_n[S_1;S_2]$ to be the sub matrix of $\cM_n$ which consists of the rows and columns in $S_1$ and $S_2$, respectively. The notation $\cM_n^t$ is used to denote the transpose of $\cM_n$. Further, we write $L_{\cM_n}$ to denote the empirical measure of the eigenvalues of $\cM_n$. That is, $L_{\cM_n}:=\f1n\sum_{i=1}^n \lambda_i(\cM_n)$, where $\{\lambda_i(\cM_n)\}_{i=1}^n$ are eigenvalues of $\cM_n$. We use ${\rm Id}_n$ to denote the identity matrix of dimension $n$. For a vector $v \in \C^n$ and $I \subset [n]$ we write $v_I$ to denote the vector of length $|I|$ obtained from $v$ by deleting the indices corresponding to $[n]\setminus I$. Unless otherwise mentioned, for any vector $v$ we further let $\|v\|$ and $\|v\|_\infty$ be its Euclidean and $\ell^\infty$ norms, respectively. 

For two sets $S_1$ and $S_2$ we write $S_1\pm S_2$ to denote their Minkowski sum and difference. That is, $S_1 \pm S_2 := \{s_1 \pm s_2; s_1 \in S_1, s_2 \in S_2\}$. For any $w \in \C$ and $A \subset \C$ we let ${\rm dist} (w, A):= \inf_{u \in A} |w-u|$. For $a \in \C$ and $r >0$ we write $\D(a,r)$ to denote the open disk in the complex plane of radius $r$ centered at $a$.

\section{Random perturbations of non-normal Toeplitz matrices}\label{sec:results}
%In this section we present results on the spectral properties of random perturbations of non-normal Toeplitz matrices

We begin this section with definitions of infinite dimensional Toeplitz and Laurent operators and some useful spectral properties of these operators. %before moving on to describe results on random perturbations of non-normal Toeplitz matrices (and some related matrices).

\begin{dfn}
For any map ${\bm a}: \mathbb{S}^1 \mapsto \C$ we write $\|{\bm a}\|_{\infty, \mathbb{S}^1}$ to denote the standard essential supremum norm with respect to the uniform measure on $\mathbb{S}^1$. We let $L^\infty(\mathbb{S}^1):=\{{\bm a} : \|{\bm a}\|_{\infty,\mathbb{S}^1} < \infty\}$. For any ${\bm a} \in L^\infty(\mathbb{S}^1)$  we define $L({\bm a}): \ell^2(\Z) \mapsto \ell^2(\Z)$, the Laurent operator with symbol ${\bm a}$, given by $\langle \mathfrak{e}_i, L({\bm a})\mathfrak{e}_j \rangle = a_{i-j}$, where $\{\mathfrak{e}_i\}_{i \in \Z}$ are canonical basis vectors in $\ell^2(\Z)$, $\langle \cdot, \cdot \rangle$ is the Euclidean inner product on $\ell^2(\Z)$ over \corAB{the} complex field,  and $\{a_n\}_{n \in \Z}$ are Fourier coefficients of ${\bm a}$. 
That is, 
\[
a_n:= \f{1}{2 \pi}\int_0^{2\pi} {\bm a}(e^{{\rm i} \theta}) e^{-{\rm i} n \theta} d \theta. 
\]
For ${\bm a} \in L^\infty(\mathbb{S}^1)$ we define $T({\bm a}): \ell^2(\N) \mapsto \ell^2(\N)$, the Toeplitz operator with symbol ${\bm a}$ given by $\langle \mathfrak{e}_i, T({\bm a})\mathfrak{e}_j \rangle = a_{i-j}$, where by a slight abuse of notation we now let $\{\mathfrak{e}_i\}_{i \in \N}$ to be the canonical basis vectors in $\ell^2(\N)$ and $\langle \cdot, \cdot \rangle$ to be the Euclidean inner product on $\ell^2(\N)$ over \corAB{the} complex field.
\end{dfn}

It is well known that for ${\bm a} \in L^\infty(\mathbb{S}^1)$ the operators $L({\bm a})$ and $T({\bm a})$ are bounded linear operators on $\ell^2(\Z)$ and $\ell^2(\N)$, respectively (cf.~\cite[Theorems 1.1 and 1.9]{BS99}). The following is a classical result regarding the spectra of infinite dimensional Laurent and Toeplitz operators. 

\begin{proposition}[{\cite[Theorems 1.2 and 1.17]{BS99}}]\label{prop:BS99}
Let ${\bm a} \in C(\mathbb{S}^1)$. For any $w \in \C\setminus {\bm a}(\mathbb{S}^1)$ we write ${\rm wind}({\bm a}, w)$ to denote the winding number of the closed curve ${\bm a}(\mathbb{S}^1)$ around $w$. Then $\sigma(L({\bm a}))= {\bm a}(\mathbb{S}^1)$ and
\[
\sigma(T({\bm a}))= {\bm a }(\mathbb{S}^1) \cup \left\{ w\in \C\setminus \corAB{{\bm a}(\mathbb{S}^1)}: {\rm wind}({\bm a}, w) \neq 0 \right\}.
\]
\end{proposition}

Next we define finite dimensional Toeplitz matrices. 
\begin{dfn}\label{dfn:toep-finite}
For $N \in \N$ let $\Pi_N: \ell^2(\N) \mapsto \C^N$ be the canonical projection map. For ${\bm a} \in L^\infty(\mathbb{S}^1)$ we set $T_N({\bm a})$ to  be the restriction of $\Pi_N T({\bm a}) \Pi_N$ to the image of $\Pi_N$, the $N \times N$ dimensional Toeplitz matrix with symbol ${\bm a}$. Therefore,  
\[
T_N({\bm a}):=\begin{bmatrix}
a_{0} & a_{-1} & a_{-2} & \cdots & \cdots & a_{-(N-1)}\\
a_{1}& a_{0} & a_{-1} &\ddots &  & \vdots\\
a_{2} &a_{1}& \ddots & \ddots & \ddots & \vdots \\
\vdots & \ddots & \ddots & \ddots & a_{-1} & a_{-2}\\
\vdots & & \ddots & a_{1} & a_0 & a_{-1}\\
a_{N-1} & \cdots & \cdots & a_{2} &a_{1} & a_{0}
\end{bmatrix}
.
\]
\end{dfn}
For most of this article we will consider the case when the symbol ${\bm a}$ is a Laurent polynomial. That is,
\beq\label{eq:laurent-poly}
{\bm a}(\lambda) = \sum_{n =-N_-}^{N_+} a_k \lambda^k, \quad \lambda \in \C,
\eeq
some integers $N_\pm$ such that $-N_- \le N_+$. Without loss of generality we further assume $|N_+|+|N_-| > 0$ (otherwise the matrix is a constant multiple of identity \corAB{and hence normal}) and $\min\{|a_{N_-}|, |a_{N_+}|\} >0$. Notice that for such a symbol ${\bm a}(\cdot)$ the corresponding Toeplitz matrix $T_N({\bm a})$ is finitely banded. That is, there exists a finite band around the main diagonal such that outside that band the entries of $T_N({\bm a})$ are all zero. 

%Recall that, for any bounded linear operator $H$ on a Banach space its spectrum $\sigma(H)$ is the collection of all $\lambda \in \C$ such that 

%It is natural is ask whether $\sigma(T_N({\bm a}))$ approximates $\sigma(T({\bm a}))$ well for large $N$ and whether  

Next we move on to describe the limit of $L_{T_N({\bm a})}$ for ${\bm a}$ a Laurent polynomial. We need a few standard definitions and some notation. Let $\{\nu_N\}_{N \in \N}$ and $\nu$ be probability measures on $\C$. We say that $\nu_N$ converges weakly to $\nu$, to be denoted by $\nu_N \Lra \nu$, if $\int f(x) d\nu_N(x) \to \int f(x) d\nu(x)$, as $N \to \infty$, for every bounded continuous function $f: \C \mapsto \R$. If $\nu_N$'s are random, e.g.~the empirical measure of eigenvalues of random matrices, and $\int f(x) d\nu_N(x) \to \int f(x) d\nu(x)$ in probability, as $N \to \infty$, for any $f$ as above, then $\nu_N$ is said to converge weakly, in probability, to $\nu$, as $N \to \infty$. 

We also need the following useful notion of $\log$-potential of a probability measure. 

\begin{dfn}[Log-potential]
Let $\mathscr{P}(\C)$ be the set of probability measure on $\C$ that integrates the $\log|\cdot|$ in a neighborhood of infinity. For $\mu \in \mathscr{P}(C)$ we define its log-potential $\cL_\mu: \C \mapsto [-\infty, \infty)$ as follows:
\[
\mathcal{L}_\mu(z):= \int \log |z - x| d\mu(x), \qquad z \in \C.
\]
\end{dfn}
The usefulness of $\log$-potentials of probability measures is captured in the lemma below. 

\begin{lemma}\label{lem:cL-converges}
(i) Let $\{\mu_N\}_{N \in \N}$ be a sequence of probability measures in $\mathscr{P}(\C)$ such that 
\beq\label{eq:cL-converges}
\cL_{\mu_N}(z) \to \cL(z), \quad \text{ as } N \to \infty,
\eeq
for Lebesgue a.e.~every $z \in \C$ and some $\cL: \C \mapsto [-\infty, \infty)$, and 
\[
\sup_N \int |\lambda|^2 d\mu_N(\lambda) < \infty. 
\]
Then, there exists some $\mu \in \mathscr{P}(\C)$ such that $\cL_\mu(z) = \cL(z)$ for Lebesgue a.e.~$z \in \C$, $\mu = \f{1}{2\pi} \Delta \cL$ in the sense of distribution, where $\Delta$ denotes the two dimensional Laplacian, and $\mu_N \Lra \mu$, as $N \to \infty$. 

(ii) Let $\{\mu_N\}_{N \in \N}$ be a sequence of random probability measures in $\mathscr{P}(\C)$ such that \eqref{eq:cL-converges} holds in probability for Lebesgue a.e.~every $z \in \C$ and 
\beq\label{eq:bdd-in-prob}
\lim_{K \to \infty} \liminf_{N \to \infty} \P\left( \int |\lambda|^2 d\mu_N(\lambda) \leq K \right)=1.
\eeq
Then, $\mu_N \Lra \mu$, in probability, as $N \to \infty$, where $\mu$ is as given in part (i). 
\end{lemma}
Lemma \ref{lem:cL-converges}(i) is a direct consequence of \cite[Remark 4.8]{BC12}, {while Lemma \ref{lem:cL-converges}(ii) follows from an adaptation of the proofs of \cite[Remark 4.8]{BC12}, \cite[Lemma 3.1]{TVK10}, and Vitali's convergence theorem.} We leave the details to the reader. 

Lemma \ref{lem:cL-converges}(ii) will be used for $\mu_N=L_{A_N}$, where $A_N$ is an $N \times N$ dimensional random matrix. In all the cases to be considered in this article we will have $\limsup_{N \to \infty} N^{-1} \E[\|A_N\|_{{\rm HS}}^2] < \infty$. This bound together with Weyl's inequality (e.g.~\cite[Eqn.~(1.4)]{BC12}) imply that one has the probability bound in \eqref{eq:bdd-in-prob}. Therefore, to deduce that $L_{A_N}$ converges weakly, in probability to some probability measure it \corAB{will suffice} to show that $\cL_{L_{A_n}}(z) \to \cL(z)$, in probability, as $N \to \infty$, for some $\cL$ as in Lemma \ref{lem:cL-converges}.

Next let ${\bm a}$ be as in \eqref{eq:laurent-poly}. While dealing with the limiting spectral distribution of finitely banded Toeplitz matrices and their random perturbations, without loss of generality we may and will assume $N_\pm \ge 0$\footnote{If $N_- <0$ we modify the symbol by setting $a_{-N_--1}=\cdots=a_0=0$ and work with this modified symbol. If $N_+ < 0$ we work with $T_N({\bm a})^t$ so that the roles of $N_+$ and $N_-$ are now reversed, and modify the symbol associated to \corAB{$T_N({\bm a})^t$, if needed,} similarly in the other case.}.
For any $z \in \C$ set ${\bm a}_z(\cdot):= {\bm a}(\cdot) -z$. Further let
\beq\label{eq:zeta-dfn}
|\zeta_1(z)| \geq |\zeta_2(z)| \ge \cdots \ge |\zeta_{N_++N_-}(z)|
\eeq
be the roots of the polynomial $\zeta^{N_-} {\bm a}_z(\zeta)$ arranged in nonincreasing order of their magnitude. For ease in writing setting $\zeta_{N_++N_-+1}(z)=0$ for all $z \in \C$ define
\[
\Lambda({\bm a}):= \{z \in \C: |\zeta_{N_+}(z)| = |\zeta_{N_++1}(z)|\}. 
\]
We need one more notation: for any Laurent polynomial ${\bm b}$ and $\varrho \in (0,\infty)$ we let ${\bm b}^{\varrho}(\lambda):= {\bm b}(\varrho \lambda)$, for all $\lambda \in \C$. 
We are now ready to state the result on the limiting spectral distribution of $T_N({\bm a})$, for ${\bm a}$ a Laurent polynomial. 

\begin{proposition}[{\cite[\corAB{Lemmas 11.10 and 11.14, and Theorem 11.9}]{BG05}}]\label{prop:lsd-toep-non-perturbed}
Let ${\bm a}$ be as in \eqref{eq:laurent-poly}. Then, $L_{T_N({\bm a})} \Lra \wt\mu_{{\bm a}}$, as $N \to \infty$, where $\wt\mu_{{\bm a}} = \f{1}{2\pi} \Delta \cL_{{\bm a}}$ and for any $z \in \C\setminus \Lambda({\bm a})$,
\[
\cL_{{\bm a}}(z):= \f{1}{2\pi}\int_0^{2\pi} \log |{\bm a}^{\varrho(z)}(e^{{\rm i} \theta}) - z| d\theta, \qquad \text{ and } \qquad \varrho(z) \in \corAB{(|\zeta_{N_++1}(z)|, |\zeta_{N_+}(z)|)} =:\Gamma(z)
\]
(existence of such a $\varrho(z)$ is guaranteed by the definition of $\Lambda({\bm a})$ and the definition $\cL_{{\bm a}}(z)$ does not depend on any specific choice of $\varrho(z)$ as long as $\varrho(z) \in \Gamma(z)$). Further $\Lambda({\bm a})$ is a union of finite number of pairwise disjoint open analytic arcs and a finite number of points, and thus it is of zero (two-dimensional) Lebesgue measure. Moreover, $\wt{\mu}_{{\bm a}}$ is supported on $\Lambda({\bm a})$. 
\end{proposition}

If $T_N({\bm a})$ is triangular then $\Lambda({\bm a})=\{a_0\}$ is singleton and $\wt{\mu}_{\bm a}$ is the Dirac measure at $a_0$. In the non-triangular setting there is no straightforward description of $\wt\mu_{{\bm a}}$. See \cite[Theorem 11.16]{BG05} for a description of $\wt \mu_{{\bm a}}$ in the finitely banded setting, and \cite[p.~160]{BS99} for the Laurent polynomial ${\bm a}(\lambda) = \lambda + \vep^2 \lambda^{-1}$, $\vep \in (0,1)$. For \abbr{LSD}'s of Toeplitz matrices beyond the finitely banded setting we refer the reader to \cite[Chapter 5]{BS99}. 

Unlike the normal case, where Szeg\H{o}’s first limit theorem (see \cite[Theorems 5.8 and 5.9]{BS99}) ties the \abbr{LSD} of finite dimensional Toeplitz matrices to their associated symbols (see \cite[Corollary 5.12]{BS99}), for non-normal Toeplitz matrices the \abbr{LSD} bears little to no relation to the spectrum of the infinite-dimensional Toeplitz operator.  We will see below that suitable random perturbations bridge this disconnect and yield eigenvalue distributions that align more closely with the spectrum of the underlying infinite dimensional operator.

\subsection{Limiting spectral distribution under random perturbation}
To study the limiting spectral distribution of Toeplitz matrices under additive random perturbations, as needed, we will impose several conditions on the noise matrix. Unless otherwise stated, it will be implicitly assumed that the entries of the noise matrix $E_N$ are jointly independent with zero mean. We work with various additional assumptions on the entries of $E_N$ which we state below. 
\begin{assumption}[Bound on expected Hilbert-Schmidt norm]\label{ass:hs}
$\E \|E_N\|_{{\rm HS}}^2 \lesssim N^2$. 
\end{assumption}

\begin{assumption}[Bound on expected spectral norm]\label{ass:spectral-norm}
$\E \|E_N\| \lesssim N^{1/2}$. 
\end{assumption}

\begin{assumption}[Bound on the smallest singular value]\label{ass:s-min}
For any $\kappa \in (0,\infty)$, there exists a $\beta=\beta(\kappa) \in (0,\infty)$ and $\vep_N(\kappa) \ll 1$ so that for any fixed deterministic matrix $\cM_N$ with $\|\cM_N\|  \lesssim N^\kappa$, we have
\[
\P\left(s_{\min}(E_N+ \cM_N) \le N^{-\beta} \right) \leq \vep_N(\kappa).  
\] 
\end{assumption}

\begin{assumption}[Anti concentration bound]\label{ass:anti-conc} 
For any complex valued random variable $X$ and $\vep >0$ we define its L\'evy concentration function as follows:
\beq\label{eq:gL}
\corAB{\gL(X,\vep):= \sup_{w \in C} \P(|X-w| \le \vep).}
\eeq
There exists $\upeta >0$ and some absolute constant $C< \infty$ such that 
\[
\limsup_{N \to \infty} \max_{i,j=1}^N \gL(E_{N}(i,j), \vep) \le C \vep^\upeta,
\]
for all $\vep >0$ sufficiently small, where $\{E_{N}(i,j)\}_{i,j=1}^N$ are the entries of $E_N$. 

\end{assumption}

\begin{assumption}[Bounds on moments]\label{ass:mom} $\limsup_{N \to \infty}\max_{i,j \in [N]} \E[|E_{N}(i,j)|^h] < \infty$ for all $h \in \N$. 

\end{assumption}

Assumption \ref{ass:hs} holds as soon as the entries of $E_N$ have a uniformly bounded second moment, while by \cite{L04}, Assumption \ref{ass:spectral-norm} holds under a finite fourth moment assumption. By \cite[Theorem 2.1]{TV08}, Assumption \ref{ass:s-min} holds whenever the entries of $E_N$ are i.i.d. (complex or real) with common $N$-independent distribution having a finite variance. On the other hand, it is straightforward to note that an existence of uniformly bounded densities of the entries of $E_N$ yield Assumption \ref{ass:anti-conc}.

Our goal is to understand spectral properties of $T_N({\bm a})+\updelta_N E_N$. If the coupling constant $\updelta_N \gg N^{-1/2}$, as
\beq\label{eq:T-op-norm}
\|T_N({\bm a})\| \le \|T({\bm a})\| = \|{\bm a}\|_{\infty, \mathbb{S}^1}, \qquad \text{ for any } {\bm a} \in L^\infty(\mathbb{S}^1)
\eeq
(see \cite[Eqn.~(1.14)]{BS99}), using Weyl's inequality for the singular values (e.g.~\cite[Eqn.~(1.6)]{BC12}), bounds on the smallest singular value (\cite[Lemma 4.12]{BC12}) and the intermediate singular values (\cite[Lemma 4.11]{BC12}), followed by an application of the standard replacement principle \cite[Theorem 2.1]{TVK10}, one finds that for any symbol ${\bm a} \in L^\infty(\mathbb{S}^1)$, the \abbr{LSD} of $T_N({\bm a})+\updelta_N E_N$, after a proper normalization, is the celebrated circular law. 

The rest of the ranges for $\updelta_N$ can be broadly split into two sub cases: (i) $\updelta_N \ll N^{-1/2}$ -- microscopic perturbation (under Assumption \ref{ass:spectral-norm} the perturbation $\updelta_N E_N$ is of vanishing spectral norm), and (ii) $\updelta_N \asymp N^{-1/2}$ -- macroscopic perturbation.

First we consider the case of microscopic perturbation. Below is a result on the \abbr{LSD} of random perturbations of Toeplitz matrices, which is a culmination of several works.

\begin{theorem}[Polynomially vanishing perturbation]\label{thm:sv-general}Let ${\bm a} \in L^\infty(\mathbb{S}^1)$ be such that $\sum_{n =-\infty}^\infty |n a_n| < \infty$. Assume that $E_N$ satisfies Assumptions \ref{ass:hs} and \ref{ass:s-min}. Let $\updelta_N \ll 1$ be such that $\log (\updelta_N^{-1} N^{-1/2}) \gg \log \log N$ and $\updelta_N \ge N^{-L}$. Then $L_{T_N({\bm a})+ \updelta_N E_N} \Lra {\bm \mu}_{{\bm a}({\bm u})}$, in probability, as $N \to \infty$. 
\end{theorem}

Recall that ${\bm \mu}_{{\bm a}({{\bm u}})}$ is the Brown measure of the operator ${\bm a}({\bm u})$, where ${\bm u}$ is the Haar unitary operator. In other words, ${\bm \mu}_{{\bm a}({\bm u})}$ is push forward of the uniform measure on $\mathbb{S}^1$ by the symbol ${\bm a}$.

In \cite{BPZ} Theorem \ref{thm:sv-general} was proved for triangular finitely banded Toeplitz matrices under the additional assumption that the entries of $E_N$ are complex Gaussians. Using a different approach the work \cite{BPZ2} extended it for the non-triangular finitely banded case. On the other hand, using yet another different approach, \cite{SV} treated the finitely banded case, which was later extended in \cite{SV1} for ${\bm a}$ as in Theorem \ref{thm:sv-general}. 
It should be mentioned that \cite{SV1} works with the assumption that the entries of $E_N$ are complex Gaussian, requires a more stringent upper bound $\updelta_N$, while allowing \corAB{for $\updelta_N \ge \exp(-N^\delta)$} for any $\delta \in (0,1)$. A relaxation on the distributional assumption on the entries of $E_N$ as well as the upper bound on $\updelta_N$ is achieved by using a refinement of the replacement principle \cite[Theorem 1.8]{BPZ2}. While \cite[Theorem 1.8]{BPZ2} requires $\updelta_N \le N^{-\gamma}$ for some $\gamma >1/2$, a close inspection of its proof reveals that the same proof works as long as $\log(N^{-1/2} \updelta_N^{-1}) \gg \log \log N$.  

One expects Theorem \ref{thm:sv-general} to hold as soon as $\updelta_N \ll N^{-1/2}$. The sub optimality arises from using the trivial bound $|{\rm Im} \, G(z)| \le 1/({\rm Im}\, z)$, for $z \in \C_+$, where $G(\cdot)$ is the resolvent of the Hermitized version of $T_N({\bm a})+\updelta_N E_N$ in the proof of \cite[Theorem 1.8]{BPZ2}. If one is able to prove that the resolvent remains bounded even if ${\rm Im} \,z \asymp N^{-\delta}$, for some $\delta \in (0,1)$, one may be able remove this sub optimal upper bound on $\updelta_N$.

Let us further add that the requirement that $\updelta_N \ge N^{-L}$ stems from the use of Assumption \ref{ass:s-min}. By imposing the additional assumption that the entries of $E_N$ are real valued and have $\log$-concave densities, and using \cite{T20} one can relax the lower bound on $\updelta_N$ allowing for \corAB{$\updelta_N \ge \exp(-N^{\delta})$} for any $\delta \in (0,1)$.

We complement Theorem \ref{thm:sv-general} with the following extension that allows a large class of continuous symbol.  

\begin{theorem}[Sub exponential perturbation and continuous symbol]\label{thm:lsd-cont}
Let ${\bm a} \in C(\mathbb{S}^1)$ be such that ${\bm a}(\mathbb{S}^1)$  has a zero two dimensional Lebesgue measure. Let $\updelta_N \ll N^{-1/2}$.
\begin{enumerate}
\item[(i)] Assume that $\{E_N\}_{N \in \N}$ satisfy Assumptions \ref{ass:hs}, \ref{ass:spectral-norm}, and \ref{ass:s-min}. Then, for any $\updelta_N$ such that $\updelta_N \ge N^{-L}$ for some $L >1/2$, $L_{T_N({\bm a})+\updelta_N E_N}\Lra {\bm \mu}_{{\bm a}({\bm u})}$, in probability, as $N \to \infty$. 
\item[(ii)] Let the entries of $E_N$ be real valued, have unit variance, and have $\log$-concave densities with respect to the Lebesgue measure. Then, for any $\updelta_N$ such that $\log (1/\updelta_N) \ll N$, $L_{T_N({\bm a})+\updelta_N E_N}\Lra {\bm \mu}_{{\bm a}({\bm u})}$, in probability, as $N \to \infty$.
\end{enumerate}
\end{theorem}

We expect Theorem \ref{thm:lsd-cont} to hold only under Assumption \ref{ass:hs} and some anti-concentration bound on the entries of $E_N$. In Section \ref{sec:proof-lsd} we present a simple proof of Theorem \ref{thm:lsd-cont} that bypasses the machineries employed in \cite{BPZ2, SV, SV1}. However, it should be emphasized that those machineries have their own advantages. For example, \cite{SV, SV1} provides quantitative bound on the difference between the proportion of eigenvalues in any given arc of the symbol curve ${\bm a}(\mathbb{S}^1)$ and their limiting value. On the other hand, the ideas emerged from \cite{BPZ2} have been useful to study outlier distribution and to identify zones of `forbidden' regions around the symbol curve. Further, it will be used \corAB{in this article} to study the \abbr{LSD} under an exponentially vanishing random perturbation. The strategy employed in the proof of Theorem \ref{thm:lsd-cont} does not seem to be suitable in this latter setting. 

Next we move to state the result on the \abbr{LSD} under an exponentially vanishing random perturbation. This (and later for a description of the limit laws \corAB{of} outliers) requires to split the complex plane into several regions. Let ${\bm a}$ be as in \eqref{eq:laurent-poly}. For any $d \in \Z$ we let 
\beq\label{eq:cS}
\cS_d:=\left\{z \in \C \setminus {\bm a}(\mathbb{S}^1): {\rm wind}({\bm a}, z) =d\right\}.
\eeq
For any $z \notin {\bm a}(\mathbb{S}^1)$, letting $m_+=m_+(z)$ and $m_-=m_-(z)$ to be the number of roots of the polynomial $\zeta \mapsto \zeta^{N_-}{\bm a}_z(\zeta)$ that are greater then and less than one in modulus, respectively, we observe that ${\rm wind}({\bm a}, z)= N_+ - m_+(z)= m_-(z) - N_-$. 
Now for $d > 0$, $r \in (0,1)$, and $\ell =0,1,2,\ldots, m_-$, we further let 
\beq\label{eq:cS-exp-1}
\cS_{d, \ell}^r:= \left\{ z \in \cS_d\setminus {\bm a}(r \mathbb{S}^1): |\zeta_{m_++\ell}(z)| > r > |\zeta_{m_++\ell+1}(z)| \right\},
\eeq
where for ease in writing we set $\zeta_0(z)=\infty$. For $d < 0$, $r \in (0,1)$, and $\ell=0,1,2,\ldots, m_+$, define 
\beq\label{eq:cS-exp-2}
\cS_{d, \ell}^r:= \left\{ z \in \cS_d\setminus {\bm a}(r^{-1} \mathbb{S}^1): |\zeta_{\ell}(z)| > r^{-1} > |\zeta_{\ell+1}(z)| \right\}.
\eeq
Set
\beq\label{eq:cL-a-r}
\cL_{{\bm a}, r}(z):= \left\{
\begin{array}{ll}
\log |a_{N_+}| + \sum_{j=1}^{m_++ \ell \wedge d} \log |\zeta_j(z)|  + (d-\ell)_+ \log r, & \mbox{if } d > 0 \mbox{ and } z \in \cS_{d,\ell}^r,\\
\log |a_{N_+}| + \sum_{j=1}^{\ell \vee N_+} \log |\zeta_j(z)|  + (\ell-N_+)_+ \log r, & \mbox{if } d < 0 \mbox{ and } z \in \cS_{d,\ell}^r,\\
\log |a_{N_+}| + \sum_{j=1}^{m_+} \log |\zeta_j(z)|, & \mbox{if } z \in \cS_{0}.
\end{array}
\right.
\eeq
As 
\beq\label{eq:bar-cS}
\bar \cS:=\C\setminus ({\bm a}(\mathbb{S}^1) \cup {\bm a}(r\mathbb{S}^1) \cup {\bm a}(r^{-1}\mathbb{S}^1))= \cS_0 \cup_{d >0} \left(\cup_{\ell=0}^{m_-} \cS_{d,\ell}^r \right) \cup_{d <0} \left(\cup_{\ell=0}^{m_+} \cS_{d,\ell}^r \right),
\eeq
\eqref{eq:cL-a-r} defines $\cL_{{\bm a}, r}$ outside a set of zero Lebesgue measure. The next result shows that this is the $\log$-potential of the \abbr{LSD} of a finitely banded Toeplitz matrix under an exponentially vanishing random perturbation. 

Let us note in passing that setting $r =1$ in \eqref{eq:cS-exp-1}-\eqref{eq:cL-a-r} we get that $\cS_{d,0}^1=\cS_d$ for $d >0$, $\cS_{d,m_+}^1=\cS_d$ for $d <0$, $\cS_{d,\ell}^1 =\emptyset$ for other choices of $\ell$, and
\[
\cL_{{\bm a}, 1}(z)= \log |a_{N_+}| + \sum_{j=1}^{m_+} \log |\zeta_j(z)|, \qquad \text{ for } z \notin {\bm a}(\mathbb{S}^1).
\]
The reader may note that $\cL_{{\bm a}, 1}$ equals the $\log$-potential \corAB{of ${\bm \mu}_{{\bm a}({\bm u})}$, the limit measure for} the sub-exponential random perturbation of a finitely banded  Toeplitz matrix. It should be further pointed out that in the triangular setting, say $N_-=0$, $\cL_{{\bm a}, r}$ is the $\log$-potential of ${\bm \mu}_{{\bm a}(r {\bm u})}$. For the non-triangular setting there does not seem to be a direct relation between these two. 

\begin{theorem}[Exponentially vanishing random perturbation]\label{thm:lsd-exp}
Let ${\bm a}$ be as in \eqref{eq:laurent-poly} and $E_N$ satisfies Assumption \ref{ass:anti-conc}. Further assume that the entries of $E_N$ are centered and have uniformly bounded second moment. Fix $r \in (0,1)$. Then, $L_{T_N({\bm a})+r^N E_N} \Lra \mu_{{\bm a}, r}$, in probability, as $N \to \infty$, where $\mu_{{\bm a}, r} = \f{1}{2\pi} \Delta \cL_{{\bm a}, r}$. 
\end{theorem}

Theorems \ref{thm:lsd-cont} and \ref{thm:lsd-exp} leave the case of super exponentially vanishing random perturbation. The following result shows that a super exponentially vanishing random perturbation of a finitely banded Toeplitz matrix has no effect on its \abbr{LSD}. Its proof is straightforward. 

\begin{theorem}[Super exponentially vanishing random perturbation]\label{thm:toep-super-exp}
Let $\delta_N$ be such that $N^{-1}\log(1/\delta_N) \gg 1$. Assume that $E_N$ satisfies Assumption \ref{ass:hs} and ${\bm a}$ is a Laurent polynomial. Then $L_{T_n({\bm a})+\updelta_N E_N} \Lra \wt \mu_{{\bm a}}$, in probability, as $N \to \infty$, where $\wt\mu_{{\bm a}}$ is as in Proposition \ref{prop:lsd-toep-non-perturbed}. 
\end{theorem}

Theorem \ref{thm:toep-super-exp} is stated assuming that $E_N$ satisfies Assumption \ref{ass:hs}. The reader can check from its proof that the same proof goes through as long as $\log(\E\|E_N\|_{{\rm HS}}^2) \lesssim N$. To keep consistency with the previous results we have chosen to state Theorem \ref{thm:toep-super-exp} in its current form. 

\begin{remark}
Let us give some heuristics behind the appearance of different limits in Theorems \ref{thm:lsd-cont}, \ref{thm:lsd-exp}, and \ref{thm:toep-super-exp} through a toy problem. \corAB{The case of the Jordan block is easy to analyze. See Example \ref{ex:J1}. So we consider the case of a tridiagonal Toeplitz matrix. Hence, after rescaling the matrix its symbol can be written as ${\bm a}(\zeta)=\zeta + a_{-1} \zeta^{-1}$}. For any $z \in \C$ let $\zeta_1(z)$ and $\zeta_2(z)$ be the roots of the polynomial $\zeta \mapsto \zeta {\bm a}_z(\zeta)$, with $|\zeta_1(z)| \ge |\zeta_2(z)|$.  Let $\wt E_N \in {\rm Mat}_N(\C)$ matrix with its $(1,N)$-th and $(N,1)$-th entries set to $\omega_N$, for some $\omega _N \ll 1$, and all other entries zero. It follows straight from  the definition of determinant and Widom's formula for the determinant of a finitely banded Toeplitz matrix (see also Lemma \ref{lem:widom}) that
\beq\label{eq:tridiagonal}
|\det(T_N({\bm a}_z)+\wt E_N) | \asymp |\zeta_1(z)|^N + |\zeta_1(z)|^N |\zeta_2(z)|^N \omega_N + \omega_N, \qquad \text{ for all } z \notin \Gamma, 
\eeq
where $\Gamma \subset \C$ is some set of zero Lebesgue measure. Recall \eqref{eq:bar-cS}. \corAB{Let} $\wt \cS:= \bar \cS \setminus (\cS_0 \cup \cS_{1,0}^r \cup \cS_{-1,2}^r)$. It immediately follows from \eqref{eq:tridiagonal},
\beq\label{eq:toy}
\lim_{N \to \infty} \f1N \log |\det(T_N({\bm a}_z)+\wt E_N) | = 
\left\{\begin{array}{ll}
\log |\zeta_1(z)|, & \mbox{ if } z \in \cS_0,\\
\log |\zeta_1(z)| + \log |\zeta_2(z)|, & \mbox{ if } \log(1/\omega_N) \ll N \mbox{ and } z \in \cS_{-1},\\
0, & \mbox{ if } \log(1/\omega_N) \ll N \mbox{ and } z \in \cS_{1},\\
\log |\zeta_1(z)| + \log |\zeta_2(z)| +\log r, & \mbox{ if } \omega_N =r^N \mbox{ and } z \in \cS_{-1,2}^r,\\
\log r,& \mbox{ if } \omega_N=r^N \mbox{ and } z \in \cS_{1,0}^r,\\
\log |\zeta_1(z)|, & \mbox{ if } \omega_N=r^N \mbox{ and } z \in \wt \cS,\\
\log |\zeta_1(z)|, & \mbox{ if } \log(1/\omega_N) \gg N \mbox{ and } z \in \notin \Gamma.
\end{array}
\right.
\eeq
The reader may note that the expressions in the \abbr{RHS} of \eqref{eq:toy} are indeed the $\log$-potentials of the limits appearing in Theorems \ref{thm:lsd-cont}, \ref{thm:lsd-exp}, and \ref{thm:toep-super-exp} under sub-exponential, exponential, and super-exponential random perturbations, when the symbol ${\bm a}(\zeta)= \zeta+ a_{-1} \zeta^{-1}$. \corAB{These three theorems show that analogs of \eqref{eq:toy} continue to hold for generic random perturbations and a wide class of symbols}.
\end{remark}

The next result is regarding \corAB{the} \abbr{LSD} of finitely banded Toeplitz matrices under macroscopic perturbations. Its proof essentially follows from \cite{S02} and  \cite{TVK10}. 
To state the result we need a couple of terminologies:
We let ${\bm c}$ to be the circular operator whose Brown measure is the uniform measure on the unit disk. Recall the definition of freeness from \cite[Definition 5.3.1]{AGZ}.

\begin{theorem}[macroscopic perturbation {\cite[Section 3.1]{BCC24}}]\label{thm:macro-lsd}
Fix $\upsigma > 0$. Let ${\bm a}: \mathbb{S}^1 \mapsto \C$ be a Laurent polynomial and let the entries of $E_N$ are i.i.d.~with zero mean and unit variance. Then the empirical measure $L_{T_N({\bm a}) + \upsigma N^{-1/2} E_N}$ converges \corAB{weakly, almost surely,} to $\bm{\mu}_{{\bm a}({\bm u})+ \upsigma {\bm c}}$, as $N \to \infty$, {where ${\bm u}$ and ${\bm c}$ are free}.
\end{theorem}

%The proof of Theorem \ref{thm:macro-lsd} is quite straightforward: Since $J_N$, the Jordan block converges to ${\bm u}$ in $\star$-moments and ${\bm a}$ is a Laurent polynomial it is immediate that $T_N({\bm a})$ converges to ${\bm a}({\bm u})$ in $\star$-moments. Therefore, by \cite[Theorem 6]{S02} the result follows when $E_N$ is the standard complex Ginibre matrix. The extension to an $E_N$ with i.i.d.~zero mean and unit variance follows from the replacement principle of Tao and Vu. 

\subsubsection{Finer description of the eigenvalues}
Theorems \ref{thm:lsd-cont}, \ref{thm:lsd-exp}, and \ref{thm:macro-lsd} do not shed any light whether the limits are well approximated by the eigenvalues of the finite dimensional matrices under consideration at the mesoscopic scale \corAB{(known as local laws in the literature)}. Such local laws have been considered \cite{OW23, SV} for the finitely banded case. One interesting find of \cite{OW23} is a one-to-one pairing between the eigenvalues of $T_N({\bm a})+N^{-\gamma}E_N$ and the push forward of the $N$-th roots of unity by the symbol. See the first two parts of the result below taken from \cite{OW23} for the Jordan block setting. The reader is referred to \cite{OW23} for such results for a wider class of Toeplitz matrices.

\begin{theorem}[{\cite[Theorem 1.1]{OW23}}]\label{thm:ow23}
Let $E_N$ be the random matrix with i.i.d.~$\pm 1$ valued entries of equal probability. Let $\{\xi_i\}_{i=1}^N$ be the $N$-th roots of unity.
\begin{enumerate}

\item[(i)] For any $\gamma >1/2$ and $p \ge 1$, there exists some $\vep >0$, so that, with probability $1-o(1)$, 
\[
\min_{\uppi} \left( \f1N \sum_{i=1}^N \left|\lambda_i(J_N+N^{-\gamma}E_N) - \xi_{\uppi(i)}\right|^p\right)^{1/p} \lesssim N^{-\vep},
\]
where the minimum is taken over all permutations $\uppi$ over $[N]$. 

\item[(ii)] For $\gamma \ge 3$ and any $\vep >0$, with probability $1-o(1)$,
\[
\max_{k \in [N]} \min_{i \in [N]} \left| \xi_k-\lambda_i(J_N+N^{-\gamma}E_N)\right| \lesssim \f{(\log N)^{1+\vep}}{N}. 
\]

\item[(iii)] Let $\varphi: \C \mapsto \C$ be a smooth function with compact support. For any $w \in \C$ and $a \ge 0$ let $\varphi_{w,a}(z):= \varphi(N^a(z-w))$. If $\gamma \ge 2+a$, then with probability $1-o(1)$,
\[
\left| \sum_{i=1}^N \varphi_{w,a}(\lambda_i(J_N+N^{-\gamma}E_N)) -\sum_{i=1}^N \varphi_{w,a}(\xi_i) \right| \lesssim \log N.
\] 
\end{enumerate} 
\end{theorem}

Theorem \ref{thm:ow23}(ii) shows that the minimum distance of the eigenvalues of $J_N+N^{-\gamma}E_N$ from $\mathbb{S}^1$ is $O((\log N)^{1+\vep}/N)$. It is natural to seek whether such a bound is optimal. The following complementary result taken from \cite{BVZ} is indeed the case (see \cite[Sections 5--7]{BVZ} for the general case). 

\begin{theorem}[{\cite[Theorems 5.4 and 7.1]{BVZ}}]\label{thm:bvz}
Let $E_N$ satisfy Assumption \ref{ass:mom}. Fix $\gamma >1$. Then, for any $\vep >0$, 
\[
\P\left(\exists i \in [N]: \lambda_i(J_N+ N^{-\gamma}E_N) \notin \D\left(0, 1-\f{(\corAB{\gamma -1 -\vep}) \log N}{N}\right)\right) \ll 1. 
\]
Additionally, assume that $E_N$ satisfies Assumption \ref{ass:anti-conc} for some $\upeta >1$. Then, there exists some $C_\gamma < \infty$, so that for any $\upalpha >0$,
\[
\P\left( \left|\left\{ i \in [N]: \lambda_i(J_N+ N^{-\gamma}E_N) \in \D\left(0, 1- \f{\upalpha \log N}{N}\right)\right\}\right| \le \f{C_\gamma N}{\upalpha}\right) \ll 1. 
\]
\end{theorem}
Let us add that the constraint $\gamma >1$ in Theorem \ref{thm:bvz} is sharp. This is supported by heuristic computations in the Jordan block case and simulations.

\subsubsection{Limiting spectral distribution for other matrix models}
Beyond the Toeplitz setting, results on spectral properties of random perturbation of non-normal matrices are very limited. Such results are available only when the matrix is a twisted Toeplitz matrix is triangular \cite{BPZ} or non-triangular with a periodic boundary \cite{V20}, or a random bidiagonal matrix \cite{BPZ} (sometimes referred to as one way Hatano-Nelson model). We collect the following result from \cite{BPZ} to illustrate the phenomena that even if $L_{\cM_N^{(1)}}$ and $L_{\cM_N^{(2)}}$ have the same limit for two sequences of non-normal matrices $\{\cM_N^{(1)}\}_{N \in \N}$  and $\{\cM_N^{(2)}\}_{N \in \N}$, their pseudospectral behaviors and \abbr{LSD}s under random perturbations can be vastly different. See also Figures \ref{fig8} and \ref{fig9}.

\begin{theorem}[{\cite[Theorem 1.3]{BPZ}}]\label{thm:HN}
Let $f: [0,1] \mapsto \C$ be a H\"older continuous function and $D_N^{(1)}$ be a $N \times N$ diagonal matrix such that its $i$-th diagonal entry is $f(i/N)$. Let $D_N^{(2)}$ be another $N \times N$ diagonal matrix with entries that are i.i.d.~copies of some random variable $X$. Assume that the law of $X$ is supported on a simply connected compact set in $\C$ with zero two dimensional Lebesgue measure. Let $\updelta_N$ be as in Theorem \ref{thm:sv-general} and $E_N$ satisfy Assumptions \ref{ass:hs} and \ref{ass:s-min}. Define $\cM_N^{(i)}:= J_N+D_N^{(i)}$, for $i=1,2$.
\begin{enumerate}

\item[(i)] $L_{\cM_N^{(1)}+\updelta_N E_N} \Lra \wt \nu$, in probability, as $N \to \infty$, where $\wt \nu$ is the law of \corAB{$f(U_1)+U_2$, $U_1$} and $U_2$ are uniforms on $[0,1]$ and $\mathbb{S}^1$, respectively, and independent of each other.    

\item[(ii)] $L_{\cM_N^{(2)}+\updelta_N E_N} \Lra \nu$, in probability, as $N \to \infty$, for some probability measure $\nu$ on $\C$ such that such that $\cL_\nu(z) = (\E \log |X-z|)_+$ for Lebesgue a.e.~$z \in \C$.

\end{enumerate}
\end{theorem}

\begin{figure}[htbp]
  \centering
   \begin{minipage}[b]{0.4\linewidth}
 % \centering
   \includegraphics[width=\textwidth]{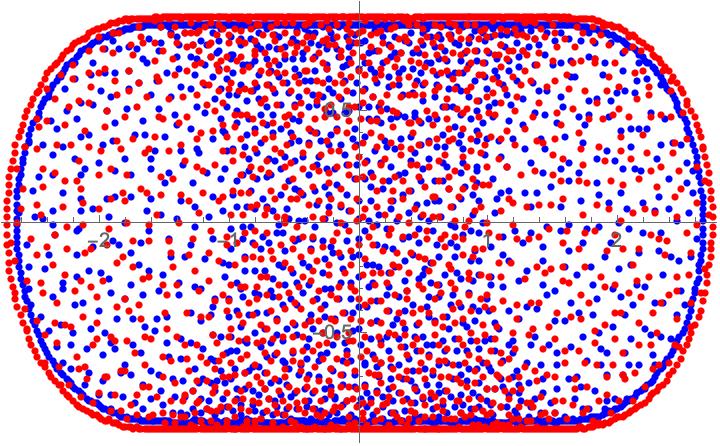}
 
%   \caption{}
%   \label{fig3}
 \end{minipage}
 \hspace{2cm}
 \begin{minipage}[b]{0.4\linewidth}
  \includegraphics[width=\textwidth]{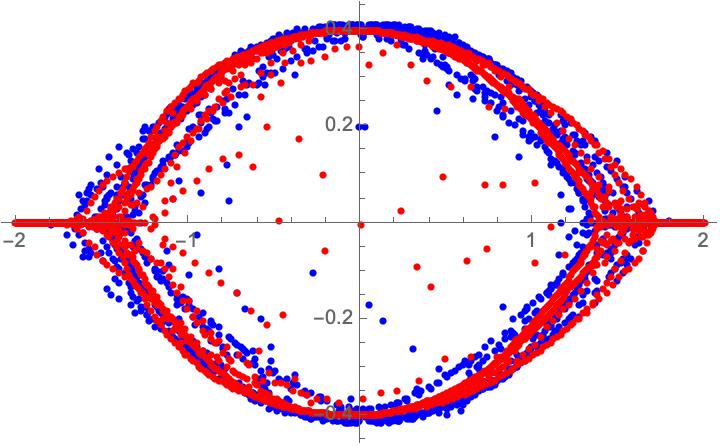}
%  \caption{}
%   \label{fig4}
 \end{minipage}
 \caption{Eigenvalues of $U_N A_N U_N^*$ are in blue, where $U_N$ is a simulated Haar unitary matrix via Mathematica, while eigenvalues of $A_N+N^{-3} G_N$ are in red, where $G_N$ is a matrix with i.i.d.~standard complex Gaussians.  $N=2000$. $A_N= \diag(\{-2+2i/N\}_{i \in [N]})+J_N$ (\texttt{left panel}) and $A_N= \diag(\{X_i\}_{i \in [N]})+J_N$\corAB{(\texttt{right panel})}, where $\{X_i\}_{i \in [N]}$ are i.i.d.~${\rm Unif}[-2,2]$.}
 \label{fig8}
\end{figure}

\begin{figure}[htbp]
  \centering
   \begin{minipage}[b]{0.4\linewidth}
 % \centering
   \includegraphics[width=\textwidth]{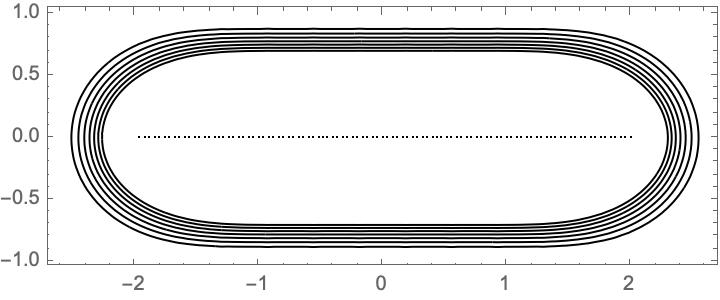}
 
%   \caption{}
%   \label{fig3}
 \end{minipage}
 \hspace{1cm}
 \begin{minipage}[b]{0.47\linewidth}
  \includegraphics[width=\textwidth]{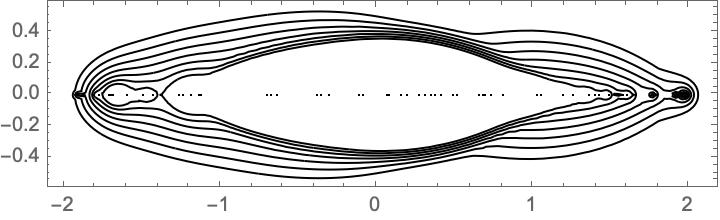}
%  \caption{}
%   \label{fig4}
 \end{minipage}
 \caption{$\vep$-pseudospectral level lines for $\vep=10^{-2}, 10^{-2.4}, \ldots, 10^{-4.4}$ for $A_N= \diag(\{-2+2i/N\}_{i \in [N]})+J_N$ (\texttt{left panel}) and $A_N= \diag(\{X_i\}_{i \in [N]})+J_N$ (\texttt{right panel}), where $\{X_i\}_{i \in [N]}$ are i.i.d.~${\rm Unif}[-2,2]$. $N=100$.}
 \label{fig9}
\end{figure}
\corAB{Let us remark in passing that motivated by the non-Hermitian Anderson model there have been quite a few works on understanding the properties of eigenvalues of non-Hermitian random tridiagonal matrices with periodic boundary conditions. See \cite{GK98, GK00, GK03, GS18}. For a certain range (determined by the Lyapunov exponent) of the `non-Hermitianity' parameter, these models, similar to the one in Theorem \ref{thm:HN}(ii), also exhibit the `bubble' plus `wing' like structure in the large $N$ limit.}

\subsection{Limiting distribution of outliers}
The reader may note from Figure \ref{fig7} that although most of the eigenvalues are near the symbol curves, there seem to be some which are of distance order one from those curves. One may wonder whether they are a finite $N$-effect or indeed there are eigenvalues in the limit that are of distance $O(1)$ from ${\bm a}(\mathbb{S}^1)$. We discuss two results, the first of which describes the region in the complex plane where no eigenvalues reside in the large $N$ limit. 

It is instructive to note that in the right panel of Figure \ref{fig7} there is a region enclosed by the symbol curve where no eigenvalues reside -- the reader may check that any $z$ in that region has winding number zero with respect to the symbol curve. This is the content of the first result. \corAB{Recall Proposition \ref{prop:BS99}.}

\begin{theorem}[{\cite[Theorem 1.1]{BZ} and \cite[Theorem 1.1]{BCC24}}]\label{thm:no-outlier}
Let ${\bm a}$ be as in \eqref{eq:laurent-poly}.
\begin{enumerate}
\item[(i)] Let the entries of $E_N$ be centered and have unit variance, and $\updelta_N \ll N^{-1/2}$ be such that $\log(1/\updelta_N)\ll N$. Then, for any $\vep >0$, on a set of probability $1-o(1)$, all eigenvalues of $T_N({\bm a})+\updelta_N E_N$ are contained in $\sigma(T({\bm a}))+ \D(0,\vep)$. 
\item[(ii)] Additionally assume that the entries of $E_N$ have finite fourth moments. Then, for any $\upsigma >0$ and any compact set $\cK \subset \C\setminus ({\rm supp}({\bm \mu}_{{\bm a}({\bm u})+\upsigma {\bm c}}) \cup \sigma (T({\bm a})))$, almost surely, for all large $N$, there are no eigenvalues of $T_N({\bm a})+ \upsigma N^{-1/2}E_N$ in $\cK$.
\end{enumerate}
\end{theorem}
The next result is on the distribution of the outliers in ${\rm supp}({\bm \mu}_{{\bm a}({\bm u})+\upsigma {\bm c}}) \cup \sigma (T({\bm a}))$ (for $\upsigma =0$, i.e.~in case of microscopic perturbation ${\rm supp}({\bm \mu}({\bm a}({\bm u}))) \subset \sigma (T({\bm a}))$). To state the result we need a few definitions. 
For $\D\subset \C$ we let $\cB(\D)$ denote the Borel $\sigma$-algebra on it. Recall that a Radon measure on $(\D, \cB(\D))$ is a measure that is finite for all Borel compact subsets of $\D$.
\begin{dfn}
A random point process $\Xi$ on an open set $\D \subset \C$ is a probability measure on the space of
all \corAB{nonnegative} integer valued Radon measures on $(\D,\cB(\D))$. Given a sequence of random point processes
$\{\Xi_n\}_{n \in \N}$ on $(\D,\cB(\D))$, we say that $\Xi_n$ converges weakly to a (possibly random) point process $\Xi$ on the same space, and write $\Xi_n \rightsquigarrow \Xi$, if for all compactly supported bounded real-valued continuous functions $f$ on $\D$ we have
\[
\int f(z) d\Xi_n(z) \Lra \int f(z) d\Xi(z), \quad \text{ as } n \to \infty. 
\]
\end{dfn}
Recall \eqref{eq:cS}. For $\upsigma >0$, set 
\[
\cS_{\upsigma, d}:=\left\{z \in \C \setminus {\rm supp}({\bm \mu}({\bm a}({\bm u})+\upsigma {\bm c}): {\rm wind}({\bm a}, z) =d\right\}.
\]
\begin{theorem}[{\cite[Theorem 1.11]{BZ} and \cite[Theorems 1.2 and Theorems 1.3]{BCC24}}]\label{thm:outlier}
Fix an integer $d \ne 0$. Let ${\bm a}$ be as in \eqref{eq:laurent-poly}. 
\begin{enumerate}
\item[(i)]Let $\updelta_N$ and $E_N$ be as in Theorem \ref{thm:no-outlier}(i). Additionally assume that the entries of $E_N$ are i.i.d.~satisfying Assumption \ref{ass:anti-conc}. Set $\Xi_N^d$ be the random point process induced by the eigenvalues of $T_N({\bm a})+\updelta_N E_N$ that are in $\cS_d$. Then there exists a random analytic function ${\bm F}^d$ on $\cS_d$ so that $\Xi_N^d \rightsquigarrow \Xi^d$, the zero set of ${\bm F}^d$. 
\item[(ii)] Let the entries if $E_N$ be i.i.d.~satisfying Assumption \ref{ass:mom}, and have a common symmetric distribution that is absolutely continuous with respect \corAB{to} the Lebesgue measure either on $\C$ or on $\varpi \R$ for some $\varpi \in \mathbb{S}^1$. Set $\Xi_N^{\upsigma, d}$ be the random point process induced by the eigenvalues of $T_N({\bm a})+\upsigma N^{-1/2} E_N$ that are in $\cS_{\upsigma, d}$. Then there exists a random analytic function ${\bm F}^{\upsigma, d}$ on $\cS_{\upsigma, d}$ so that $\Xi_N^{\upsigma, d} \rightsquigarrow \Xi^{\upsigma, d}$, the zero set of ${\bm F}^{\upsigma, d}$. 
\end{enumerate}
\end{theorem}
A few remarks are in order. The article \cite{BCC24} allows $E_N$ to be replaced by $U_N E_N U_N^*$ for $U_N$ a unitary matrix and some conditions on $U_N$ for the limit result -- to keep the exposition simple we have avoided including that generalization. Explicit expressions for the limiting random field can be found in \cite{BZ, BCC24}. It was noted in \cite{BZ, BC16} that for the Jordan block $\lim_{\upsigma \to 0} {\bm F}^{\upsigma, d} = {\bm F}^d$ uniformly on compact sets. From the proofs in \cite{BCC24} it follows that the same holds for any finitely banded Toeplitz matrix. 
It is worth pointing out, unlike results on \corAB{the} \abbr{LSD} where the limit is agnostic of the distribution of the entries of $E_N$, the limit of the random point process induced by the outliers depends heavily on the distribution of the entries. Further, unless $d = 1$, the limiting random field is not Gaussian even if the entries of $E_N$ are so.

\subsection{Localization of eigenvectors} Studying properties of eigenvectors of random matrices are of much interest. There is a huge literature studying different notions of delocalizations, such as $\sup$-norm delocalization \cite{ADK22, EKYY13, ESY09b, ESY09a}, no-gaps delocalization \cite{RV16}, quantum ergodicity \cite{AM15, BY17}, and eigenstate thermalization hypothesis \cite{CEH25, ER24} for the eigenvectors of various Hermitian random matrix ensembles (as well as localization properties in some cases \cite{ADK24, D25}). On the non-Hermitian side the literature is quite limited \cite{LO20, LT, RV16}. 

In this section we discuss localization properties of eigenvectors corresponding to `bulk' eigenvalues of randomly perturbed finitely banded Toeplitz matrices. Unlike the Wigner setting or the circular law limit setting, here we term an eigenvalue to be a bulk eigenvalue if it avoids neighborhoods of some `bad' points, defined below, determined by the symbol. 

\begin{dfn}\label{dfn:bad-pts}
Let ${\bm a}$ be as in \eqref{eq:laurent-poly}. Set $\cB({\bm a})$ to be the set of $z$'s such that the polynomial $\zeta \mapsto \zeta^{N_-} {\bm a}_z(\zeta)$ has double roots. Next, let $\ga_{{\bm a}}$ be the contracted symbol so that ${\bm a}(\zeta)=\ga_{{\bm a}}(\zeta^{\rm g})$, where ${\rm g}:={\rm gcd}\{|j|: j \ne 0 \text{ and } a_j \ne 0\}$. Define $\cB'({\bm a})$ to be the set of self-intersection points of $\ga_{{\bm a}}$. Set $\cB_\star({\bm a}):= \cB({\bm a}) \cup \cB'({\bm a})$ and $\cG_\vep:=\cG_{{\bm a}, \vep}:=  {\bm a}(\mathbb{S}^1)\setminus (\cB_\star({\bm a}) + \D(0,\vep))$, for any $\vep >0$. 
\end{dfn}

To ensure that there cannot be too many eigenvalues of $T_N({\bm a})+ N^{-\gamma}E_N$ in a neighborhood of $\cB_\star({\bm a})$ we work with the following assumption. 

\begin{assumption}\label{ass:symbol}
\corAB{$\cB_\star({\bm a})$ is a finite set.} 
\end{assumption}
By \cite[Lemma 11.4]{BG05} \corAB{the set} $\cB({\bm a})$ is always finite. On the other hand, by \cite{KK23} it follows that $\cB'({\bm a})$ is also finite unless $N_-=N_+$ and $|a_{-N_-}|=|a_{N_+}|$. 

A key to the localization property of most of the eigenvectors of polynomially vanishing random perturbation of finitely banded Toeplitz matrices is that most of its eigenvalues are at a distance of order $\log N/N$ from the symbol curves. Such a domain is definited below:
\beq\label{eq:Omega-ep-C}
\Omega(\vep, C, N):= \{z \in \C: C^{-1}\log N/N \le {\rm dist}(\cG_\vep, z) \le C \log N/M, d(z) \ne 0 \}, \quad 0 < \vep, C < \infty. 
\eeq
Since the symbol ${\bm a}$ and the strength of the random perturbation will be kept fixed, to lighten the notation we will write $\{\lambda_i^N\}_{i \in [N]}$ to denote the eigenvalues of $T_N({\bm a})+N^{-\gamma} E_N$. Further, for any $\lambda_i^N$ we will let $v_i^N$ to be a right eigenvector of unit $\ell^2$ norm corresponding to it. 
\begin{theorem}[{\cite[Theorems 1.5 and 1.6, and Corollary 1.7]{BVZ}}]\label{thm:eigenvec}
Let Assumptions \ref{ass:mom} and \ref{ass:symbol} hold. Further assume that Assumption \ref{ass:anti-conc} holds with $\upeta >1$. Fix $\mu >0$ and $\gamma >1$. 

\begin{enumerate}
\item[(i)] There exist $0 < \vep_\star, C_\star < \infty$ (depending on $\mu, \gamma$, and ${\bm a}$) such that for any $\vep \in (0,\vep_\star)$
\[
\P(\left(\left|\left\{ i \in [N]: \lambda_i^N \in \Omega(\vep,C_\star, N)\right\}\right| \le (1-\mu)N  \right) \ll 1. 
\]

\item[(ii)] Fix $\vep \in (0,\vep_\star/2)$. There exists some constant $c(\gamma) >0$ such that the following event holds with probability $1-o(1)$: For any $i \in [N]$ such that $\lambda_i^N \in \Omega(\vep, C_\star, N)$ and any $\ell \in [N]$ 
\[ 
      \|(v_i^N)_{[\ell, N]}\| {\bf 1}(d(\lambda_i^N) >0) + \|(v_i^N)_{[1, N-\ell]}\| {\bf 1}(d(\lambda_i^N) <0) \le c(\gamma)^{-1} \exp(-c(\gamma) \ell \log N/N). 
     \]
     \item[(iii)] Fix $\wh z = \wh z(N) \in \Omega(2\vep, 2 C_\star,N)$, large constants $C_0$ and $\wt C_0$, and $\eta \in (0,1)$ small. Then there exist constants 
$c_1=c_1(\eta,C_0, \wt C_0)$ and $c_0(\gamma) \in (0,1)$, with $\lim_{\gamma \to 1}c_0(\gamma)= 1$ and $\lim_{\gamma \to \infty} c_0(\gamma) = 0$, so that the following event holds with probability at least $1-\eta$ for all large $N$: For any $i \in [N]$ such that $\lambda_i^N \in \D(\wh z, C_0 \log N/N)$ and any $1 \le \ell \le \ell' \le \wt C_0 N/\log N$ satisfying $\ell' -\ell > N^{c_0(\gamma)}$  
\[ 
      \|(v_i^N)_{[\ell, \ell']}\|^2 {\bf 1}(d(\lambda_i^N) >0) + \|(v_i^N)_{[N-\ell', N-\ell]}\|^2 {\bf 1}(d(\lambda_i^N) <0) \ge c_1 (\ell'-\ell)\log N/N. 
     \]

\item[(iv)] There exist $\mu_1, \mu_2 >0$ so that with $|{\rm supp}_{\mu_1}(v)|:= \min\{|I|: \|v_I\| > 1-\mu_1\}$, for any $v \in \C^N$,
\[
\limsup_{N \to \infty} \f1N \E\left|\left\{i \in [N]: |{\rm supp}_{\mu_1}(v_i^N))| < \mu_2 N/\log N \right\} \right| \le \mu. 
\] 
\end{enumerate}
\end{theorem}
Let us explain the consequences of Theorem \ref{thm:eigenvec}. Its first part shows that indeed most of the eigenvalues are at a distance of order $\log N/N$ from ${\bm a}(\mathbb{S}^1)$. Theorem \ref{thm:eigenvec}(ii)-(iii) show that for all eigenvalues in the good region, the corresponding eigenvectors localize at scale $N/ \log N$, and for most eigenvalues, this is the scale at which the eigenvectors \corAB{spread} out. The last part consolidates on the first three parts to claim that for most eigenvectors one needs at least an order of $N/\log N$ many sites to accumulate a non-negligible $\ell^2$ mass. Let us also add that the restriction $\gamma >1$ is sharp both \corAB{for} the location of the eigenvalues with respect to the symbol curve, and \corAB{for} the localization of the eigenvectors. See Figure \ref{fig10} for the contrasting behavior of eigenvectors in the two regimes: $\gamma > 1$ and $\gamma \in (1/2,1)$. \corAB{Establishing the delocalization properties of} eigenvectors for $\gamma \in (1/2,1)$ is a work in progress \cite{BVZ2}.

\begin{figure}[htbp]
  \centering
   \begin{minipage}[b]{0.45\linewidth}
 % \centering
   \includegraphics[width=\textwidth]{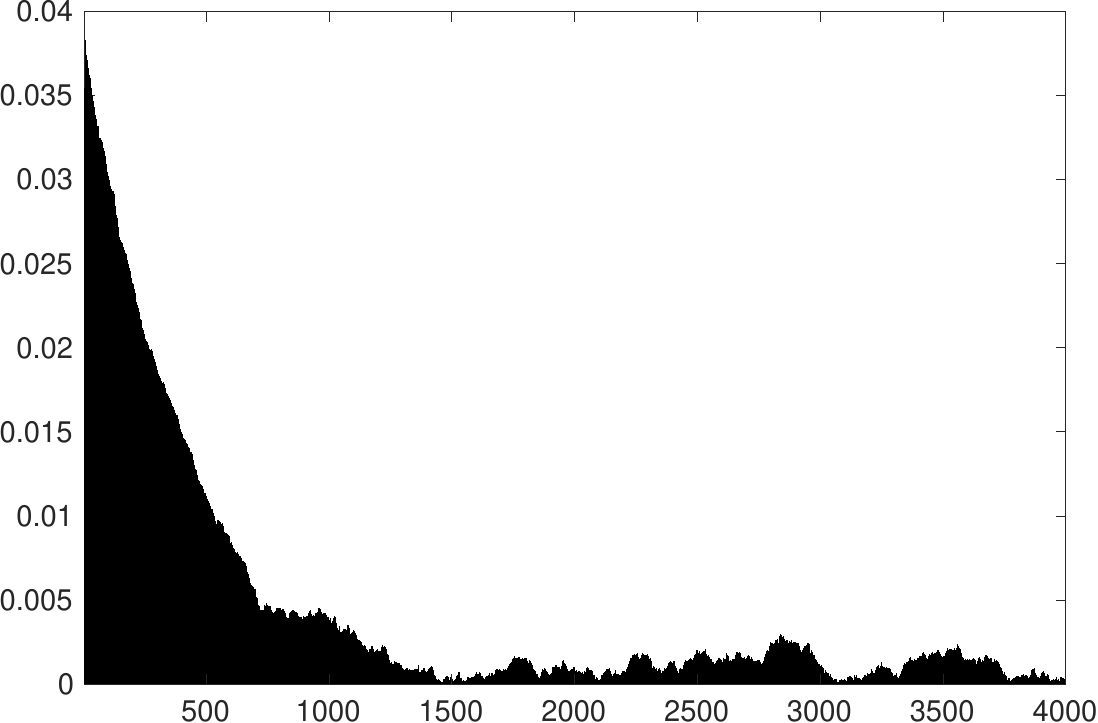}
 \end{minipage}
 \hspace{1cm}
 \begin{minipage}[b]{0.313\linewidth}
  \includegraphics[width=\textwidth]{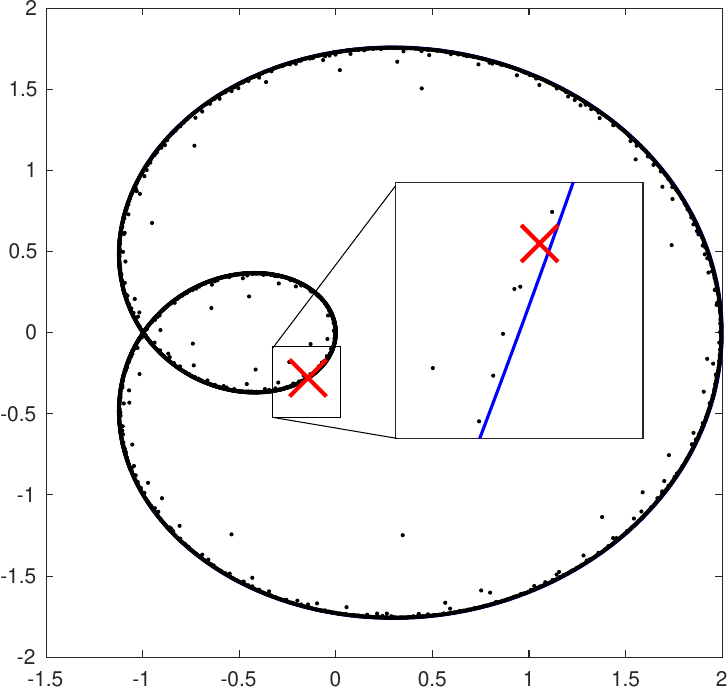}
 \end{minipage}
 \\[1.5ex]
  \begin{minipage}[b]{0.45\linewidth}
% % \centering
   \includegraphics[width=\textwidth]{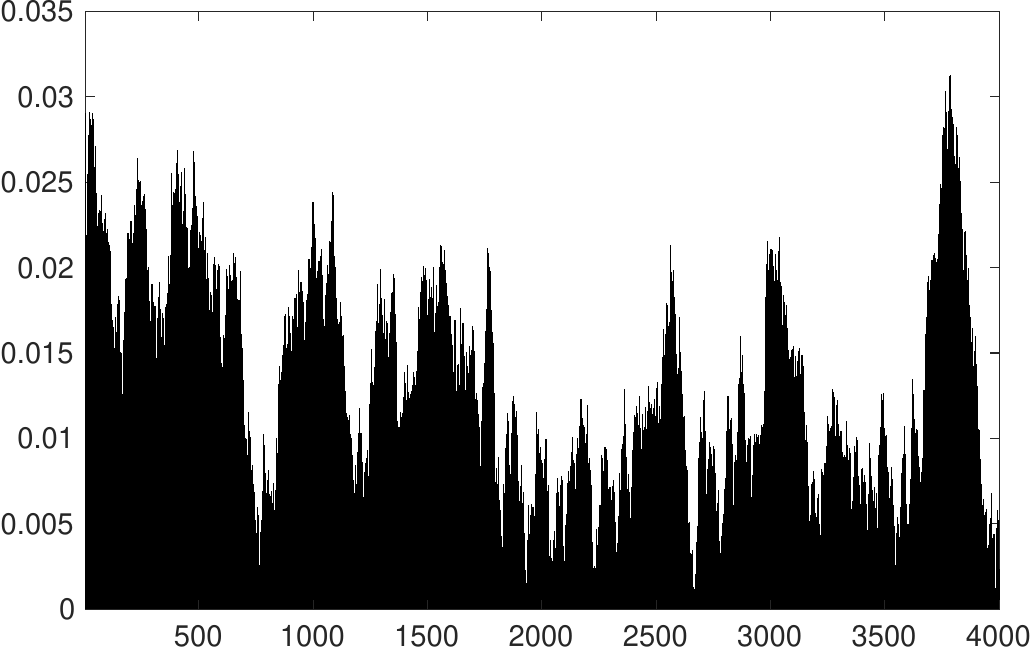}
 \end{minipage}
 \hspace{1cm}
  \begin{minipage}[b]{0.29\linewidth}
  \includegraphics[width=\textwidth]{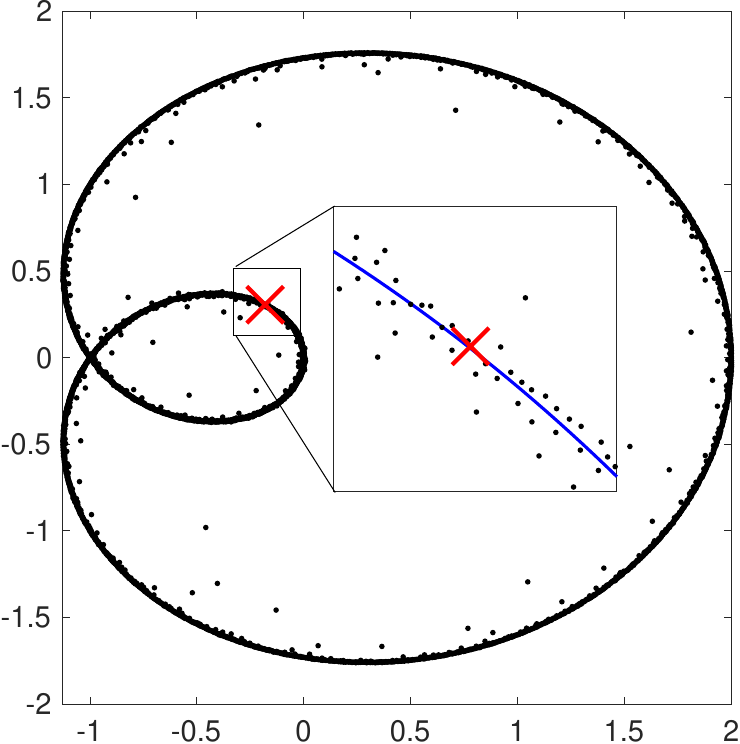}
 \end{minipage}
 \caption{Eigenvectors (\texttt{left panel}) and eigenvalues (\texttt{right panel}) for 
 $N=4000$, $\gamma=1.2$ (\texttt{top panel}) and $\gamma=0.8$ (\texttt{bottom panel}), and symbol ${\bm a}(\zeta)=\zeta+\zeta^{2}$.
Plotted are the moduli of the entries of the eigenvector that
 corresponds to the eigenvalue marked with a red $\times$. (reproduced with permission from \cite{BVZ})
}
  \label{fig10}
\end{figure}

It was predicted in \cite{BVZ} that eigenvectors corresponding to eigenvalues that are \corAB{at a} distance of order greater than $\log N/N$ from the symbol curve should have different localization lengths. In particular, those that are of distance $O(1)$ from ${\bm a}(\mathbb{S}^1)$ should localize on $O(1)$ many sites. We \corAB{confirm the prediction.} 
To keep the exposition short we only consider the Jordan block setting below. The same strategy can be used to treat the finitely banded setting. 

\begin{theorem}[Complete localization of eigenvectors for outlier eigenvalues]\label{thm:eigenvec-outlier}
Let \corAB{${\bm a}(\zeta)=\zeta^{-1}$}. Fix $\vep >0$ and $\gamma >1/2$. Let $E_N$ satisfy Assumption \ref{ass:spectral-norm}. There exists some constant $\wh c_0(\gamma) >0$ with $\lim_{\gamma \to 1/2} \wh c_0(\gamma)=0$ and $\lim_{\gamma \to \infty} \gamma^{-1} \wh c_0(\gamma)=1$ so that the following event holds with probability $1-o(1)$: For any $i \in [N]$ such that $\lambda_i^N \in \D(0,1-\vep)$, $\|v_i^N\|_\infty \ge \vep^{1/2}/2$ and $\|(v_i^N)_{[\ell, N]}\| \le (1-\vep)^{\ell-1}$ for all $1 \le \ell \le \wh c(\gamma) \log N$. 
\end{theorem}

%Although Theorem \ref{thm:eigenvec-outlier} imposes the restriction $\gamma >1$ it should extend for all $\gamma >1/2$. See Remark \ref{rmk:complete-loc} for a further discussion on this. 

\subsection*{{Open problems}}There are several interesting directions one can pursue. For example, it would be interesting to investigate finer details of the spectrum (such as local laws and eigenvalue spacing) in the setting of Theorem \ref{thm:lsd-cont}. Further, there have been no work addressing the microscopic behavior of eigenvalues (e.g.~$k$ point correlation function). It would be of interest to see if such behaviors depend on $\gamma$ or not, where $\updelta_N=N^{-\gamma}$. 

Although we do not discuss it here, applying Fej\'er's theorem and approximating ${\bm a} \in C(\mathbb{S}^1)$ by Laurent polynomials, one can extend Theorem \ref{thm:no-outlier} for continuous symbols. More challenging will be to find an analog of Theorem \ref{thm:outlier} beyond the finitely banded setup. It would be also interesting to understand the support of the \abbr{LSD} when the perturbation is exponentially vanishing, investigate whether there would be outliers, and if \corAB{so to identify their} limit distribution.

Regarding properties of eigenvectors much is yet to be understood beside those in \cite{BVZ, BVZ2}. We list a couple of unexplored directions. Based on analogies with pseudoeigenvectors (see \cite{TC04}) it is expected that the eigenvectors of randomly perturbed finitely banded twisted Toeplitz matrices should localize, but unlike \cite{BVZ} the localization site would depend on the location of the corresponding eigenvalues. On the other hand, for random bidiagonal matrices a more intricate behavior of the eigenvectors are expected. Eigenvectors corresponding to eigenvalues that are in the so called `Hatano-Nelson wing' are expected to be localized, while those corresponding to the eigenvalues in the `bubble' are believed to be delocalized.

\subsection*{Acknowledgements} The author thanks Mahan Mj for the invitation to contribute to the Proceedings of the International Colloquium, Tata Institute of Fundamental Research, 2024, which led to this article \corAB{and Ofer Zeitouni for helpful comments.} Research was partially supported by DAE Project no. RTI4001 via ICTS and by the Infosys Foundation via the Infosys-Chandrashekharan Virtual Centre for Random Geometry.

\section{Proofs of \corAB{limiting laws of eigenvalues}}\label{sec:proof-lsd}

First we consider the case of super-exponential perturbation. The result hinges on the following bounds on the determinant of a Toeplitz matrix. \corAB{To state the result, we need} to separate out a collection of bad $z$'s that are of zero Lebesgue measure.

\begin{dfn}\label{dfn:bad-points}
Let ${\bm a}$ be as \corAB{in \eqref{eq:laurent-poly} and $\Gamma({\bm a})$} be the set of $z$'s such that $\zeta_j(z)/\zeta_k(z) \in \mathbb{S}^1$ for some $j \ne k$. Recall $\cB({\bm a})$ from Definition \ref{dfn:bad-pts}. Set $\wh \Gamma({\bm a}):= \cB({\bm a}) \cup \Gamma({\bm a}) \cup \{a_0\}$. 
\end{dfn}

It is well known that $\Gamma({\bm a})$ is a union of finitely many analytic curves (cf.~\cite[pp.~265-266]{BG05}) and thus has a zero Lebesgue measure. Therefore, $\cB({\bm a})$ being a finite set, $\wh \Gamma({\bm a})$ has a zero Lebesgue measure. 

\begin{lemma}\label{lem:widom}
Let ${\bm a}$ be as in \eqref{eq:laurent-poly} and $z \notin \wh \Gamma({\bm a})$. Then, there exist some $\vep_0(z) >0$ such that
\[
\vep_0(z)\cdot |a_{N_+}|^N \prod_{j=1}^{N_+} |\zeta_j(z)|^N \le |\det(T_N({\bm a}_z))| \le \vep_0(z)^{-1} \cdot |a_{N_+}|^N \prod_{j=1}^{N_+} |\zeta_j(z)|^N.
\]
\end{lemma}

\begin{proof}
The proof is a straightforward consequence of Widom's formula \cite[Theorem 2.8]{BG05}. Indeed, by Widom's formula, for any $z \notin \cB({\bm a})$ we have
\beq\label{eq:widom}
\det T_N({\bm a}_z) = \sum_{\cI \in \binom{[N_++N_-]}{N_+}} C_{\cI} \cdot a_{N_+}^N (-1)^{N N_+}\prod_{\ell \in \cI} \zeta_\ell(z)^N, \qquad \text{ where } \quad C_{\cI}:= \prod_{\substack{j \in \cI\\k \notin \cI}} \f{\zeta_j(z)}{\zeta_j(z)-\zeta_k(z)}.
\eeq
Next note that for any $z \notin \cB({\bm a}) \cup \{a_0\}$ (and thus none of the roots are zero) there exists some $\vep(z) >0$ such that 
\[
\vep(z) \le \inf_\cI |C_{\cI}| \le \sup_\cI |C_{\cI}| \leq \vep(z)^{-1}.
\]
On the other hand, $z \notin \Gamma({\bm a})$ implies 
\[
\inf_{\substack{j \in [N_+]\\ k \notin [N_+]}} \left|\f{\zeta_j(z)}{\zeta_k(z)} \right| > 1. 
\]
Plugging the last two estimates in \eqref{eq:widom} we obtain the desired upper and lower bounds. 
\end{proof}

\begin{proof}[Proof of Theorem \ref{thm:toep-super-exp}]
Fix any $z \notin \wh \Gamma({\bm a})$. By the lower bound in Lemma \ref{lem:widom} the \corAB{matrix} $T_N({\bm a}_z)$ is invertible for any such $z$. Therefore
\begin{multline}
\f1N \log |\det(T_N({\bm a}_z)+\updelta_N E_N)|- \f1N \log |\det T_N({\bm a}_z)| = \f1N \log |\det({\rm Id}_N+ \updelta_N T_N({\bm a}_z)^{-1} E_N)| \\
= \f1N \sum_{i=1}^N \log |1+ \updelta_N \lambda_i(T_N({\bm a}_z)^{-1} E_N)|.
\end{multline}
Since $E_N$ satisfies Assumption \ref{ass:hs}, by Markov's inequality $\P(\Omega_N) \ll 1$, where $\Omega_N:=\{\|E_N\| \ge N^{1+c}\}$ and $c >0$. Recall \eqref{eq:T-op-norm}. Applying the lower bound from Lemma \ref{lem:widom} yet again, the fact $N^{-1} \log (1/\updelta_N) \gg 1$, and Weyl's inequality we deduce that 
\[
\updelta_N \max_{i=1}^N |\lambda_i(T_N({\bm a}_z)^{-1} E_N)| \leq \updelta_N \|T_N({\bm a}_z)^{-1}\| \cdot \|E_N\| \le \updelta_N N^{1+c} \cdot \|{\bm a}_z\|_{\infty, \mathbb{S}^1}^{N-1} \cdot C(z)^{N} \ll 1,
\]
on the event $\Omega_N^c$, for some $C(z) < \infty$. Therefore, $\cL_{L_{T_N({\bm a})+\updelta_N E_N}} \to \cL_{{\bm a}}$, in probability, as $N \to \infty$, for Lebesgue a.e.~$z \in \C$. Hence, upon applying Lemma \ref{lem:cL-converges}(ii) the proof completes. 
\end{proof}

\subsection{Deterministic equivalence} Given a deterministic matrix this approach attempts to approximate \corAB{the $\log$-determinant of $M_N+ \updelta_N E_N$ by the empirical average of the $\log$-singular values} of $M_N$ that are not `too small'. Therefore, applying this deterministic equivalence with $M_N$ replaced by $M_N -z {\rm Id}_N$ one \corAB{can approximately} express the $\log$-potential of $L_{M_N+\updelta_N E_N}$ \corAB{at a $z$} by \corAB{the empirical average of the $\log$-singular values} of $M_N - z {\rm Id}_N$ truncated at some small value. The latter being deterministic maybe easier to handle. This approach has been successfully employed in \cite{BPZ} to study the \abbr{LSD} of twisted Toeplitz and random bidiagonal matrices. The result below is taken from \cite{VZ21}, which unlike \cite{BPZ} allows for generic random perturbations.

\begin{theorem}[{\cite[Theorem 3]{VZ21}}]\label{thm:VZ}
Let $M_N$ be a deterministic $N \times N$ matrix with $\|M_N\| \lesssim N^{\kappa_1}$ and $E_N$ be a random matrix satisfying Assumptions \ref{ass:spectral-norm} and \ref{ass:s-min}. Fix $L >0$ and $N^{-L} \lesssim \alpha \le 1$. Assume that 
\[
\left|\{j: s_j(M_N) \le \alpha\}\right| =:M \ll \f{N}{\log N}.
\]
Then, \corAB{for any $C< \infty$ with $N^{-C} \le \updelta_N \ll N^{-1/2}$,} provided $\alpha \gg \updelta_N N^{1/2}$, there exists some constant $C' < \infty$ such that
\[
\P\left(\left| \f1N \log |\det (M_N+\updelta_N E_N)| - \f{1}{N} \sum_{j: s_j(M_N) > \alpha} \log s_j(M_N) \right| =o(1) \right)= 1 -o(1). 
\]
\end{theorem}

\subsubsection{Grushin problem and the proof of Theorem \ref{thm:VZ}}\label{sec:grushin} The Grushin problem is a simple algebraic tool used often to study spectral problems where one replaces the operator in context by a suitable operator on an enlarged domain that is `well' invertible. The origin of this construction goes back to \cite{G71} where it was used to study hypoelliptic operators. It has also been quite useful in bifurcation 
theory, numerical analysis, and for treatments of spectral problems arising in 
electromagnetism and quantum mechanics. See also the review paper \cite{SZ07} for various uses of the Grushin problem.

Let $A=A_N$ be a  complex $N\times N$-matrix with singular values $0\leq t_1 \leq t_2 \leq \cdots \leq t_{N}$, corresponding right singular vectors $e_1, e_2, \dots,e_{N}\in \C^N$, and left singular vectors $f_1,f_2,\dots,f_{N}\in \C^N$ (note that these two sets of vectors are two orthonormal bases in $\C^N$). That is,  
 \begin{equation}\label{gp2}
	A^* f_i = t_i e_i, \quad Ae_i = t_i f_i, \quad i=1,\dots, N.
\end{equation}
Fix $ \alpha >0 $ `small' and let $M>0$ be the number of singular values $t_i
\in 
[0,\alpha]$, i.e. 
 \begin{equation}\label{gp2.1}
   0 \leq t_1 \leq \cdots \leq t_M \leq \alpha < t_{M+1}\leq \cdots\leq t_N.
\end{equation}
Let $\delta_i$, $1\leq i \leq M$, denote an orthonormal basis of $\C^M$.
Set 
\begin{equation}\label{gp5}
R_+:=\sum_{i=1}^M \delta_i e_i^* \quad  \text{ and } \quad R_-:=\sum_{i=1}^M f_i \delta_i^*.
\end{equation}
Then the Grushin problem for the unperturbed operator is defined as follows:
\begin{equation}\label{gp6}
\corAB{\mathcal{P}: = \begin{pmatrix}
			A & R_- \\ R_+ & 0 \\
		\end{pmatrix} : \C^N\times \C^M \mapsto \C^N\times \C^M.} 
\end{equation}
It is not difficult to check that the map $\cP$ is bijective. It further follows that 
\begin{equation}\label{gp7.2}
		\mathcal{P}^{-1}= :\mathcal{E} =
		\begin{pmatrix} 
		E  & E_+\\ 
		E_- & E_{-+}\\
		\end{pmatrix} ,
\end{equation}
where
\begin{equation}\label{gp8}
\begin{split}
		&E := \sum_{i=M+1}^{N} \frac{1}{t_i} e_i  f_i^*, \quad 
	         E_+ := \sum_{i=1}^M e_i  \delta_i^*,  \quad
		E_- :=\sum_{i=1}^M  \delta_i f_i^*, \quad \text{ and } \quad
		 E_{-+} := - \sum_{i=1}^M t_j \delta_j \delta_j^*.
\end{split}
\end{equation}
From \eqref{gp2.1} and \eqref{gp8} it follows that we have the following norm estimates:
\begin{equation}\label{gp9}
	\|E\| \leq \frac{1}{\alpha}, \quad \| E_{\pm } \| =1, \quad \text{ and } \quad
	\| E_{-+}\| \leq \alpha.
\end{equation}
Further, using that $\{e_j\}_{j=1}^N$ and $\{f_j\}_{j=1}^N$ are orthonormal bases in $\C^N$, we \corAB{find}
\[
\cE^* \cE= \begin{pmatrix}
\sum\limits_{j= M+1}^N \f{1}{t_j} f_j f_j^* + \sum\limits_{j =1}^M f_j f_j^* &-\sum\limits_{j=1}^M t_j  f_j \delta_j^* \\
-\sum\limits_{j=1}^M t_j \delta_j f_j^*& \sum\limits_{j=1}^M (1+t_j^2) \delta_j \delta_j^*\\
\end{pmatrix},
\]
and subsequently using the formula for the determinant of a block matrix we \corAB{derive}
\beq\label{eq:det-cP}
|\det \cP | = \prod_{j=M+1}^N t_j.
\eeq
Next, as we work with a perturbation of a  given matrix we need to extend the notion of Grushin problem to such setting. Namely, we set 
\begin{equation}\label{gpp1}
	A^\delta:=A+\delta Q, \quad \delta \geq 0,
\end{equation}
where $Q$ is a complex $N\times N$-matrix (eventually, random). 
Let $R_{\pm}$ be 
as in \eqref{gp5}, and set 
\begin{equation}\label{gpp2}
\mathcal{P}^{\delta}: = \begin{pmatrix}
			A^{\delta} & R_- \\ R_+ & 0 \\
		\end{pmatrix} : \C^N\times \C^M \mapsto  \C^N\times \C^M.
\end{equation}
Applying $\mathcal{E}$ (see \eqref{gp7.2}) from
the right to \eqref{gpp2} yields 
\begin{equation*}%\label{gpp3}
  \mathcal{P}^{\delta}\mathcal{E} = {\rm Id}_{N+M} + 
\begin{pmatrix}
			\delta Q E &  \delta Q E_+\\ 0 & 0 
		\end{pmatrix}. 
\end{equation*}
Now under the additional assumption that 
\begin{equation*}%\label{gpp4}
	2\delta \|Q\| \alpha^{-1} \leq 1,
\end{equation*}
upon applying the bound \eqref{gp9} we find that \corAB{$({\rm Id}_N+\delta Q E)$} is invertible. It is then 
straightforward to check that $\mathcal{P}^\delta$ is invertible, with
inverse
\begin{equation}\label{eq:cE-delta}
	(\mathcal{P}^{\delta})^{-1} 
=: \mathcal{E}^{\delta} = \begin{pmatrix} 
		E^{\delta}  & E^{\delta}_+\\ 
		E^{\delta}_-& E^{\delta}_{-+}
		\end{pmatrix},
\end{equation}
where
\begin{equation}
  \label{eq-Edelta}
E^\delta = E ({\rm Id}_N+ \delta Q E)^{-1}, E_-^\delta= E_- ({\rm Id}_N+ \delta Q E)^{-1},  E^\delta_{-+}= E_{-+} -E_- ({\rm Id}_N+\delta Q E)^{-1} \delta Q E_+,
\end{equation}
and 
%\begin{equation}
%\label{eq-march1a}
%E^\delta_{-+}= E_{-+} -E_- (I+\delta Q E)^{-1} \delta Q E_+,
%\end{equation} 
%and
\begin{equation}
\label{eq-march1b}
E^\delta_+= E_+-E({\rm Id}_N+\delta Q E)^{-1}\delta Q E_+.
\end{equation}
The following norm estimates \corAB{are then easy} to obtain:
 \begin{equation}\label{gpp6}
 \begin{split}
		&\| E^{\delta} \| = \|  E( {\rm Id}_N+ \delta Q E)^{-1} \| \leq 2 \|E\| \leq 2 \alpha^{-1}, \\
%		&\| E_+^{\delta} \| = \|  ( 1+ \delta Q E)^{-1}E_+ \| \leq 2 \|E_+\| \leq 2,  \\
%		&\| E_-^{\delta} \| = \| E_- ( 1+ \delta Q E)^{-1} \| \leq 2\|E_-\| \leq 2,  \\
		&\| E_{-+}^{\delta} -E_{-+}\| = 
		\| E_- ( {\rm Id}_N+ \delta Q E)^{-1}\delta Q E_+ \| \leq 2 \|\delta Q  \| \leq  \alpha.  \\
\end{split}
\end{equation}
To conclude the proof we notice that $\f{d}{d \delta} \log \det \cP^\delta = \tr (\cE^\delta \f{d}{d\delta} \cP^\delta)$ and ${\rm Re} \log \det \cP^\delta = \log |\det \cP^\delta|$. Therefore, noting that $\cP^0=\cP$ and using \eqref{eq-march1b},
\begin{equation*}
\left|\log |\det \cP^\delta| - \log |\det \cP^0|\right|  = \left| {\rm Re} \int_0^\delta \tr (\cE^\tau \f{d}{d\tau} \cP^\tau) d\tau \right|  = \left| {\rm Re} \int_0^\delta \tr(E^\tau Q) d\tau \right| \leq 2 \alpha^{-1} \delta N \|Q\|.
\end{equation*}
On the other hand, using the Schur complement formula we find 
\[
\log |\det A^\delta| = \log |\det \cP^\delta| + \log |\det E_{-+}^\delta|.
\]
The last two observations \corAB{together with \eqref{eq:det-cP} yield} 
\[
\left| \f1N \log |\det (M_N+\updelta_N Q)| - \f{1}{N} \sum_{j: s_j(M_N) > \alpha} \log s_j(M_N) \right| \leq 2 \alpha^{-1} \updelta_N\|Q\| + \f{1}{N} \left|\log |\det E_{-+}^\delta|\right| = o(1),
\]
on the event
\[
\cT_N:=\{\|Q\| \le \gc_N(\alpha) \cdot N^{1/2}\} \cap \{\log|\det E_{-+}^\delta| \le M |\log (2\alpha)|\} \cap \{\log|\det E_{-+}^\delta| \ge -\corAB{\beta M \log N} \},
\]
\corAB{where $\gc_N(\alpha):= \alpha N^{-1/2} /\updelta_N$, %$\gd_N:= \log c_N(1)/\log N$
 and $\beta=\beta(\kappa+\gamma)$ of Assumption \ref{ass:s-min}}. From \cite[Lemma 18]{V20} we also have that $s_{\min}(A^\delta) \le s_{\min}(E_{-+}^\delta)$. Therefore, from the \corAB{assumptions on $\updelta_N$, $\alpha$, and $M$,} Assumptions \ref{ass:spectral-norm} and \ref{ass:s-min}, \corAB{and upon recalling \eqref{gp9} and \eqref{gpp6} it} follows that $\P(\cT_N) \ll 1$ completing the proof of Theorem \ref{thm:VZ}. 
\qed

\begin{remark}
Although the formulation is different from \eqref{gp6}, \cite{SV, SV1} rely on Grushin problems to derive \abbr{LSD} of randomly perturbed Toeplitz matrices. In particular, \cite{SV1} uses a composition of two Grushin problems with the first one essentially allowing them to replace the Toeplitz matrix under consideration with an appropriately chosen \corAB{circulant matrix}, while the second one enables them to ignore certain small singular values, similar in spirit to the approach employed in the proof of Theorem \ref{thm:VZ}. 
\end{remark}

%\subsection{Approximation by a circulant matrix via construction of Grushin problems} \red{Needs to address}

\subsection{Determinant expansion}\label{sec:det-expansion} At a high level this approach relies on expanding the determinant of $T_N({\bm a}_z) + \updelta_N E_N$, expressing it as a sum of homogeneous polynomials in the entries of $E_N$ of degrees varying from zero to $N$, and identifying the dominant term in that expression. This approach has been useful in studying the limiting spectral distribution, locating the zone of no outliers, and the limiting outlier distribution in the complement of that zone, for polynomially randomly perturbed finitely banded Toeplitz matrices (see \cite{BPZ2, BZ}). Beyond these settings this general idea of determinant expansion (together with a joint combinatorial central limit theorem) has been fruitful in deriving convergence of the spectral radius of various non-Hermitian random matrix ensembles, e.g.~\cite{BCG22, C23, CGZ23, HL25, QG23}. 

In this section we explain the determinant expansion approach in somewhat detail and later apply them to derive \abbr{LSD} under an exponentially small vanishing perturbation. The starting point of this approach is the following simple identity: Let $A_N$ and $B_N$ be two $N \times N$ matrices. Then 
\begin{equation}\label{eq:det_decomposition}
        \det(A_N+B_N) = \sum_{\substack{\row,\col \subset [N] \\ |\row|=|\col|}} (-1)^{\sgn(\sigma_\row)\sgn(\sigma_\col)} \det(A_N[{\row}^c; {\col}^c])\det(B_N[\row; \col]),
\end{equation}
where ${\row}^c:=[N]\setminus \row$, ${\col}^c:=[N] \setminus \col$ and $\sigma_Z$ for $Z\in\{\row,\col\}$ is the permutation on $[N]$ which places all the elements of $Z$ before all the elements of ${Z}^c$, but preserves the order of elements within the two sets. Its proof follows straight from the definition of the determinant,
  see e.g. \cite{college}. Applying \eqref{eq:det_decomposition} in our setting we notice that
  \[
  \det(T_N({\bm a}_z)+ \delta_N E_N) = \sum_{k=0}^N {\det}_k(z),
  \]
  where ${\det}_0(z):= \det(T_N({\bm a}_z)) = \det(T_N({\bm a}_z)^t)$ and
  \beq\label{eq:det-k}
  {\det}_k(z):=  \updelta_N^k \sum_{\substack{X, Y \subset [N]\\ |X|=|Y|=k}} (-1)^{\sgn(\sigma_{X}) \sgn(\sigma_{Y})} \det (T_N({\bm a}_z)^t[{X}^c; {Y}^c]) \cdot \det (E_N^t[X; Y]), \qquad k \in [N].
  \eeq
  Identifying the order of magnitude of ${\det}_k(\cdot)$ requires the same for $\det (T_N({\bm a}_z)[{X}^c; {Y}^c])$ for all $X, Y \subset [N]$. For certain choices of $X$ and $Y$ this sub matrix is still a Toeplitz matrix, and hence one can either use Widom's formula \cite[Theorem 2.8]{BG05} (in the case of no double roots of the associated symbol) or Trench's formula \cite[Theorem 2.10]{BG05} (in the case of a double root). For arbitrary choices of $X$ and $Y$ one can only express the determinant as some Schur polynomial (see \cite{A12}) and latter is not easy to work with. 
  
To tackle this issue the following two stage approach was adopted in \cite{BVZ2, BZ}. First let us assume that $T_N({\bm a})$ is a lower triangular matrix, i.e.~$N_-=0$ in \eqref{eq:laurent-poly}. In this case, recalling that $\{\zeta_j(z)\}_{j=1}^{N_+}$ are the roots of the polynomial $\zeta \mapsto{\bm a}_z(\zeta)$ we note that 
\beq\label{eq:toep-prod}
T_N({\bm a}_z)^t= {\bm a}_z(J_N) = a_{N_+} \cdot \prod_{j=1}^{N_+} (J_N -\zeta_j(z)). 
\eeq
Hence, if one is able to determine minors of bidiagonal matrices then one can hope to use Cauchy-Binet theorem to determine the order of magnitude of minors of $T_N({\bm a}_z)$. The following lemma provides the necessary estimate on minors of a bidiagonal matrix.

\begin{lemma}[{\cite[Lemma 2.3]{FPZ}}]\label{lem:bidiagonal-det-1}
Let $A_N = J_N+ \mathfrak{z} \Id_N$, $\mathfrak{z} \in \C$, $X =\{x_1 <x_2<\cdots <x_k\} \subset [N]$, and $Y=\{y_1<y_2<\ldots<y_k\}\subset [N]$. Then, with
$y_{k+1}=\infty$,
\[
\det(A_N[{X}^c; {Y}^c]) = \mathfrak{z}^{y_1-1} \cdot \left(\prod_{i=2}^k  \mathfrak{z}^{y_i-x_{i-1}-1}\right) \cdot \mathfrak{z}^{N- x_k} {\bf 1} \left\{  y_i \le x_i < y_{i+1}, i \in [k] \right\}.
\]
\end{lemma}

When $T_N({\bm a})$ is non-triangular the representation \eqref{eq:toep-prod} no longer holds. Nevertheless, one notes that in such cases $T_N({\bm a})$ can still be viewed as a sub matrix of a triangular Toeplitz matrix and therefore one can obtain an analog of \eqref{eq:toep-prod}. To explain the idea more clearly let us introduce the following definition.  
\begin{dfn}[Toeplitz with a shifted symbol]\label{dfn:toep-shifted}
Let ${\bm a}$ be as in \eqref{eq:laurent-poly}. 
For $n> N_+ + N_-$, $z \in \C$  and
$\bar N_+, \bar N_- \in \N_0$ such that $\bar N_+ + \bar N_-=N_+ +N_-\defeq \widetilde N$,
let $T_n({\bm a}, z; \bar N_+)$ denote the $n \times n$ Toeplitz matrix with symbol $\wt{\bm a}_{z, \bar N_+}$, where %$\wh{\bm a}(\lambda):= \wt {\bm a}(1/\lambda)$, 
\(
\widetilde {\bm a}_{z, \bar N_+}(\lambda) \defeq \sum_{j=-N_-}^{N_+} a'_j \lambda^{-(j+ \bar N_+  - N_+)},
\)
and $a'_j := a_j - z \delta_{j,0}$, for $j=-N_-, -N_-+1, \ldots, N_+$. 
\end{dfn}

From Definition \ref{dfn:toep-shifted}
it follows that
%with $\wh N_{+}=N+N_+$,
\begin{equation}\label{eq:P-wtN}
 \corAB{ T_{\wh N_{-}}({\bm a}, z;\widetilde N)}=\begin{bmatrix}
a_{-N_-} &\cdots &a_0-z& \cdots & a_{N_+}&0&\cdots& 0\\
0& a_{-N_-} & & a_0-z & &\ddots& & \vdots\\
%0 &0& a_{-d_2} & \ddots & a_0-z&\ddots&a_{d_1}& 0\cdots\\
\vdots & \ddots &\ddots &&\ddots&&\ddots & \vdots\\
\vdots & & \ddots & \ddots &&\ddots& & a_{N_+}\\
\vdots & &  & \ddots & \ddots & &\ddots & \vdots\\
\vdots && &  &\ddots& \ddots & & a_0 -z\\
\vdots & & & &&  \ddots & \ddots & \vdots \\
 0 & \cdots & \cdots & \cdots&\cdots &\cdots & 0 & a_{-N_-}
\end{bmatrix}
, \qquad \wh N_{-}:=N+N_-.
\end{equation}
Therefore, 
\[
T_N({\bm a}_z)^t=T_N({\bm a}, z;N_+) = T_{\wh N_{-}}( {\bm a}, z; \widetilde N) [[N]; [\wh N_-]\setminus [N_-]]. 
\]
Recalling \eqref{eq:zeta-dfn}, letting $\eta_j(z) := -\zeta_j(z)$ for $j \in [\wt N]$, and setting $\wh{\bm a}_z(\zeta) = \zeta^{N_-} {\bm a}_z(\zeta)$ we note that
\[
\corAB{T_{\wh N_-}({\bm a}, z; \wt N)} = \wh{\bm a}_{z}(J_{\wh N_-}) = a_{N_+} \corAB{\prod_{j=1}^{\wt N}} (J_{\wh N_-}+ \eta_j(z){\rm Id}_{\wh N_-}). 
\]
Using Cauchy-Binet theorem we obtain from \eqref{eq:det-k}
\begin{eqnarray}\label{eq:det-decompose-1}
{\det}_k(z) 
&= &\updelta_N^k \sum_{\substack{X, Y \subset [N]\\ |X|=|Y|=k}} \sum_{i=2}^{\wt N-1} \sum_{\substack{X_i \subset [\wh N_-]\\ |X_i|=k+N_-}} (-1)^{\sgn(\sigma_{X}) \sgn(\sigma_{Y})}a_{N_+}^{N-k} \cdot \prod_{i=1}^d \det\left((J_{\wh N_-} + \eta_i(z)\Id_{\wh N_-})[\COMP{X}_i; \COMP{X}_{i+1}]\right)\nonumber\\
  &&\qquad \qquad \qquad\cdot\det (E_N^t[X; Y]),
\end{eqnarray}
where 
\begin{equation}\label{eq:X-Y-constraint}
X_1:= X_1(X):=X \cup [\wh N_-]\setminus [N], \qquad X_{\wt N+1}:=X_{\wt N+1}(Y):= (Y+N_-) \cup [N_-],
\end{equation}
and $\COMP{Z}:=[\wh N_-]\setminus Z$ for any set $Z \subset [\wh N_-]$.  Note the notational difference between $\COMP{Z}$ and $Z^c$. The \corAB{former is used} to write the complement of $Z$ when viewed as a subset of $[\wh N_-]$, where for the \corAB{latter $Z$ is viewed as} a subset of $[N]$. 

Lemma \ref{lem:bidiagonal-det-1} allows one to evaluate each term in the summand in  the \abbr{RHS} of \eqref{eq:det-decompose-1}. However, to obtain any meaningful and usable bound we need some further preprocessing. The idea is to express the sums in \eqref{eq:det-decompose-1} over $\{X_i\}_{i=1}^{\wt N+1}$ as an iterated sums, where the inner sum will be over the choices of $\{X_i\}_{i=1}^{\wt N+1}$ such that the product of the determinants of the bidiagonal matrices is constant and the outer sum will be over all possible values of that product of the determinants. This requires a further set of notation. 

We let
\begin{equation}
\label{eq-Lnew}
X_i := \left\{ x_{i,1} < x_{i,2} < \cdots < x_{i,k+N_-} \right\}, \quad i \in [\wt N+1]; \quad \quad \cX_k:=(X_1, X_2, \ldots, X_{\wt N+1}).
\end{equation}
Next for any ${\bm \ell}:=(\ell_1,\ell_2,\ldots,\ell_{\wt N})$ with $0 \le \ell_i \le \wh N_-$ for 
$i \in [\wt N]$, and $k\in[N]$, we define
\begin{multline*}
L_{{\bm \ell},k}:= \{ \cX_k:  \, 1 \le x_{i+1,1} \le x_{i,1} < x_{i+1,2} \le x_{i,2} < \cdots < x_{i+1,k+N_-} \le x_{i,k+N_-} \le \wh N_-,\\
\text{ and }  \, x_{i+1,1}+ \sum_{j=2}^{k+N_-}(x_{i+1,j}-x_{i,j-1})+ (\wh N_--x_{i,k+d_2}) =\ell_i+k+N_-, \text{ for all } i \in [\wt N]\}.
%\label{eq-Lnew1}
\end{multline*}
Note that Lemma \ref{lem:bidiagonal-det-1} implies that the summand in \eqref{eq:det-decompose-1} is non-zero only when $\cX_k \in L_{{\bm \ell}, k}$ for some ${\bm \ell}$ and in that case
\begin{equation}\label{eq:prod-decomp}
a_{N_+}^N \prod_{i=1}^{\wt N} \det\left((J_{\wh N_-} + \eta_i(z)\Id_{\wh N_-})[\COMP{X}_i; \COMP{X}_{i+1}]\right) = a_{N_+}^N \prod_{i=1}^{\wt N} \lambda_i(z)^{\ell_i} =: {\bm \eta}(z, {\bm \ell}).
\end{equation}
Furthermore, the restriction \eqref{eq:X-Y-constraint} and the fact that the outer sum in \eqref{eq:det-decompose-1} is over $X, Y \subset [N]$ implies that the summand in \eqref{eq:det-decompose-1} vanishes unless $\cX_k \in \gL_{{\bm \ell},k}$, where
\begin{equation}\label{eq:gL-ell-k}
\gL_{{\bm \ell}, k} := \{\cX_k \in L_{{\bm \ell}, k}: x_{1,k+j}=N+j; \ j  \in [N_-] \quad \text{ and } \quad x_{\wt N+1,j}=j; \ j \in [N_-]\}.
\end{equation}
Thus, setting
\begin{equation}\label{eq:mathbb-X}
 \mathbb{X}:=\mathbb{X}(X_1):= X_1 \cap [N] \qquad \text{ and } \qquad \mathbb{Y}:= \mathbb{Y}(X_{\wt N+1}) := (X_{\wt N+1} - N_-) \cap [N],
\end{equation}
and
\begin{equation}\label{eq:hat-Q-ell-k}
Q_{{\bm \ell}, k}:=\sum_{ \cX_k \in \gL_{{\bm \ell},k}}  (-1)^{\sgn(\sigma_{\mathbb{X}}) \sgn(\sigma_{\mathbb{Y}})}\det(E_N^t[\mathbb{X}; \mathbb{Y}]),
\end{equation}
we deduce 
\begin{equation}\label{eq:P-k-decompose-0}
{\det}_k(z)= \updelta_N^{k} a_{N_+}^{-k}  \sum_{{\bm \ell}: \gL_{{\bm \ell}, k} \neq \emptyset}  {\bm \eta}(z, {\bm \ell}) \cdot Q_{{\bm \ell}, k}. 
\end{equation}
Having obtained the above usable decomposition of ${\det}_k(\cdot)$ the strategy would be to identify an index $k(z) \in [N]$ such that ${\det}_{k(z)}(z)$ is the dominant term in the expansion of $\det(T_N({\bm a}_z)+ \updelta_N E_N)$, i.e.~${\det}_k(z) \ll {\det}_{k(z)}(z)$ for $k \ne k(z)$. This strategy will be implemented in two steps. In the first step, using a second moment bound one identifies an upper bound on the negligible terms. This bound needs to be tight upto some small polynomial in $N$ factors in the case of sub exponential decay of $\updelta_N$, while in the case of exponential decay although one has more room, as we will see below for an $O(1)$ many choices of $k$ the correct exponential rate of decay of magnitude of ${\det}_k$ compared to that of the dominant term needs to be identified. 
In the second step one uses the anti-concentration property of the entries of $E_N$, as stated in Assumption \ref{ass:anti-conc}, to obtain a lower bound on the dominant term. 

\vskip5pt
\noindent
{\bf Step I (Upper bound):} 
For ease in writing we will assume that the entries $E_N$ have unit variances. The reader may note that the same proof works under the uniformly bounded variance assumption. As the entries of $E_N$ are independent \corAB{with zero mean it} is straightforward to see that
\[
\E \left[\det(E_N[\mathbb{X}_*; \mathbb{Y}_*]) \cdot \overline{\det(E_N[\mathbb{X}'; \mathbb{Y}'] )}\right] = \left\{\begin{array}{ll}
 k! & \mbox{if } \mathbb{X}_*= \mathbb{X}' \text{ and } \mathbb{Y}_*=\mathbb{Y}',\\
0 & \mbox{ otherwise},
\end{array}
\right.
\]
for any collection of subsets $\mathbb{X}_*, \mathbb{Y}_*, \mathbb{X}', \mathbb{Y}' \subset [N]$, each of cardinality $k$. Hence, we deduce that 
\begin{align}\label{eq:var-bound-multi}
\E[|Q_{{\bm \ell}, k}|^2]=\Var (Q_{{\bm \ell}, k})  & \, = k! \cdot \mathfrak{N}_{{\bm \ell}, k},
\end{align}
where
\begin{equation}\label{eq:N-ell-k}
\mathfrak{N}_{{\bm \ell}, k}:=\left| \left\{\cX_k =(X_1, X_2,\ldots, X_{\wt N+1}), \cX_k'=(X_1', X_2', \ldots, X_{\wt N+1}') \in \gL_{{\bm \ell}, k}: X_1=X_1', X_{\wt N+1}=X_{\wt N+1}' \right\}\right|.
\end{equation}
Using the above notation and applying Cauchy-Schwarz inequality we find
\beq\label{eq:det-k-second-moment}
\E|{\det}_k(z)|^2 \le \left|\f{\updelta_N}{a_{N_+}}\right|^{2k} \left(\sum_{{\bm \ell}: \gL_{{\bm \ell}, k} \ne \emptyset} |{\bm \eta}(z, {\bm \ell})| \cdot (\E |Q_{{\bm \ell}, k}|^2)^{1/2}\right)^2 \le \left|\f{\updelta_N}{a_{N_+}}\right|^{2k} \left(\sum_{{\bm \ell}: \gL_{{\bm \ell}, k} \ne \emptyset} |{\bm \eta}(z, {\bm \ell})| \cdot (k! \gN_{{\bm \ell}, k})^{1/2} \right)^2. 
\eeq
To bound the \abbr{RHS} of \eqref{eq:det-k-second-moment} we again need two different approaches. For $k \ge |d|$, where $d = d(z)={\rm wind}({\bm a}, z)$, we use the following bound on $\gN_{{\bm \ell}, k}$. Recall that for $z \in \C$ the notation $m_+=m_+(z)$ denotes the number of roots of the polynomial $\zeta \mapsto \zeta^{N_-} {\bm a}_z(\zeta)$ that are outside $\ol{\D(0,1)}$ .

%\red{Since $x_{i+1,j} \le x_{i,j}$ for any $i=1,2,\ldots, d$, and $j=1, 2, \ldots, k+d_2$, we note from \eqref{eq:L-ell-k} that
%\begin{equation}\label{eq:hat-ell-lbd}
%\hat \ell_i \ge k+d_2, \qquad \text{ for } i=1,2,\ldots, d_0. 
%\end{equation}
%}

\begin{lemma}[{\cite[Lemma 2.5]{BZ}}]\label{lem:combinatorial-1}
Let $z \in \cS_d$ for some $d \in \Z$. Fix $k \in \N$ such that $|d| \le k \le N$, and ${\bm \ell}=(\ell_1,\ell_2,\ldots,\ell_d)$ such that $0 \le \ell_i \le \wh N_-$ for all $i \in [\wt N]$.
%, where $d_1,d_2$ and $d$ are as in Lemma \ref{lem:rouche-multi-gr-d_0}.
Then %For  $|\gd| \le k \le N$ and  ${\mathfrak N}_{{\bm \ell}, k}$  as in \eqref{eq:N-ell-k}, we have
\begin{align}\label{eq:N-l-k-bd}
\mathfrak{N}_{{\bm \ell}, k} & \, \le    \binom{\wh N_-}{k-|d|} \cdot \prod_{i=1}^{m_+} \binom{\hat \ell_i-1}{k+N_- -1}^2 \cdot \prod_{i=m_++1}^{\wt N} \binom{\hat \ell_i+k+N_-}{k+N_-}^2,
\end{align}
where
\begin{equation}\label{eq:hat-ell}
\hat \ell_i := \left\{\begin{array}{ll} \ell_i & \mbox{ if } i > m_+,\\
\wh N_-- \ell_i & \mbox{ if } i \le m_+.
\end{array}
\right.
\end{equation}
\end{lemma}
%The proof of Lemma \ref{lem:combinatorial-1} is postponed to later in the section. 
%We are now almost ready to prove Lemma \ref{lem:rouche-multi-gr-d_0}.
 
 To treat the case $k < |d|$ we will show that the summand in the \abbr{RHS} of \eqref{eq:det-k-second-moment}  is either zero or exponentially small compared the dominant term, which in turn will yield that ${\det}_k(z)$, for such a $k$, is exponentially small compared to the dominant term.
\begin{lemma}\label{lem:combin-2}
Let $z \in \cS_d$ for some $d$ such that $|d| \ge 2$. Fix any $k \in \N$ such that $k < |d|$. Fix ${\bm \ell}=(\ell_1, \ell_2, \ldots, \ell_{\wt N})$ with $ 0 \le \ell_i \le \wh N_-$ for $i \in [\wt N]$. 
\begin{enumerate}
\item[(i)] Let $d >0$. Then $\gL_{{\bm \ell}, k} \ne \emptyset$ implies $\sum_{i=m_++1}^{\wt N} \ell_i \ge (d-k) \wh N_- - \wt N^2$. 
\item[(ii)] Let $d <0$. Then  $\gL_{{\bm \ell}, k} \ne \emptyset$ implies $\sum_{i=1}^{m_+} \ell_i \le (m_++k+d) \wh N_- +\wt N^2$.
\end{enumerate}
\end{lemma} 

The proof of Lemma \ref{lem:combin-2} is a refinement of that of \cite[Lemma 2.9]{BZ}.
\begin{proof}[Proof of Lemma \ref{lem:combin-2}]
Consider first
the case $d = N_+ - m_+>0$. For $k < d$, $s \in [d-k]$, and any $\cX_k \in \gL_{{\bm \ell}, k} \subset L_{{\bm \ell}, k}$ the indices $\{x_{m_++s+i+1,j}, i, j \in [k+N_-]\}$ are well defined.
On the other hand observe that
\begin{multline}\label{eq:d-pos-telescope}
x_{m_++s+1,1} + \sum_{j=1}^{k+N_--1} (x_{m_++s+j+1,j+1} - x_{m_++s+j,j}) + (\wh N_- - x_{m_++s+k+N_-, k+N_-} ) = \wh N_-. 
\end{multline}
Adding both sides of \eqref{eq:d-pos-telescope} for $s=1,2,\ldots, d-k$, we see that the \abbr{RHS} is $(d-k) \wh N_-$, while recalling the definition of $L_{{\bm \ell}, k}$, noting that each of the summand in \eqref{eq:d-pos-telescope} is nonnegative, and rearranging the sum we see that the sum of the \abbr{LHS} is bounded above by $\sum_{i=m_++1}^{\wt N} (\ell_i+k+N_-)$. The desired bound for $d >0$ is now immediate. 

Next we consider the case $d <0$. Recalling \eqref{eq:hat-ell} and using the observation
\[
x_{i+1,1}+  \sum_{j=2}^{k+N_-}(x_{i+1,j}-x_{i,j-1}) +  \sum_{j=1}^{k+N_-}(x_{i,j}-x_{i+1,j}) + (\wh N_-- x_{i,k+N_-}) = \wh N,
 \]
we see that
\beq\label{eq:d-pos-telescope1}
\sum_{i=1}^{m_+} \wh \ell_i = \sum_{i=1}^{m_+} \sum_{j=1}^{k+N_-}(x_{i,j}-x_{i+1,j}+1) \ge \sum_{j=k+1}^{|d|}(x_{1,j}-x_{m_++1,j}+1).
\eeq
Now, for any $\cX_k \subset \gL_{{\bm \ell}, k}$ we have $x_{1,k+j} = N+j$ for $j \in [N_-]$. Further $x_{\wt N+1, \ell} = \ell$ for $\ell \in [N_-]$. Since $x_{i+1,\ell} \le x_{i,\ell} < x_{i+1, \ell+1}$ for any $\ell \in [k+N_- -1]$, and $\{x_{i, \ell}\}$ are integers, using induction, we obtain that $x_{m_++1, \ell}= \ell$ for any $\ell \in [|d|]$. The last two observations together implies that $x_{1,j} - x_{m_++1, j} = N$ for $j=k+1, \ldots, |d|$. Plugging this estimate in the \abbr{RHS} of \eqref{eq:d-pos-telescope1} the desired bound is now immediate. 
\end{proof}

Building on Lemmas \ref{lem:combinatorial-1} and \ref{lem:combin-2} we have the following upper bound. 

\begin{lemma}\label{lem:ubound-detk}
Let ${\bm a}$ be as in \eqref{eq:laurent-poly} and $\updelta_N=r^N$ for some $r \in (0,1)$. Assume that the entries of $E_N$ are centered and have unit variance. 
\begin{enumerate}
\item[(i)] Let $z \in \cS_d$ for some $d \in \Z$. Then, there exists some constant $\bar C(z) < \infty$ so that for any $k \in \N$ such that $|d| \le k \le N$, any $\vep >0$, and all large $N$,
\[
\E|{\det}_k(z)|^2 \le \left( r^{2|d|N} |a_{N_+}|^{2N}\prod_{i=1}^{m_+}|\zeta_i(z)|^{2N}  \right) \cdot \bar C(z) \cdot (r(1+\vep))^{2(k-|d|)N}.
\]
\item[(ii)] \corAB{Recall the definition of $\wh \Gamma({\bm a})$ from Definition \ref{dfn:bad-points}. Let $z \in \cS_d\setminus \wh \Gamma({\bm a})$ for some $d \in \Z$ such that $d \ne 0$. Then,} for any $k \in \N$ such that $k < |d|$ and all large $N$,
\beq\label{eq:detk-small-k}
\E|{\det}_k(z)|^2 \leq \left( r^{2kN} |a_{N_+}|^{2N}\prod_{i=1}^{N_+-k {\bf 1}(d >0) +k {\bf 1}(d <0)}|\zeta_i(z)|^{2N}  \right) \cdot \bar C'(z) \cdot N^{2k+4\wt N}.
\eeq
\end{enumerate}
\end{lemma}
Notice the difference in the upper bound in \eqref{eq:detk-small-k} for $d >0$ and $d < 0$. 

\begin{proof}[Proof of Lemma \ref{lem:ubound-detk}]
As $z \notin {\bm a}(\mathbb{S}^1)$ there exists some $\vep_0(z) > 0$ such that 
\[
 |\zeta_{m_+}(z)|^{-1} \wedge |\zeta_{m_++1}(z)| \le 1 -\vep_0(z). 
\]
Recalling \eqref{eq:prod-decomp}, and applying the combinatorial identity
\[
\sum_{s \geq m} \binom{s-1}{m-1} \uplambda^{s-m} = (1-\uplambda)^{-m}, \qquad \text{ for } \uplambda \in (0,1) \text{ and } m \in \N,
\]
together with \eqref{eq:det-k-second-moment} and Lemma \ref{lem:combinatorial-1} we obtain
\begin{align}
&\left(\E|{\det}_k(z)|^2\right)^{1/2}  \cdot \prod_{i=1}^{m_+}|\zeta_i(z)|^{-\wh N_-} \cdot |a_{N_+}|^{-N} \\
 \leq & \, r^{kN}  |a_{N_+}|^{-k} \wh N^{\f{k-|d|}{2}}_- k^{\f{|d|}{2}} \sum_{{\bm \ell}: \gL_{{\bm \ell}, k} \ne \emptyset}  \prod_{j=1}^{m_+} |\zeta_i(z)|^{-\wh \ell_i} \prod_{j=m_++1}^{\wt N} |\zeta_j(z)|^{\wh \ell_i}     \cdot \prod_{i=1}^{m_+} \binom{\hat \ell_i-1}{k+N_- -1} \prod_{i=m_++1}^{\wt N} \binom{\hat \ell_i+k+N_-}{k+N_-}\notag\\
 \leq &  \, r^{kN}   N^{\f{k-|d|}{2}} k^{\f{|d|}{2}} C(z)^k, \notag
\end{align}
for some $C(z) < \infty$. This completes the proof of part (i). 

Turning to prove part (ii) we first consider the case $d >0$. Recalling \eqref{eq:prod-decomp}, as $\ell_i \le \wh N_-$ and $|\zeta_{m_++1}(z)| < 1$, applying Lemma \ref{lem:combin-2}(i) we note that for any ${\bm \ell}$ such that $\gL_{{\bm \ell}, k} \neq \emptyset$, as $z \notin \wh \Gamma({\bm a})$, we have
\beq\label{eq:bm-eta-bd}
|{\bm \eta}(z, {\bm \ell})| \le |a_{N_+}|^N \prod_{i=1}^{m_+} |\zeta_i(z)|^{\wh N_-}  \cdot \prod_{j=m_++1}^{\wt N} |\zeta_j(z)|^{\ell_j} \le  C'(z) \cdot |a_{N_+}|^N \prod_{i=1}^{N_+-k} |\zeta_i(z)|^{\wh N_-},  
\eeq
for some constant $C'(z) < \infty$. The desired bound for $d >0$ now follows upon using the trivial bound $\gN_{{\bm \ell}, k} \le N^{2k}$ and noting that the number of possible choices for $\ell_i$ is at most $\wh N_-+1$ for each $i \in [\wt N]$, and then plugging them in \eqref{eq:det-k-second-moment}. For $d < 0$ applying Lemma \ref{lem:combin-2}(ii) for any ${\bm \ell}$ such that $\gL_{{\bm \ell}, k} \neq \emptyset$, and arguing analogously as in \eqref{eq:bm-eta-bd} we have
\[
|{\bm \eta}(z, {\bm \ell})| \le |a_{N_+}|^N \prod_{i=1}^{m_+} |\zeta_i(z)|^{\ell_i} \le  C'(z) \cdot |a_{N_+}|^N \prod_{i=1}^{N_++k} |\zeta_i(z)|^{\wh N_-}.
\]
Now proceeding similarly as in the proof for the case $d >0$, the proof for $d <0$ also completes. 
\end{proof}

\vskip5pt
\noindent
{\bf Step II (Lower bound):} The lower bound relies on the following anti-concentration bound. 

\begin{lemma}[{\cite[Proposition 4.1]{BZ}}]\label{prop:anti-conc}
Fix $k, n \in \N$ and let $\{U_i\}_{i=1}^n$ be a sequence of
independent (possibly complex valued) random variables
such that
\(
\max_{i=1}^n \gL(U_i, \vep) \le C \vep^{\upeta},
\)
where $C, \vep$, and $\upeta$ are as in Assumption \ref{ass:anti-conc} (recall \eqref{eq:gL}). Let $P_k(U_1,U_2,\ldots,U_n)$ be a 
homogenous polynomial of degree $k$ such that the degree of 
each variable is at most one. That is,
\[
P_k(U_1,U_2,\ldots,U_n) := \sum_{\cI \in \binom{[n]}{k}} a(\cI) \prod_{i \in \cI} U_{i},
\]
for some collection of complex valued coefficients $\{a(\cI); \, \cI \in \binom{[n]}{k}\}$. Assume that there exists an 
$\cI_0$ such that $|a(\cI_0)| \ge c_\star$ for some absolute constant $c_\star>0$. Then for $\vep \in (0, {e^{-1}}]$ sufficiently small and some large $\bar C < \infty$ we have
\beq\label{eq:P-k-anti-conc}
\P\left( |P_k(U_1,U_2,\ldots,U_n)| \le \vep\right) \le \bar C (c_\star \wedge 1)^{-\upeta}\vep^\upeta \left(\log\left(\frac{1}{\vep}\right)\right)^{k-1}.
\eeq
\end{lemma}

At a high level, the proof of Lemma \ref{prop:anti-conc} relies on the observation that if $U$ has a good anti-concentration property, characterized by an upper bound on its L\'evy concentration function, then one can derive anti-concentration property of the random variable $AU+V$, where $\{A,V\}$ is jointly independent of $U$, on the event that $A$ is bounded below. Since the degree of each variable is at most one in \corAB{$P_k(\cdot)$ and at least one coefficient is bounded below the} idea can propagated to obtain the bound \eqref{eq:P-k-anti-conc}. 

\corAB{Building on Lemmas \ref{lem:widom} and \ref{prop:anti-conc} we derive the following lower bound on ${\det}_k$ for any $k \in \N$ with $k < |d|$.}  

\begin{lemma}\label{lem:lower-bd-anti-conc}
Let $d \in \Z$ and $z\in \cS_d\setminus \wh \Gamma({\bm a})$. Then for any $\delta >0$ and $k \in \N$ so that $k \le d$ we have
\[
\P\left(|{\det}_k(z)| \le N^{-\delta} \cdot r^{kN} |a_{N_+}|^{N}\prod_{i=1}^{N_+-k {\bf 1}(d >0) +k {\bf 1}(d <0)}|\zeta_i(z)|^{N}  \right) \ll 1.
\]
\end{lemma}\label{lem:lbound-detk}

\begin{proof}
\corAB{The case $d=0$ is immediate from Lemma \ref{lem:widom}.} Consider the case $d >0$. Set $\X_k:=[N]\setminus [k]$ and $\Y_k:=[k]$. Let ${\bm a}_{k,z}(\zeta):=\zeta^{-k} {\bm a}_z(\zeta)$. Then $T_N({\bm a}_z)[\Y^c_k; \X^c_k]= T_{N-k}({\bm a}_{k,z})$. As $z \notin \wh \Gamma({\bm a})$, applying Lemma \ref{lem:widom} for the symbol ${\bm a}_{k,z}(\zeta)$ we get
\[
|\det(T_N({\bm a}_z)^t[\X^c_k;\Y^c_k])| \gtrsim  |a_{N_+}|^{N-k}\prod_{i=1}^{N_+-k}|\zeta_i(z)|^{N-k}.
\]
The desired bound now follows upon recalling \eqref{eq:det-k} and applying Lemma \ref{prop:anti-conc}. 

For $d<0$ setting $\Y_k:=[N]\setminus [k]$ and $\X_k:=[k]$, and repeating the same argument as above the proof completes.
\end{proof}

Combining Lemmas \ref{lem:ubound-detk} and \ref{lem:lbound-detk} we now complete the proof of Theorem \ref{thm:lsd-exp}. 

\begin{proof}[Proof of Theorem \ref{thm:lsd-exp}]
As the entries of $E_N$ have uniformly bounded second moments, by \eqref{eq:T-op-norm} and Weyl's inequality, it follows that \eqref{eq:bdd-in-prob} holds for $\mu_N= L_{T_N({\bm a})+r^N E_N}$. Hence, by Lemma \ref{lem:cL-converges}(ii) it suffices to show that $\cL_{L_{T_N({\bm a})+r^N E_N}}(z) \to \cL_{{\bm a}, r}(z)$, in probability, as $N \to \infty$, for Lebesgue a.e.~$z \in \C$. 

To this end, fix any $z \in \cS_0\setminus \wh \Gamma({\bm a})$. Fix $\vep >0$ such that $r_\vep:=r(1+\vep) < 1$. By Lemma \ref{lem:ubound-detk}(i) it follows that 
\beq
\P\left(|{\det}_k(z)| \ge  r_\vep^{kN/3}\cdot \exp(N \cL_{{\bm a}, r}(z)) \right) \leq r_\vep^{kN}, \quad \text{ for } k \in [N] \text{ and all large } N. 
\eeq
Therefore, by a union bound $\P(\Omega_N^c) \le r_\vep^{N/2}$, for all large $N$, where
\[
\Omega_N:= \left\{{\sum_{k=1}^N|{\det}_k(z)|}\le r_\vep^{N/6} \cdot \exp(N \cL_{{\bm a}, r}(z))\right\}.
\]
On the other hand, by Lemma \ref{lem:widom}, $|{\det}_0(z)| \gtrsim \exp(N\cL_{{\bm a}, r}(z))$. Therefore, as $\P(\Omega_N^c) \ll 1$, using both the lower and upper bounds from Lemma \ref{lem:widom} we deduce that 
\beq\label{eq:conv-in-prob-1}
 \cL_{L_{T_N({\bm a})+r^N E_N}}(z)\to \cL_{{\bm a}, r}(z), \quad \text{ as } N \to \infty, \quad \text{ in probability}, \quad \text{ for all } z \in \cS_0\setminus \wh \Gamma({\bm a}). 
\eeq
Next assume that $z \in \cS_{d,\ell}^r\setminus \wh \Gamma({\bm a})$ for some $d >0$ and $\ell \in \{0,1,2,\ldots, m_-\}$. Fix any $\delta >0$. Define
\[
\ol \Omega_N:=\left\{|{\det}_{(d-\ell)_+}(z)| \geq  \exp(N \cL_{{\bm a}, r}(z)) \cdot N^{-\delta} \right\}, \quad \Omega_N':= \left\{|{\det}_{(d-\ell)_+}(z)| \leq  \exp(N \cL_{{\bm a}, r}(z)) \cdot N^{(d-\ell)_++3\wt N} \right\},
\]
and
\[
\wt \Omega_N:=\left\{ \sum_{k \in[N]\cup\{0\}\setminus \{ (d-\ell)_+\}} |{\det}_k(z)| \le (1-\wt \vep(z))^N \cdot \exp(N \cL_{{\bm a}, r}(z)) \right\},
\]
for some $\wt \vep(z) >0$ to be determined below. For $\ell \ge d$, using Lemma \ref{lem:widom} again we find $\ol{\Omega}_N$ holds for all large $N$. For $\ell <d$, by Lemma \ref{lem:lower-bd-anti-conc}, it follows that $\P(\ol\Omega_N^c) \ll 1$.  Further, by Lemma \ref{lem:ubound-detk}(ii), for $\ell < d$ and Lemma \ref{lem:widom}, for $\ell \ge d$, we have $\P(\Omega_N'^c) \ll 1$. 

We now turn to bound $\P(\wt \Omega_N^c)$. As $z \in \cS_{d,\ell}^r$, from the definition of $\cS_{d,\ell}^r$ it follows that there is some $\vep_0(z) >0$ such that
\beq\label{eq:ratio-exp}
\prod_{j=1}^{m_++d \wedge \ell} |\zeta_j(z)|^N \cdot r^{N(d-\ell)_+} \ge (1+2\vep_0(z))^N \cdot \max_{k \in \{0,1,2,\ldots, d\}\setminus \{(d-\ell)_+\}} \prod_{j=1}^{N_+-k} |\zeta_j(z)|^N \cdot r^{kN}.
\eeq
Therefore, by Lemma \ref{lem:ubound-detk}(ii), for any $k \in [d-1]\setminus\{(d-\ell)_+\}$, 
\beq\label{eq:prob-bd1}
\P\left(|{\det}_k(z)| \geq (1-\vep_0(z))^{N/3} \cdot \exp(N\cL_{{\bm a},r}(z))\right) \le (1-\vep_0(z))^{N}, \qquad \text{ for all large } N. 
\eeq
We now need to consider two cases $\ell =0$ and $\ell \ne 0$ separately. Focusing on the case $\ell=0$ first, by Lemma \ref{lem:ubound-detk}(i), for any $k \in \N$ such that $d < k \le N$ we find
\beq\label{eq:prob-bd2}
\P\left(|{\det}_k(z)| \geq r_\vep^{(k-d)N/3} \cdot \exp(N\cL_{{\bm a},r}(z))\right) \le r_\vep^{(k-d)N}, \qquad \text{ for all large } N. 
\eeq
Next consider the case $\ell \ne 0$. In this case, using \eqref{eq:ratio-exp} and applying Lemma \ref{lem:ubound-detk}(i) again, we find that for any $k \in \N$ such that $d \le k \le N$
\beq\label{eq:prob-bd3}
\P\left(|{\det}_k(z)| \geq (1-\vep_0(z))^{N/3} r_\vep^{(k-d)N/3} \cdot \exp(N\cL_{{\bm a},r}(z))\right) \le r_\vep^{(k-d)N} (1-\vep_0(z))^N, \qquad \text{ for all large } N. 
\eeq
Now combining \eqref{eq:prob-bd1}-\eqref{eq:prob-bd3} (and the bound \eqref{eq:ratio-exp} when $\ell < d$) it follows that there indeed exists some $\wt \vep(z) >0$ such that $\P(\wt \Omega_N^c) \ll 1$. 

As $\P(\ol \Omega_N \cap \wt \Omega_N\cap \Omega_N') = 1-o(1)$, it is immediate to note that for any $d >0$ and any $\ell=0,1,2,\ldots, m_-$,
\beq\label{eq:conv-in-prob-2}
 \cL_{L_{T_N({\bm a})+r^N E_N}}(z)\to \cL_{{\bm a}, r}(z), \quad \text{ as } N \to \infty, \quad \text{ in probability}, \quad \text{ for all } z \in \cS_{d,\ell}^r\setminus \wh \Gamma({\bm a}). 
\eeq
Arguing similarly as above it also follows that \eqref{eq:conv-in-prob-2} continues to hold for any $d <0$ and $\ell=0,1,2,\ldots, m_+$ \corAB{(details omitted)}. This observation together with \eqref{eq:conv-in-prob-1} now completes the proof. 
\end{proof}

We end the section with the proof of Theorem \ref{thm:lsd-cont}. 
\begin{proof}[Proof of Theorem \ref{thm:lsd-cont}]
We only prove part (ii). The proof of part (i) being identical is omitted. 

The Avram-Parter theorem (see \cite[Theorem 5.17]{BS99}) together with the observation \eqref{eq:T-op-norm} imply that the empirical measure of the singular values of $T_N({\bm a}_z)$ converges weakly $\nu^{{\bm a}, z}$, the law of $|{\bm a}({\bm u}) -z|$, as $N \to \infty$, for all $z \in \C$. 
As the entries of $E_N$ have $\log$-concave densities it follows that they have finite moments, and so $E_N$ satisfies Assumption \ref{ass:hs}. Since $\updelta_N \ll N^{-1/2}$, the Hoffman-Wielandt inequality (see \cite[Lemma 2.1.19]{AGZ}) yields that $\nu_N^{{\bm a}, z}$, the empirical measure of the singular values of $T_N({\bm a}_z)+\updelta_N E_N$ converges weakly, in probability, to $\nu^{{\bm a}, z}$, as $N \to \infty$, for all $z \in \C$. This, in particular, implies that for any $z \in \C$, $\vep >0$, and $R< \infty$,
\beq\label{eq:log-conv-1}
\int_\vep^R \log(x) d\nu_N^{{\bm a}, z}(dx) \to \E\left[\log|{\bm a}({\bm u})-z| \cdot{\bf 1}(|{\bm a}(u)-z| \in [\vep, R])\right], \quad \text{ as } N \to \infty, \quad \text{ in probability}.  
\eeq
Moreover, by \eqref{eq:T-op-norm}, for any $z \in \C$,
\beq\label{eq:bdd-sec-mom}
\limsup_{N \to \infty} N^{-1}\E\left[\|T_N({\bm a}_z)+\updelta_N E_N\|^2_{{\rm HS}}\right] \lesssim \corAB{\|{\bm a}_z\|_{\infty, \mathbb{S}^1}^2} +   \limsup_{N \to \infty} \updelta_N^2 N^{-1}\E\left[\|E_N\|^2_{{\rm HS}}\right] < \infty. 
\eeq
This further implies that given any $\delta_0>0$ and $z \in \C$ there exists some $R(z,\delta_0) < \infty$ such that 
\beq\label{eq:log-conv-2}
\lim_{N \to \infty} \P\left(\left|\int_{R(z,\delta_0)}^\infty \log(x) d\nu_N^{{\bm a}, z}(dx) - \E\left[\log|{\bm a}({\bm u})-z| \cdot {\bf 1}\left(|{\bm a}({\bm u})-z| \geq R(z,\delta_0)\right)\right]\right| \ge \delta_0 \right)=0.
\eeq
Now fix any $z \notin {\bm a}(\mathbb{S}^1)$. As ${\bm a}(\mathbb{S}^1)$ is a closed set $\vep_1(z):= {\rm dist}(z, {\bm a}(\mathbb{S}^1)) >0$. Thus, for any $\vep \in (0,\vep_1(z))$,
\beq\label{eq:log-conv-3}
\E\left[\log|{\bm a}({\bm u})-z| \cdot{\bf 1}(|{\bm a}(u)-z| < \vep)\right] =0. 
\eeq  
Further let $k=k(z):={\rm wind}({\bm a}, z) \in \Z$. By \cite[Theorem 1.17]{BS99} and \cite[Proposition 4.7]{BS99} we have $\liminf_{N \to \infty} s_{|k|+1}(T_N({\bm a}_z)) = \vep_2(z) >0$. Hence, applying Weyl's inequality for singular values,
\[
\nu_N^{{\bm a}, z}((0, \vep_2(z)/2)) \le |k|/N, \qquad \text{ for all large } N, 
\]
on the event $\Omega_N:=\{\|E_N\| \le N^{1/4}\updelta_N^{-1/2} \}$. Using that entries of $E_N$ finite moments of all orders, by \cite{L04} we find that $\P(\Omega_N^c) \ll 1$. On the other hand, by \cite[Corollary 1.4]{T20} we further have $\P(\wt \Omega_N^c) \ll 1$, where
\[
\wt \Omega_N:= \left\{s_{\min}(T_N({\bm a}_z) + \updelta_N E_N) \ge \updelta_N/N\right\}.
\]
To complete the proof, we note that on the event $\Omega_N \cap \wt \Omega_N$, setting $\vep_\star(z):=  \min\{\vep_1(z), \vep_2(z)/2\}$ we have
\beq\label{eq:log-conv-4}
\left| \int_0^{\vep_\star(z)} \log (x) d\nu_N^{{\bm a}, z}(dx) \right| \le \corAB{(|k|/N) \cdot \log(N/\updelta_N)} \ll 1. 
\eeq
Combining \eqref{eq:log-conv-1} and \eqref{eq:log-conv-2}-\eqref{eq:log-conv-4}, we find that $\cL_{L_{T_N({\bm a})+\updelta_N E_N}}(z) \to \cL_{{\bm a}({\bm u})}(z)$, as $N \to \infty$, in probability, for any $z \in \C\setminus {\bm a}(\mathbb{S}^1)$. Further, by \eqref{eq:bdd-sec-mom} the assumption \eqref{eq:bdd-in-prob} holds for $\mu_N=L_{T_N({\bm a})+\updelta_N E_N}$. Consequently, recalling that ${\bm a}(\mathbb{S}^1)$ has a zero Lebesgue measure and applying Lemma \ref{lem:cL-converges}(ii) the proof completes.   
\end{proof}

%Define 
%\[
%{\bm a}^{(M)}(\zeta):= \sum_{k=-M}^M a^{(M)}_k  \zeta^k, \,  \zeta \in \mathbb{S}^1, \qquad \text{ where } \qquad a^{(M)}_k:= \left(1-\f{|k|}{M+1}\right)_+ a_k.
%\]
%By Fej\'er's theorem, $\|{\bm a}^{(M)} - {\bm a}\|_{\infty, \mathbb{S}^1} \to 0$ as $M \to \infty$. Therefore, by Weyl's inequality and using the fact $\|T_N({\bm b})\| \le \|T({\bm b})\| = \|{\bm b}\|_{\infty, \mathbb{S}^1}$ for any ${\bm b} \in L^\infty(\mathbb{S}^1)$ we find that
%\[
%\lim_{M \to \infty} \sup_{N \in \N} \max_{i \in [N]}|s_i(T_N({\bm a}_z)) - s_i(T_N({\bm a}^{(M)}_z))| \leq \lim_{M \to \infty} \sup_{N \in \N} \|T_N({\bm a}-{\bm a}^{(M)})\| =0. 
%\]
%Next, for any $M \in \N$ we have $T_N({\bm a}^{(M)})^t= \sum_{i=0}^{M} a_i^{(M)} J_N^i + \sum_{i=1}^{N_-} a_{-i}^{(M)}(J_N^*)^i$. As $J_N$ converges to ${\bm u}$ in $*$-moments, it is straightforward to note that the empirical measure of the singular values of $T_N({\bm a}_z^{(M)})$ converges weakly to the law of $|{\bm a}^{(M)}({\bm u})-z|$, for any $z \in \C$ and $M \in \N$. This last observation together with \eqref{} further implies that the empirical measure of the singular values of $T_N({\bm a}_z)$ converges weakly $\nu^{{\bm a}, z}$, the law of $|{\bm a}({\bm u}) -z|$, as $N \to \infty$, for all $z \in \C$. 

\subsection{Proofs of properties of outliers} Proofs of Theorems \ref{thm:no-outlier} and \ref{thm:outlier} use the idea of determinant expansion. For example, in the case of microscopic perturbation, following the same approach described above it was shown that $|{\det}_0(z)| \gg |\sum_{k \in [N]} {\det}_k(z)|$ uniformly for all $z \notin \cS_0^c +\D(0,\vep)$ and any $\vep>0$. This observation together \corAB{with a uniform} lower bound on $|{\det}_0(z)|$ combined with Rouch\'e's theorem yield Theorem \ref{thm:no-outlier}(i). To identify the limiting distribution of outliers in $\cS_d$ the role of ${\det}_0(\cdot)$ is now taken up \corAB{by ${\det}_{|d|}(\cdot)$ with the lower bound obtained by applying Lemma \ref{prop:anti-conc}}. 

The proofs in the case of macroscopic perturbation follow a similar strategy. However, as the perturbation is macroscopic one no longer expects ${\det}_k(\cdot)$ to be dominant for a single $k=O(1)$. To tackle this issue, \cite{BCC24} uses Sylvester identity to derive that any outlier eigenvalue of $T_N({\bm a})+\upsigma N^{-1/2}E_N$ must be a root of the polynomial   $z \mapsto \det({\rm Id}_{\wt N} + Q_N R_N'(z) P_N + F_N(z))$, where $F_N(z):= Q_N(R_N(z) - R_N'(z) )P_N$, $R_N'(z):=(z - C_N)^{-1}$, $R_N(z):= (z - C_N - \upsigma N^{-1/2} E_N)^{-1}$, and $C_N=C_N({\bm a})$ is an  appropriately chosen circulant matrix so that $T_N({\bm a}) = C_N({\bm a}) - P_N Q_N$ for some matrices $P_N$ and $Q_N$ of dimensions $N \times \wt N$ and $\wt N \times N$, respectively. With this reduction, they now apply \eqref{eq:det_decomposition} with $A_N={\rm Id}_{\wt N} + Q_N R_N'(z) P_N$ and $B_N=F_N(z)$, and proceeds by identifying the dominant term in the expansion.

\section{Proof for the localization of eigenvectors}
The proof for the localization property of the eigenvectors corresponding to bulk eigenvalues of randomly perturbed finitely banded Toeplitz matrices uses the Grushin problem described in Section \ref{sec:grushin} and the idea of determinant expansion introduced in Section \ref{sec:det-expansion}. 

We will apply the Grushin problem for $\delta=N^{-\gamma}$, $Q=E_N$, and $A=T_N({\bm a})-z {\rm Id}_N$, for $z \in \C$. To avoid confusions, it will be imperative to make the dependence in $z$ explicit. So, we will write $A^\delta_z$ instead of $A^\delta$ (see \eqref{gpp1}) and $E_\pm(z)$ instead of $E_\pm$ (see \eqref{gp8}) etc. The starting point is the following couple of observations, which are straightforward consequences of the definition of $\cE^\delta$ (see \eqref{eq:cE-delta}): 
\begin{equation}\label{eq:algeb-id}
E^\delta(z)A^\delta_z+E_+^\delta(z)R_+(z)={\rm Id}_N \quad \text{ and } \quad E_-^\delta(z) A_z^\delta + E_{-+}^\delta(z) R_+(z) =0.
\end{equation}
It would have been ideal to use \eqref{eq:algeb-id} for $z$ an actual eigenvalue of the randomly perturbed Toeplitz matrix under consideration. As the eigenvalues are complicated functions of the entries of a matrix, using \eqref{eq:algeb-id} for an actual eigenvalue $z$ will prevent us using the joint independence structure of the entries of $E_N$. This obstacle is tackled by working on a net of the region in the complex plane where most of the eigenvalues of the perturbed matrix reside with a high probability.  However, this means that one needs various probability bounds to be strong enough so that the desired estimates hold simultaneously for {\em all} net points, which in turn necessitates Assumptions \ref{ass:anti-conc} and \ref{ass:mom}. 

For $\theta >0$, to be determined below, let $\cN_\theta$ be a net of $\Omega(\vep, C, N)$ (recall \eqref{eq:Omega-ep-C}) with mesh size $N^{-\theta}$. Given any eigenvalue $\lambda \in \Omega(\vep, C, N)$ of $T_N({\bm a})+N^{-\gamma} E_N$ we let  $z_\lambda \in \cN_\theta$ be such that $|z_\lambda - \lambda| \le N^{-\theta}$ and $v_\lambda$ be a right (normalized) eigenvector corresponding to $\lambda$. For any $M \in \N$, by \eqref{eq:algeb-id}, \eqref{gp5}, \eqref{gp8}, and \eqref{eq-march1b}, as 
 \beq\label{eq:simple1}
A_\lambda^\delta v_\lambda =0 \qquad \text{ and } \qquad A_{z}^\delta -A_\lambda^\delta = (\lambda - z){\rm Id}_N, 
\eeq
for any $z \in \C$, we see that 
\begin{eqnarray}
\sum\nolimits_{i=M+1}^N (e_i(z_\lambda)^* v_\lambda) \cdot e_i(z_\lambda)&= &  ({\rm Id}_N-E_+(z_\lambda)R_+(z_\lambda)) v_\lambda \nonumber\\
%\vspace{-20pt}
& =&({\rm Id}_N-E_+^\delta( z_\lambda)R_+(z_\lambda)) v_\lambda - E(z_\lambda)({\rm Id}_N+\delta QE(z_\lambda))^{-1} \delta QE_+(z_\lambda) R_+(z_\lambda)v_\lambda \nonumber\\
%& = E^\delta(z) (z-\hat z) x+
&=&E^\delta(z_\lambda)(\lambda- z_\lambda) v_\lambda - E(z_\lambda) ({\rm Id}_N+\delta QE(z_\lambda))^{-1} \delta QE_+(z_\lambda) R_+(z_\lambda)v_\lambda. \label{eq:ef-to-pm-1}
\end{eqnarray}
If the \abbr{RHS} of \eqref{eq:ef-to-pm-1} is $o(1)$ then, as $\{e_i(z_\lambda)\}_{i \in [N]}$ forms an orthonormal basis of $\C^N$, one obtains that $\|v_\lambda - w_{z_\lambda}(v_\lambda)\|=o(1)$, where
\[
w_z(v):= \sum_{i=1}^M (e_i(z)^* v) \cdot e_i(z), \qquad z \in \C \text{ and } v \in \C^N.
\]
Hence, the localization properties of the eigenvector $v_\lambda$ will be determined by those of $w_{z_\lambda}(v_\lambda)$. 
To employ this \corAB{strategy one needs the} following intermediate estimates:
\begin{enumerate}
\item[(a)] With probability $1-o(1)$ all but an arbitrary small fraction of the eigenvalues of $T_N({\bm a})+N^{-\gamma} E_N$ fall inside $\Omega(\vep, C, N)$, for all large $N$, and some $0 < \vep, C < \infty$.
\item[(b)] There exist $\alpha_1, \alpha_2, \beta \in (0,\infty)$, so that for any $k \in \N$ with probability $1-o(1)$, simultaneously for all $z \in \cN_\theta$
\beq\label{eq:norm-bds}
\|E^\delta(z)\| \lesssim N^{\alpha_1}, \qquad \|({\rm Id}_N+\delta Q E(z))^{-1}\| \lesssim N^{\alpha_2}, \text{ and } \delta^k\|(E(z)Q)^k E_+(z)\| \lesssim N^{-\beta k}.
\eeq
\item[(c)] For any $z \in \cN_\theta \cap \cS_d$, the singular vectors $\{e_i(z)\}_{i\in [d]}$ satisfy the localization bound of Theorem \ref{thm:eigenvec}(ii).   
\end{enumerate} 
%While carrying out steps (b) and (c) it will turn out that the correct choice of $M$ is $d=d(\lambda)={\rm wind}({\bm a}, \lambda)$. 

Notice \corAB{from \eqref{eq:ef-to-pm-1}-\eqref{eq:norm-bds} that the} parameter $\alpha_1$ determines $\theta$, while the other two bounds in \eqref{eq:norm-bds} together \corAB{with a resolvent} expansion yield the desired bound on the right most term in \eqref{eq:ef-to-pm-1}. 
Next, observe that Step (a) yields part (i) of Theorem \ref{thm:eigenvec}, while that combined with Steps (b)-(c) yield part (ii) of that theorem. Before proceeding further, let us remark that the general strategy employed in \cite{BVZ} is applicable for a broader class of matrices beyond Toeplitz matrices. However, while executing each of the individual steps the Toeplitz structure has been crucially used in \cite{BVZ}.

Carrying out Steps (a)-(c) require {\em significant} work! Let us briefly illustrate the main ideas behind these three steps. 
Step (a) is further decomposed into two sub steps. To show that the distance of {\em all} eigenvalues in the good region is at least of order $\log N/N$ from the symbol curve, \cite{BVZ} relies on the determinant expansion idea. However, it now requires a more careful analysis as one needs to consider $z$'s that are of vanishing distance from the curve. To show that the distance of most eigenvalues is at most of order $\log N/N$ from the curve \cite{BVZ} uses the determinant expansion idea yet again together with Jensen's formula to estimate the number of roots of the characteristic polynomial in the desired domain. 

Upon choosing $M=|d(z)|$, where we recall $d(z)={\rm wind}({\bm a}, z)$, and relating $T_N({\bm a}_z)$ to a circulant matrix, one finds that $t_{M+1}(z) \gtrsim \log N/N \gg t_M(z)$ (this crucially uses that any $z \in \cN_\theta$ is of distance order $\log N/N$ from ${\bm a}(\mathbb{S}^1)$) and bounds on $t_j(z)$ for $j > M$. These bounds further yield bounds on the Hilbert-Schmidt norms of $E(z)$ and related \corAB{matrices. 
Those bounds} are then used to obtain bounds on moments of $\|(E(z)Q)^k E_+(z)\|_{{\rm HS}}$ from which the rightmost bound in \eqref{eq:norm-bds} follows. The lower bound $s_{\min}({\rm Id}_N+ \delta Q E(z))$ essentially uses a change of basis argument together with standard bounds on $s_{\min}(\cdot)$ of matrices that are deterministic shifts of random matrices with entries satisfying some anti-concentration properties. Recall \eqref{eq-Edelta}. The left most bound in \eqref{eq:norm-bds} is obtained together with this bound on $s_{\min}(\cdot)$, a bound analogous to the right most bound in \eqref{eq:norm-bds}, \corAB{and a resolvent} expansion. 

\corAB{The proof of Step (c)} requires to show that certain linear combinations of {\em pure states} (constructed out of the roots of the polynomial $\zeta \mapsto \zeta^{N_-} {\bm a}_z(\zeta)$) serve as approximate singular vectors corresponding to singular values for $\{t_i(z)\}_{i=1}^M$. \corAB{The pure states can be shown to satisfy the desired localization properties, and hence so does these vectors.} Executing this requires an understanding the null space of the infinite dimensional Toeplitz operator, as well as various bounds on the small singular values of $T_N({\bm a}_z)$. 

We now turn to describe the ideas behind the proof of Theorem \ref{thm:eigenvec}(iii) (proof of the last part follows upon combining the previous parts together with a local estimate on the number of eigenvalues in disks of radius order $\log N/N$ and a double counting argument). Using the resolvent expansion we obtain from \eqref{eq:algeb-id} and \eqref{eq:simple1} that 
\begin{multline}\label{eq:coeff-eM}
E_-^\delta(z) (\lambda -z)v_\lambda= - E_{-+}^\delta(z) R_+(z) v_\lambda 
 =
 - E_{-+}(z)R_+(z) v_\lambda\\
  + \delta E_-(z) Q E_+(z_\lambda) R_+(z) v_\lambda 
 -
 \delta^2 E_-(z) ({\rm Id}_N+\delta Q E(z))^{-1} Q E(z) Q E_+(z) R_+(z) v_\lambda, 
\end{multline}
By a similar argument as in \eqref{eq:norm-bds}, the third term in \eqref{eq:coeff-eM}
turns out to be of order
$o(\delta)$ and $\|E_-^\delta(z)\| \le 2 \sqrt{M}$, uniformly for all $z \in \cN_\theta$ with probability $1-o(1)$. It was further argued in \cite{BVZ} that there exists some $M>M_0\geq 0$ so that $t_j(z)$ decay exponentially in $N$ for $j\in [M_0]$ for all $z \in \cN_\theta$. Therefore, upon choosing $\theta$ sufficiently large (depending only on $\delta$), recalling \eqref{gp5} and \eqref{gp8},
with $a_j= (e_j(z_\lambda)^* v)$, from \eqref{eq:coeff-eM} it follows
\beq\label{eq:exp-removed}  
\left\|\sum_{i=1}^{M_0}
\left[\sum_{j=1}^M a_j \cdot (f_i(z_\lambda)^* Q e_j(z_\lambda))\right] \delta_i\right\|=o(1),
\eeq
on an event with probability $1-o(1)$. 

The next key ingredient is the following observation: for any $\eta >0$ there exist $0 < c_\eta, C_\eta < \infty$ such that $c_\eta \le s_{\min}(\wt A(z) [[M_0]; [M_0]]) \le \|\wt A\| \le C_\eta$ simultaneously for $z \in \cN_\theta \cap \D(\wh z_0, C_0 \log N/N)$, on an event with probability $1-\eta$, where $\wt A(z)$ is the $M \times M_0$ matrix with its $(i,j)$-th entry $f_i(z)^* Q e_j(z)$. In \cite{BVZ} this was achieved using a chaining argument. 

The upshot of this last ingredient is that it together with Theorem \ref{thm:eigenvec}(ii) and \eqref{eq:exp-removed} imply $v_\lambda$ has a non-negligible (bounded away from zero) $\ell^2$ mass in the span of $\{e_i(z_\lambda)\}_{i=M_0+1}^M$. Now, upon showing that $e_i(z)$ spreads \corAB{out at scale $\log N/N$, for $i=M_0+1,\ldots, M$, uniformly for all $z \in \cN_\theta$, completes the argument.}

\begin{proof}[Proof of Theorem \ref{thm:eigenvec-outlier}]
We follow the same strategy as described above. 
For $\theta >0$ (determined below) there exists a net $\cN_\theta$ of $\D(0,1-\vep)$ of mesh size $N^{-\theta}$ such that $|\cN_\theta| \lesssim N^{2\theta}$. For the rest of the proof we take $M=1$ in \eqref{gp8} to define $E(z)$, $E_\pm(z)$, $E_{-+}(z)$ etc. We claim:
\beq\label{eq:sing-val-bd}
s_{\min}(J_N -z {\rm Id}_N) \leq (1-\vep)^N \quad \text{ and } \quad s_2(J_N - z {\rm Id}_N) \geq \vep, \quad \corAB{\text{ for all } z \in \cN_\theta.}
\eeq
Note that this implies $\|E(z)\| \le \vep^{-1}$, for all $z \in \cN_\theta$. Assuming the bound \eqref{eq:sing-val-bd} for the moment, as $E_N$ satisfies Assumption \ref{ass:spectral-norm}, it is straightforward to note that for any $\vep_0 >0$, with probability at least $1- N^{-\vep_0/6}$, 
\beq\label{eq:hs-norm-bd}
\max_{k \in \N} \max_{z \in \cN_\theta} N^{-k/2(1+\vep_0)} \max\left\{\|(E(z)Q)^k E_+(z)\|,  \|(E(z)Q)^k E(z)\|\right\} \le 1, \quad \text{ for all large } N.
\eeq
Using the resolvent expansion and recalling \eqref{eq-Edelta}, for any $K \in \N$, we note that 
\beq\label{eq:E-resolvent-expansion}
E^\delta(z) = E(z) + \sum_{k=1}^K (-1)^k \delta^k E(z) (Q E(z))^k + (-1)^{K+1} \delta^{K+1} E(z) (Q E(z))^{K+1} ({\rm Id}_N +\delta Q E(z))^{-1}.
\eeq
On the other hand, applying \cite[Lemma 10.10]{BVZ}, together with a union bound, for any $\alpha_0 >0$ there exists some $\beta_0=\beta_0(\alpha_0, \gamma) \in (0, \infty)$ such that 
\[
\P\left(\max_{z \in \cN_\theta} \|({\rm Id}_N +\delta Q E(z))^{-1}\| \ge N^{\beta_0}\right) \le N^{3\theta - \alpha_0}, \qquad \text{ for all large } N. 
\]
Hence, as $\gamma >1/2$, upon choosing $K$ large enough in \eqref{eq:E-resolvent-expansion} depending on $\gamma$ and $\beta_0$, the above bound together with \eqref{eq:hs-norm-bd}-\eqref{eq:E-resolvent-expansion} yield that 
\[
\P\left(\max_{z \in \cN_\theta}\|E^\delta(z)\| \ge 2 \vep^{-1}\right) \le N^{3\theta -\alpha_0} + N^{-\wt \vep} \le 2 N^{-\wt \vep},
\]
for $\wt \vep >0$, where in the last step we set $\alpha_0=3\theta+\wt \vep$. 
By a similar argument we also find 
\[
\P\left(\max_{z \in \cN_\theta}\|E(z) ({\rm Id}_N+\delta QE(z))^{-1} \delta QE_+(z)\| \ge N^{-\wt \vep}\right) \le 2 N^{-\wt \vep}.
\]
The last two probability bounds together with \eqref{eq:ef-to-pm-1}, upon setting $\theta= \wt \vep$, as $\{e_i(z)\}_{i=1}^N$ is an orthonormal basis, yield that on an event with probability at least $1-o(1)$ for any eigenvalue $\lambda \in \D(0, 1-\vep)$ of $J_N + N^{-\gamma} E_N$, there exists some $\theta_\lambda \in (0, 2\pi]$ such that 
\[
\|v_\lambda - e^{{\rm i} \theta_\lambda} e_1(z_\lambda)\|_\infty \le \|v_\lambda - e^{{\rm i} \theta_\lambda} e_1(z_\lambda)\| = O(N^{-\wt \vep}),
\]
where $z_\lambda \in \cN_\theta$ satisfying $|\lambda - z_\lambda| \le N^{-\theta}$. To complete the proof it suffices to show that 
\beq\label{eq:e1-gz}
\|e_1(z) - e^{{\rm i} \wh \theta_z} \gz_z\|_\infty \le \|e_1(z) - e^{{\rm i} \wh \theta_z} \gz_z\| \le 2 \vep^{-1/2} (1-\vep)^{N/2} \qquad \text{ for all } z \in \cN_\theta,
\eeq
where $\wh \gz_z:=(1,z,z^2, \ldots, z^{N-1})^t$, $ \gz_z:=\wh \gz_z/\|\wh \gz_z\|$, and some phase $\wh \theta_z \in (0, 2\pi]$. Turning to prove \eqref{eq:e1-gz}, recalling \eqref{gp2}, as $\|\gz_z\|=1$, and $\{e_i(z)\}_{i \in [N]}$ forms an orthonormal basis we deduce that 
\beq\label{eq:final-eq}
J_N - z {\rm Id}_N= \sum_{i=1}^N t_i(z) f_i(z) e_i(z)^*, \quad \gz_z = \sum_{i=1}^N \beta_i(z) e_i(z), \quad \text{ and } \quad \sum_{i=1}^N |\beta_i(z)|^2 =1,
\eeq
for some coefficients $\{\beta_i(z)\}_{i \in [N]}$. Observe that $\|(J_N - z {\rm Id}_N) \gz_z\| \leq  |z|^N/\|\wh \gz_z\| \leq |z|^N$ (this yields the desired bound on $s_{\min}$ of \eqref{eq:sing-val-bd}). Hence, from \eqref{eq:final-eq}, as $\{f_i(z)\}_{i=1}^N$ are orthonormal, using the bound $t_2(z) \ge \vep$ (see \eqref{eq:sing-val-bd}) we further derive that $\sum_{i=2}^N |\beta_i(z)|^2 \le \vep^{-1}(1-\vep)^N$. This immediately yields \eqref{eq:e1-gz}.

It remains to prove the right most inequality in \eqref{eq:sing-val-bd}. To this end, denote $C_N$ to be the circulant matrix whose first row is same as that of $J_N$. Set $\cR_n \subset {\rm Mat}_N(\C)$ be the set of matrices of rank at most $n$, for $n \in \N$. By \cite[Proposition 9.5]{BG05}, and as $C_N$ is a rank one perturbation of $J_N$, we find 
\begin{multline*}
s_2(J_N - z {\rm Id}_N) =\min\{\|J_N -z {\rm Id}_N - A\|: A \in \cR_{N-2}\} \\
\ge \min\{\|C_N -z {\rm Id}_N - A\|: A \in \cR_{N-1}\} = s_{\min}(C_N - z {\rm Id}_N) \ge \vep,
\end{multline*}
completing the proof of \eqref{eq:sing-val-bd} and of the theorem. 
\end{proof}

%\begin{remark}\label{rmk:complete-loc}
%\red{The reader may} note that the requirement of $\gamma >1$ is used only in \eqref{eq:hs-norm-bd}. 
%\[
%\max\left\{N^{-k/2}\|(E(z)Q)^k E_+(z)\|_{{\rm HS}},  N^{-(k+1)/2}\|(E(z)Q)^k E(z)\|_{{\rm HS}}\right\} \le N^{\vep}, \qquad \text{ for any } k \in \N,
%\]
%Similar improvements are considered in \cite{BVZ2}. 
%\end{remark}

\end{document}